\documentclass[openany]{book}
\usepackage{amsmath,amssymb,amsthm,amscd}
\usepackage{epsfig, color}

\renewcommand{\thefootnote}

\newtheorem{lemma}{Lemma}[section]
\newtheorem{proposition}{Proposition}[section]
\newtheorem{theorem}{Theorem}[section]
\newtheorem{corollary}{Corollary}[section]

\theoremstyle{definition}
\newtheorem{example}[theorem]{Example}

\theoremstyle{remark}
\newtheorem{remark}[theorem]{Remark}

\numberwithin{equation}{section}

\newcommand*{\be}{\begin{equation}}
\newcommand*{\ee}{\end{equation}}
\newcommand{\nn}{\nonumber}
\providecommand{\abs}[1]{\lvert#1\rvert}

\newcommand*{\fl}[1]{\lfloor{#1}\rfloor}

\def\ddd{\displaystyle}
\def\ind{\mathbf{1}}

\DeclareMathOperator{\Var}{Var}
\DeclareMathOperator{\Cov}{Cov}

\newcommand{\limd}{\overset{\mathcal{D}}{\longrightarrow}} 

\def\abar{\bar{a}}
\def\qbar{\bar{q}}

\def\e{\varepsilon}

   \def\cN{\mathcal{N}}

 \def\bE{\mathbb{E}}
\def\bN{\mathbb{N}}  \def\N{\mathbb{N}}  
\def\bP{\mathbb{P}}  
\def\bR{\mathbb{R}}  \def\R{\mathbb{R}}  \def\bbR{\mathbb{R}}
\def\bZ{\mathbb{Z}}
\def\Zb{\mathbb Z}  \def\Z{\mathbb Z} 
\def\Rb{\mathbb R}
\def\Nb{\mathbb N}
\def\E{\bE}  \def\P{\bP}

\def\kS{{\mathfrak S}}

  \def\mE{\mathbf{E}}  \def\mP{\mathbf{P}}

   \def\thvec{\boldsymbol{\theta}}
\def\Xbar{\overline{X}}
\def\Ybar{\overline{Y}}  \def\Hbar{\bar{H}}

\def\Xtil{\tilde{X}}

\newcommand{\indd}[1]{ \mathbf{1}{ \{ #1 \} } } 

      \def\si{\sigma}  \def\la{\lambda} 
\def\om{\omega}

\def\sqn{\sqrt{n}}
\def\w{\omega}
\def\om{\omega}
\def\s{\sigma}

\def\Ev{\mathbf E}
\def\Pv{\mathbf P}  
\def\Vv{{\text{\bf Var}}} 

\newcommand{\vp}{\mathrm{v}_P}     
\def\evar{\s_0}  
 
\def\emean{\mu_0}   \def\dc{\emean}  
\def\wt{\widetilde}  \def\wh{\widehat} \def\wb{\overline}

\def\pcount{I}

\def\flux{H}  

\def\Qlb{m_Q}  
\def\lbla{a}  
\def\lblb{b} 
 
\def\bias{\theta}

\def\jr{g}  
\def\rr{\jr}  

\begin{document}
 
\renewcommand{\thepage}

\vskip .2in




\vskip 1in

\setcounter{page}{1}

\noindent {\huge\bf {Current fluctuations for}}


\vskip .13in \noindent {\huge\bf {stochastic particle systems with}}


\vskip .1in \noindent {\huge\bf  {drift in one spatial dimension}}

\thispagestyle{empty}

\vskip .4in
\noindent {\Large\bf {Timo Sepp\"al\"ainen }}

\vskip .8in

\noindent {\bf Abstract.\,} {This review article discusses limit distributions and 
variance bounds for particle current  in several dynamical stochastic systems of particles
on the one-dimensional integer lattice:   independent particles,
independent particles in a random environment, the random average
process,  the asymmetric
simple exclusion process, and a class of totally asymmetric zero range processes.
The  first three models possess linear macroscopic flux functions
and  lie in the Edwards-Wilkinson universality class with   scaling exponent $1/4$ for current fluctuations.  For these we prove 
Gaussian limits for the current process.  The latter two systems
belong to the Kardar-Parisi-Zhang class.    For these
we prove the scaling exponent $1/3$ in the form of upper and lower 
variance bounds.}

\newpage
\tableofcontents
\newpage
\renewcommand{\thefootnote}{\arabic{footnote}}
\setcounter{footnote}{0}
\renewcommand{\thepage}{\arabic{page}}
\renewcommand{\thesection}{\arabic{section}}
\setcounter{section}{0}


\pagestyle{myheadings} \markboth{\sf Timo Sepp\"al\"ainen}{\sf Current fluctuations}




\section{Introduction}
\setcounter{equation}{0}
This review article investigates  the process of particle current in several conservative 
stochastic systems 
of particles that live on the one-dimensional 
integer lattice. In conservative systems particles are neither created nor
destroyed.  On the macrosopic, deterministic scale the particle density  $\rho(t,x)$
of  such systems is governed
by a partial differential equation of   scalar conservation law type: $\rho_t+\flux(\rho)_x=0$.  
   The results covered in this article  concern the fluctuation behavior of the current,
and include some precise limit results and some coarser order-of-magnitude
bounds.   Some background on advanced probability theory is assumed.  
For readers with little probability background we can also suggest
review article \cite{sepp-review-08} where an attempt was made to explain
some of this same material for a reader with background in analysis rather
than probability. 

  The particle processes studied fall into two categories.  We can define
  these two categories    by the slope of the {\sl flux function} $\flux$: 
  (i)  processes with a linear
flux and (ii)  those with a strictly concave flux.  By definition, the flux $\flux(\rho)$ 
is the mean rate at which particle mass moves past a fixed point in space when
the system is stationary in both space and time  and  has overall density $\rho$ of particles.  
The processes we study are {\sl asymmetric} or {\sl driven}  in the sense that the particles have  a
drift, that is, a  preferred
average direction.  This assumption is not necessary for the results  
for systems with
linear flux, but for those with nonlinear flux it is crucial.  

In the statistical physics  terminology   of surface growth, we can also label these
two classes as (i)  the EW (Edwards-Wilkinson) and (ii)  the KPZ (Kardar-Parisi-Zhang)
universality classes \cite{bara-stan}.   Surface growth may seem at first  a separate topic from particle 
systems.  But in one dimension conservative particle systems can be equivalently
formulated as interface models.  The connection goes by way of regarding the
particle occupation numbers $\eta_i$  as increments, or discrete gradients, of the interface height
function:  $h_i=\eta_i-\eta_{i-1}$.   Then any movement of particles in the conservative particle system
 can be equivalently described as deposition
or removal of particles from the growing surface.   In particular, the current process
then maps directly into the height function.  
 
 In the EW class limiting current/height fluctuations are described by the linear stochastic heat
 equation   $Z_t =\nu  Z_{xx} +    \dot W$  where subscripts are partial derivatives 
and   $\dot W$ represents space-time  white noise.    
 In the microscopic model the order of magnitude of current fluctuations  
is $n^{1/4}$ in terms of a scaling parameter $n$ ($n\nearrow\infty$) that gives both the space 
and time scale.  The results  we give describe Gaussian limit distributions of the current
process.  The examples of systems with linear flux we cover are  independent random walks,
independent random walks in a random environment (RWRE), and  the random average
process. 

 In the RWRE case we in particular wish to understand how the environment
affects the fluctuations.   We find a two-level fluctuation picture.  On the diffusive
$n^{1/2}$ scale the quenched mean of the current converges to a Brownian motion. 
Around the quenched mean,  
on the   $n^{1/4}$ scale,  the  current   converges 
  to the Gaussian processes   that arose in the case of 
independent classical random walks, but with an additional Brownian random spatial
shift.  

 The systems with concave flux that we discuss are  the asymmetric
simple exclusion process (ASEP)  and a class of zero range processes.  In this case the 
assumption of asymmetry is necessary.  We consider only stationary systems 
(or small perturbations thereof), and instead of distributional limits we give only
bounds on the variance of the current and  on the moments of a second class  particle. 
   The current fluctuations are now of order
$t^{1/3}$.  This order of magnitude  goes together with superdiffusivity of the second class particle 
whose fluctuations are of order $t^{2/3}$.    Here $t$ is the time parameter of the
process. A separate scaling parameter is not needed since we have no process level
result.  These
systems are in the KPZ class. 

We cannot cover complete proofs for all   results to keep  this article at a reasonable length. 
   The most important result, namely the moment bounds for 
the second class particle in ASEP, is proved in full detail,  assuming some basic
facts about ASEP.

 In some sense interface height in the KPZ class
should be described by the KPZ equation
\be   h_t= \nu h_{xx} - \lambda (h_x)^2 + \dot W.  \label{kpz}\ee
   However, giving mathematical meaning to 
this equation has been done only indirectly.  A formal Hopf-Cole transformation 
$Z=\exp(-\lambda\nu^{-1}h)$ converts \eqref{kpz}  into a stochastic heat equation
with multiplicative noise:   $Z_t=\nu  Z_{xx} +  \lambda\nu^{-1} Z  \dot W$.  
The solution of this latter equation is well-defined, and can furthermore be obtained
as a   limit of an appropriately scaled height function of a weakly asymmetric simple
exclusion process \cite{bert-giac-97}.  Via this link the scaling exponents 
have recently been verified for the Hopf-Cole solution of the KPZ equation \cite{bala-quas-sepp}. 

The most glaring omission of this article  is that we do not treat the Tracy-Widom fluctuations
of KPZ systems such as  TASEP (totally asymmetric simple exclusion process), 
ASEP or the PNG (polynuclear growth) model.  
 In the ASEP section  we only 
state briefly the Ferrari-Spohn theorem \cite{ferr-spoh-06}  on the limit distribution of the current across
a characteristic in stationary TASEP. 
This area  is advancing rapidly and  requires a review article  of its own if any degree of
detail is to be covered. 
Tracy and Widom have recently written a series of papers where the fluctuation limits
originally proved for TASEP by Johansson  \cite{joha} 
have been extended to ASEP \cite{trac-wido-08jsp, trac-wido-08cmp,  trac-wido-09cmp,  
trac-wido-09jmp, trac-wido-10}.   Another line of recent 
work extends the TASEP limits to more general
space-time points and initial distributions
\cite{baik-ferr-pech-10, corw-ferr-pech-10}.  

\medskip
 
{\sl A notational convention.}  $\bZ_+=\{0,1,2,3,\dotsc\}$ and $\bN=\{1,2,3,\dotsc\}$.

{\sl Acknowledgements.}  
This article is based on lecture notes for a minicourse  
given at the  13th Brazilian School of Probability,  August 2-8, 2009,
at Maresias,  
S\~ao Paulo, and  again at the
University of Helsinki, Finland,  on   August 18--20, 2009.   The author
thanks the organizers and the audiences of these two occasions.
The material comes from  collaborations  with 
M\'arton Bal\'azs,  Mathew Joseph, J\'ulia  Komj\'athy,  Rohini Kumar, Jon Peterson and 
Firas Rassoul-Agha.  
The author   is grateful for financial support  from the 
National Science Foundation   through grant DMS-0701091 and from the Wisconsin Alumni Research
Foundation.

\section{Independent particles executing classical \\random walks} \label{iidch}
\setcounter{equation}{0}
\subsection{Model and results}

Fix a probability distribution $p(x)$ on $\bZ$.  Let us assume for
simplicity that $p(x)$ has finite range, that is, 
\be   \text{ the set $\{x\in\bZ: p(x)>0\}$ is finite.}   \label{pass1}\ee
For each site $x\in\bZ$ and index $k\in\bN$
let  $X^{x,k}_\centerdot $ be a discrete-time random walk 
with   initial point $X^{x,k}_0=x$ and transition probability 
\[  P(X^{x,k}_{s+1}=y\,\vert\,X^{x,k}_s=z)=p(y-z) \]
for times $s\in\bZ_+$ and space points  $z,y\in\bZ$.  
The   walks
$\{X^{x,k}_\centerdot:  x\in\bZ, k\ge 1\}$ are independent of each other. 
Let 
\[ v=\sum_{x\in\bZ}  xp(x) 
\quad\text{and}\quad   \sigma_1^2=\sum_{x\in\bZ}  (x-v)^2p(x)\] 
be the mean speed and variance of the walks.

At time $0$ we start a random number $\eta_x(0)$ of particles at site $x$. 
The assumption is that the initial occupation variables 
$\eta(0)=\{\eta_x(0):x\in\bZ\}$ are i.i.d.\ (independent and identically
distributed)  with finite mean and variance
\be
\emean=E[\eta_x(0)]  \quad\text{and}\quad  
\evar^2=\Var[\eta_x(0)].   \label{eta1}\ee
Furthermore, the variables $\{\eta_x(0)\}$ and 
  the walks 
$\{X^{x,k}_\centerdot \}$ are independent of each other. 
If the locations of individual particles are not of interest but only
the overall particle distribution, the
particle configuration at time $s\in\bN$  is
 described by the occupation variables 
$\eta(s)=\{\eta_x(s):x\in\bZ\}$ defined as
\[\eta_x(s)=\sum_{y\in\mathbb Z}\sum_{k=1}^{\eta_0(y)}{\bf 1}\{X^{y,k}_s=x\}. \]  

A basic fact is that if the initial occupation variables
are i.i.d.\ Poisson with common mean $\mu$, then so are 
the occupation $\{\eta_x(s):x\in\bZ\}$  at any fixed time $s\in\bZ_+$.  
This  invariance   follows 
   easily from basic properties of Poisson distributions:  the counts of
particles that move from $x$ to $y$ are independent Poisson variables  
across all pairs $(x,y)$.  

 Let $E^\mu$ denote expectation under the stationary process with mean 
 $\mu$ occupations.  
The flux function is the expected rate of flow in the stationary process: 
\be\begin{aligned}
\flux(\mu)&=E^\mu [\text{net number of particles that jump across edge $(0,1)$}\\ 
     &\qquad\qquad \qquad  \text{left to
right in one time step}] \\
&=E^\mu \biggl[\;\sum_{x\le 0}\sum_{k=1}^{\eta_0(x)}{\bf 1}\{X^{x,k}_1\ge 1\} 
\;-\;  \sum_{x> 0}\sum_{k=1}^{\eta_0(x)}{\bf 1}\{X^{x,k}_1\le 0\} \biggr]\\
&=v\mu.
 \end{aligned}\label{fluxex}\ee
 
The linear flux puts this system in the class where we expect $n^{1/4}$ 
magnitude current fluctuations.   Next we define the current process 
and the Gaussian process that describes the limit.  
 
Let $n\in\bN$ denote a scaling parameter that eventually goes to $\infty$.
For (macroscopic) times $t\in\bR_+$ and a spatial variable $r\in\bR$, let 
   \be Y_n(t,r)=Y_{n,1}(t,r)-Y_{n,2}(t,r) \label{Y}\ee with 
\be Y_{n,1}(t,r)=\sum_{x\le 0} \sum_{k=1}^{\eta_x(0)} 
\mathbf{1}\lbrace X^{x,k}_{\fl{nt}}
> \fl{ntv}+r\sqrt n\,\rbrace\label{Y1}\ee
and 
\be Y_{n,2}(t,r)=\sum_{x> 0} \sum_{k=1}^{\eta_x(0)} 
\mathbf{1}\lbrace X^{x,k}_{\fl{nt}}
\le  \fl{ntv}+r\sqrt n\,\rbrace.  \label{Y2}\ee
Variable  $Y_n(t,r)$ represents  the net 
left-to-right  current of particles  
seen  by an observer who starts at the origin  and 
reaches point  $\fl{ntv}+\fl{r\sqrt n}$  in  time $\fl{nt}$.  
Its mean is 
\[  EY_n(t,r)= \,\emean E\bigl( X_{\fl{nt}} - \fl{ntv}- \fl{r\sqrt n\,}\bigr) = - \emean r\sqrt n +O(1).  \]
Define the centered and appropriately scaled process by
\[  \Ybar_n(t,r) = n^{-1/4}\bigl( Y_n(t,r) - EY_n(t,r)\bigr).  \]
 
 The goal is  to prove a limit for the joint distributions of these random variables.    
We will not tackle    process-level convergence.  But
 let us point out that there is a natural path space $D_2$ of functions
 of two parameters $(t,r)$ that contains the paths of the processes $Y_n$. 
Elements of $D_2$ are    continuous from above with limits
 from below in a suitable way, and there is a metric that generalizes 
 the standard Skorohod topology of the usual $D$-space of cadlag paths. 
 (See \cite{bick-wich, kuma-08}.) 

 Let 
 \be \varphi_{\nu^2}(x)=\frac1{\sqrt{2\pi\nu^2}}\exp\Bigl(-\frac{x^2}{2\nu^2}\Bigr)
\quad\text{and}\quad  
 \Phi_{\nu^2}(x)=\int_{-\infty}^{x}\varphi_{\nu^2}(y)dy  \label{gauss}\ee
denote the centered Gaussian density with variance $\nu^2$ 
and 
its distribution function. 
 Let $W$ be a two-parameter 
Brownian motion on $\bR_+ \times \bR$ and $B$  a 
two-sided one-parameter Browian motion on $\bR$.
 $W$ and $B$ are independent.  Define the process $Z$ by 
\be\begin{aligned}
 Z(t,r)&= \sqrt{\emean} \iint_{[0,t]\times \bR} 
\varphi_{\sigma_1^2(t-s)}(r-x)\,dW(s,x) \\
&\qquad  +\evar\int_{\bR}
 \varphi_{\sigma_1^2t}(r-x)B(x)\,dx.
\end{aligned}  \label{Zdef}\ee
 $\{Z(t,r): t\in \bR_{+},r\in \bR\}$ is a mean zero  
Gaussian process. Its  covariance can be expressed as follows:
with 
\be \Psi_{\nu^2}(x)=\nu^2\varphi_{\nu^2}(x)
-x\big(1-\Phi_{\nu^2}(x)\big)\label{Psi}\ee
define 
\be \Gamma_1\bigl((s,q),(t,r)\bigr)=\Psi_{\sigma_1^2 (t+s)}(r-q ) - \Psi_{\sigma_1^2 \abs{t-s}}(r-q ) 
 \label{Ga1}\ee
and
\be \Gamma_2\bigl((s,q),(t,r)\bigr)= \Psi_{\sigma_1^2 s}(-q)+\Psi_{\sigma_1^2 t}(r)-\Psi_{\sigma_1^2 (t+s)}(r-q).  \label{Ga2}\ee
Then 
\be \mE[Z(s,q)Z(t,r)]=\emean\Gamma_1\big((s,q),(t,r)\big)+
\evar^2\Gamma_2\big((s,q),(t,r)\big). \label{Zcov}\ee
 
\begin{remark}
Some comments on Gaussian processes defined above.  
Comparison of \eqref{Zdef} and \eqref{Zcov} shows that 
$\Gamma_1$ is the covariance of the dynamical fluctuations represented
by the space-time white noise $dW$-integral,  and $\Gamma_2$ is the
covariance of the contribution of the initial fluctuations  represented by the
Brownian motion $\evar B(\cdot)$. 

 $\Gamma_1$ has this alternative formula:
 \be  \Gamma_1\bigl((s,q),(t,r)\bigr) = 
 \tfrac12  \int_{\s_1^2\abs{t-s}}^{\s_1^2(t+s)} 
\frac1{\sqrt{2\pi v}} \exp\Bigl\{  \frac1{2v}(r-q)^2\Bigr\} \,dv.  \label{Ga1b}\ee
To verify that  the process 
\[  \zeta(t,r)=     \iint_{[0,t]\times \bR} 
\varphi_{\sigma_1^2(t-s)}(r-x)\,dW(s,x) \]
 has covariance 
\be  \mE \zeta(s,q)\zeta(t,r) = \Gamma_1\bigl((s,q),(t,r)\bigr)  \label{Ucov1}\ee
  is  straightforward from the   general property that 
for $f,g\in L^2(\bR^d)$ the white-noise integrals $\int f\,dW$ and 
$\int g\,dW$ on $\bR^d$ are by definition mean zero Gaussian random variables that satisfy 
\[    \mE \Bigl[ \Bigl( \int f\,dW \Bigr)   \Bigl( \int g\,dW\Bigr) \Bigr]
= \int_{\bR^d}  f(x)g(x)\,dx.  \]
\label{stheateqex1} 
 
To say that  $B(x)$ is a (standard)  two-sided Brownian motion  means that we take two
independent standard Brownian motions $B_1$ and $B_2$ and set
\[  B(x)=\begin{cases}  B_1(x),  &x\ge 0\\  B_2(-x), &x<0. \end{cases}\]
To show that the process
\[ \xi(t,r)= \int_{\bR}
 \varphi_{\sigma_1^2t}(r-x)B(x)\,dx \]
is a mean-zero Gaussian process with    covariance 
\[   \mE \xi(s,q)\xi(t,r)= \Gamma_2\bigl( (s,q),(t,r)\bigr) \]  
  this formula is useful:  if $f,g$ are absolutely continuous functions on $\bR_+$
such that  $xf(x)\to 0$ and $xg(x)\to 0$ as $x\to\infty$, then 
\[  \iint_{\bR_+^2} f'(x)g'(y)(x\wedge y)\,dx\,dy = \int_0^\infty f(x)g(x)\,dx. \]

It is also the case that the process  $Z(t,r)$ is the unique weak solution of the following initial value
problem for  a  linear stochastic heat equation on $\bR_+\times\bR$:
\be 
Z_t = \tfrac{\s_1^2}2 Z_{rr} +  \sqrt{\emean} \, \dot W \,, \quad
Z(0,r)=\s_0 B(r). 
\notag\label{stheateq2}\ee
(Above, subscript means partial derivative.)  
A weak solution of this equation is defined by the requirement 
  \begin{align*}   &\int_{\R}  \phi(r) Z(t,r)\,dr -  \s_0 \int_{\R}  \phi(r) B(r)\,dr \\
 & =\; \tfrac{\sigma_1^2}2\,  \iint_{[0,t]\times\Rb} \phi''(r) Z(s,r) \,dr\,ds 
\; + \; \sqrt\emean\, \iint_{[0,t]\times\Rb} \phi(r)  d{W}(s,r)
\end{align*} 
for all $\phi\in C^\infty_c(\bR)$  (compactly supported, infinitely differentiable).  
See the lecture notes of Walsh \cite{wals-spde}. 
\end{remark}

We can now state the main result. 

\begin{theorem}
\label{main_thm}
 Assume
the initial occupation variables are i.i.d.\ with finite mean and variance
as in \eqref{eta1}. 
 Then  as $n\to\infty$,   the finite-dimensional distributions 
 of the  process $\{\Ybar_n(t,r): (t,r)\in\Rb_+\times\Rb \}$ converge weakly 
 to the finite-dimensional distributions 
of the   mean zero Gaussian
process $\{Z(t,r): (t,r)\in\Rb_+\times\Rb \}$.   
 \end{theorem}
 
 The  statement means that for any  space-time points 
$(t_1,r_1),\dotsc,(t_k,r_k)$, this weak convergence of
  $\bR^k$-valued random vectors holds:   
\[   (\Ybar_n(t_1,r_1),\dotsc, \Ybar_n(t_k,r_k)) 
\limd  (Z(t_1,r_1),\dotsc, Z(t_k,r_k)).  \]
 Under additional moment assumptions process level convergence in the space $D_2$
 can be proved (see \cite{kuma-08}).  
We state a corollary 
for  the special case of  the stationary occupation process $\eta(t)$.    Its
proof comes from simplifying  expression \eqref{Zcov} for the covariance. 

 \begin{corollary}
Suppose the process is stationary 
so that  $\{\eta_x(t):x\in\mathbb Z\}$ are i.i.d.\ Poisson with mean $\emean$ for each fixed $t$.
Then  at $r=0$ the limit  process 
$Z$ has  covariance
 \[\mE Z(s,0)Z(t,0)=\frac{\emean\s_1}{\sqrt{2\pi}}\bigl(\sqrt{s}+\sqrt{t}-\sqrt{|t-s|}\bigr).\]
In other words, process $Z(\cdot,0)$ is fractional Brownian motion with Hurst parameter $1/4$.
\label{fBMcor}\end{corollary}

\subsection{Sketch  of proof} \label{rwpfsec}

We turn to discuss   the proof of Theorem \ref{main_thm}. 
Independent walks allow us to compute everything in a straightforward manner.  
  Fix some  $N\in\bN$,   time points  
$0<t_1<t_2<\cdots<t_N \in \bR_{+}$,  
  space points  $r_1,r_2,\dotsc,r_N \in \bR$ and an 
  $N$-vector $\thvec=(\theta_1,\dotsc, \theta_N)\in\bR^N$.  Form the linear combinations 
 \[ \overline Y_n(\thvec)= \sum_{i=1}^N \theta_i \overline Y_n(t_i,r_i)
\quad\text{and}\quad  Z(\thvec)= \sum_{i=1}^N \theta_i  Z(t_i,r_i). \]
 
 The goal is now to prove 

\begin{proposition}
 \be   E \bigl[ \exp \bigl\{ i \overline Y_n(\thvec) \bigr\}\bigr]
 \to \mE \bigl[\exp\bigl\{i Z(\thvec)\bigr\} \bigr]. 
 \label{fddlim}\ee
\label{fddprop}\end{proposition}

  Since the random walks and 
initial occupation variables are independent,  we can  write $ \overline Y_n(\thvec)$  as a sum of
independent random variables and  take advantage of standard central limit
theorems  from the literature.   
  \be \overline Y_n(\thvec) =
 n^{-\frac{1}{4}}\sum_{i=1}^N\theta_i\big{\{}Y_{n}(t_i,r_i)-E Y_{n}(t_i,r_i)\big{\}}=
 W_n=\sum_{m=-\infty}^{\infty} u(m)  \label{sum1}\ee
with 
\begin{equation}
\label{umbar}
u(m)=\, \sum_{i=1}^N \theta_i \Bigl( U_m(t_i,r_i) \,\ind\{m \le 0\}  
- V_m(t_i,r_i) \, \ind\{m > 0\} \Bigr), 
\end{equation}
and 
\begin{align}
&\begin{split} U_m(t,r)&=n^{-\frac{1}{4}}\sum_{j=1}^{\eta_m(0)} \mathbf{1}\lbrace X^{m,j}_{nt}
> \fl{ntv}+r\sqn\rbrace \\
&\qquad\qquad \qquad\qquad-n^{-\frac{1}{4}} \emean P(X^m_{nt}> \fl{ntv}+r\sqn\,),\end{split} \label{Udef}\\
&\begin{split}V_m(t,r)&=n^{-\frac{1}{4}}\sum_{j=1}^{\eta_m(0)} \mathbf{1}\lbrace X^{m,j}_{nt}
\le  \fl{ntv}+r\sqn\rbrace \\
&\qquad\qquad \qquad\qquad  -n^{-\frac{1}{4}} \emean  P(X^m_{nt}\le  \fl{ntv}+r\sqn\,). 
\end{split}\nn\end{align}
The variables $\{u(m) \}_{m\in\bZ}$ are independent  
because initial occupation variables and walks are independent.   
  
Let $a(n)\nearrow\infty$ be a sequence that will be determined precisely in the proof.  
As the first step we observe 
 that the terms $\abs m> a(n)\sqn$ can be discarded from \eqref{sum1}.   
Define  
\be   W_n^*=\sum_{\abs m\le a(n)\sqn}  u(m) . \label{sum2}\ee
The lemma below is proved by calculating moments of the random walks.  
  
\begin{lemma}   $E\abs{W_n-W_n^*}^2\to 0$ as $n\to\infty$.  \label{Slemma}\end{lemma}

The limit $Z(\thvec)$ in our goal \eqref{fddlim}  has $\cN(0, \sigma(\thvec)^2)$
distribution with  variance 
\be 
\sigma(\thvec)^2=\sum_{1\le i,j\le N}  \theta_i\theta_j \Bigl[  
\emean\Gamma_1\big((t_i,r_i),(t_j,r_j)\big)+\evar^2\Gamma_2\big((t_i,r_i),(t_j,r_j)\big) \Bigr]. 
\label{sitheta}\ee 
The two $\Gamma$-terms, defined earlier in \eqref{Ga1} and \eqref{Ga2}, 
 have the following expressions in terms of a standard 1-dimensional
Brownian motion $B_t$:
 \be \begin{aligned}
&\Gamma_1\bigl((s,q),(t,r)\bigr)  
=  \int_{-\infty}^{\infty}\Bigl(  \mP[B_{\sigma_1^2s}\le q-x]\mP[ B_{\sigma_1^2t}> r-x] \\
&\qquad\qquad\qquad- \; \mP[B_{\sigma_1^2s}\le  q-x, B_{\sigma_1^2t}> r-x] \Bigr)\,dx 
\end{aligned}\label{Ga1a}\ee
and 
 \be \begin{aligned}
&\Gamma_2\bigl((s,q),(t,r)\bigr)  
=
 \int_{-\infty}^{0} \mP[B_{\sigma_1^2s}>  q-x]\mP[ B_{\sigma_1^2t}> r-x]\,dx \\
 &\qquad\qquad +\;\int_{0}^\infty \mP[B_{\sigma_1^2s}\le  q-x]\mP[ B_{\sigma_1^2t}\le r-x]\,dx. 
\end{aligned}\label{Ga2a}\ee

\begin{remark}
Turning formulas \eqref{Ga1a}--\eqref{Ga2a} into \eqref{Ga1}--\eqref{Ga2} involves calculus
and  these properties of Gaussians:  
 $(d/dx)\Psi_{\nu^2}(x)=-\Phi_{\nu^2}(-x)$,   $\Phi_{\nu^2}(x)=1-\Phi_{\nu^2}(-x)$ 
and  
\[  \int_{-\infty}^\infty \Phi_{\alpha^2}(x)\Phi_{\nu^2}(r-x)\,dx =
\int_{-\infty}^r \Phi_{\alpha^2+\nu^2}(x) \,dx. \] 
 \end{remark} 
 
By Lemma \ref{Slemma},  the desired limit \eqref{fddlim} follows from showing  
\be   E(e^{iW_n^*}) 
 \to   e^{-\sigma(\thvec)^2/2}.  \label{goal2}\ee 
 This  will be achieved by the Lindeberg-Feller theorem.
 
 \begin{theorem}[Lindeberg-Feller]  For each $n$, suppose 
 $\{ X_{n,j}:  1\le j\le J(n)\}$ are independent, mean-zero, square-integrable random variables
 and let $S_n=X_{n,1}+\dotsm+X_{n,J(n)}$.
 Assume that 
\[  
\lim_{n\to\infty} \sum_{j=1}^{J(n)}  E( X_{n,j}^2) = \sigma^2\]
and for each $\e>0$,   
\[ \lim_{n\to\infty} \sum_{j=1}^{J(n)}   E \big( X_{n,j}^2 
\mathbf{1}\lbrace\vert  X_{n,j} \vert \ge \e  \rbrace \big) = 0.  \]
Then as $n\to\infty$,   $S_n$ converges in distribution to a $\cN(0,\sigma^2)$-distributed Gaussian
random variable.  In terms of probabilities, the conclusion is that for all $s\in\bR$, 
\[
\lim_{n\to\infty}  P\{S_n\le s\} = \frac1{\sqrt{2\pi\sigma^2}}  \int_{-\infty}^s e^{-x^2/2\sigma^2}\,dx.\]
In terms of characteristic functions, the conclusion is that  for all $t\in\bR$, 
\[  \lim_{n\to\infty}  E(e^{itS_n})  = e^{-\sigma^2t^2/2}.  \]
 \end{theorem}

Now to prove \eqref{goal2}
 the task is to   verify the conditions of the Lindeberg-Feller theorem:  
\be 
\sum_{\abs m\le a(n)\sqn} E( u(m)^2) \to  \sigma(\thvec)^2 \label{LF-1}\ee
and 
\be  \sum_{\abs m\le a(n)\sqn}E \big(\, \vert u(m)\vert^2 
\mathbf{1}\lbrace\vert  u(m) \vert \ge \e  \rbrace \big) \to 0.    \label{LF-2}\ee
 
 We begin with the negligibility condition  \eqref{LF-2}.  This will determine $a(n)$. 
 
\begin{lemma}  
Under assumption \eqref{eta1}, 
\be  \lim_{n\to\infty}  \sum_{\abs m\le a(n)\sqn}E\bigl(\, \vert u(m)\vert^2 
\mathbf{1}\lbrace\vert  u(m) \vert \ge \e  \rbrace\bigr) =0.  \label{lf:tech0}\ee
\end{lemma}
\begin{proof}  
   Since 
\[\vert u(m)\vert \le C{n^{-{1}/{4}}}\bigl( \eta_m(0) + \emean\bigr)  \] 
so for a different $\e_1>0$ and by shift-invariance, 
\begin{align*}
&\sum_{\abs m\le a(n)\sqn}E\big[ \vert u(m)\vert^2 
\mathbf{1}\lbrace\vert  u(m) \vert \ge \e  \rbrace \big]
\le  C a(n) E\bigl[(\eta_0(0) +\emean )^2 \ind\{ \eta_0(0)\ge n^{1/4}\e_1\}\bigr].
\end{align*}
By the moment assumption \eqref{eta1} this last expression $\to 0$ for every $\e_1>0$ 
if $a(n)\nearrow\infty$ slowly enough, for example 
\[  a(n)= \Bigl( E\bigl[(\eta_0(0) +\emean )^2 \ind\{ \eta_0(0)\ge n^{1/8}\}\bigr] \Bigr)^{-1/2}  
\qedhere \]
\end{proof} 

We turn to checking \eqref{LF-1}.  
\be \begin{aligned} 
& \sum_{\abs m\le a(n)\sqn} E\big[ u(m)^2 \big]\\
&=
\sum_{1 \le i, j \le N}   \theta_i\theta_j   \sum_{\abs m\le a(n)\sqn}
\Bigl[  \ind_{\{m\le 0\,\}}E\bigl(U_m(t_i,r_i)U_m(t_j,r_j)\bigr)\\[3pt]
  &\qquad\qquad \qquad\qquad+\;   
 \ind_{\{m> 0\,\}}E\bigl(V_m(t_i,r_i)V_m(t_j,r_j)\bigr)\Bigr].
\end{aligned}\label{vpfindim1}\ee
To the expectations   we  apply this formula for the covariance of two random sums:  with 
$\{Z_i\}$ i.i.d.\ and independent of a random $K\in\bZ_+$,  
\be\begin{aligned} 
  &\Cov\Bigl(\;\sum_{i=1}^K f(Z_i)\,, \sum_{j=1}^K g(Z_j)\Bigr) \\
&\quad = 
EK\, \Cov(f(Z),g(Z))+ \Var(K)\,Ef(Z)\, Eg(Z).
\end{aligned}\label{rwcovar}\ee
For  the first expectation on the right in \eqref{vpfindim1}:
\begin{align*} 
&E\bigl(U_m(s,q)U_m(t,r)\bigr)\\
&=
n^{-1/2} \Cov\biggl(\, \sum_{j=1}^{\eta_m(0)} \mathbf{1}\lbrace X^{m,j}_{ns} > \fl{nsv}+q\sqn\,\rbrace\,,  \\
&\qquad  \qquad\qquad \qquad \qquad \qquad 
\sum_{j=1}^{\eta_m(0)} \mathbf{1}\lbrace X^{m,j}_{nt} > \fl{ntv}+r\sqn\,\rbrace 
\biggr) \\
&=n^{-1/2}\emean \Bigl[ P(X^m_{ns}> \fl{nsv}+q\sqn, \, X^m_{nt}>  \fl{ntv}+r\sqn\,)\\[4pt]
&\qquad  \qquad\qquad
 -\; P(X^m_{ns}> \fl{nsv}+q\sqn\,) P(X^m_{nt}> \fl{ntv}+r\sqn\,)\Bigr]\\[5pt]
&\quad 
+ n^{-1/2}\evar^2  P(X^m_{ns}> \fl{nsv}+q\sqn\,)P(X^m_{nt}> \fl{ntv}+r\sqn\,).
 \end{align*}
 Do the same for the $V$-terms. 
 After some rearranging of the  probabilities, we arrive at 
  \begin{align}
& \sum_{\abs m\le a(n)\sqn} E\big[ u(m)^2 \big]  \label{fd6}\\
 &= n^{-1/2}\sum_{1\le i,j\le N}\theta_i\theta_j \,\biggl[
\emean \sum_{\abs m\le a(n)\sqn}   \Bigl\{
P(X^{m}_{nt_i}\le \fl{nt_iv}+r_i\sqrt{n}\,)\\[4pt]
&\qquad\qquad \qquad\qquad\qquad \qquad\qquad \qquad 
\times P( X^{m}_{nt_j}> \fl{nt_jv}+r_j\sqrt{n} \,)
\nn\\[4pt]
& \qquad \qquad\qquad \qquad
 -\; P(X^{m}_{nt_i}\le \fl{nt_iv}+r_i\sqrt{n},\, X^{m}_{nt_j}>  \fl{nt_jv}+r_j\sqrt{n} \,)\Bigr\} 
\nn \\[4pt]
&+\evar^2\sum_{ -a(n)\sqn\le m\le 0}  
 P(X^{m}_{nt_i}> \fl{nt_iv}+r_i\sqrt{n}\,)P( X^{m}_{nt_j}> \fl{nt_jv}+r_j\sqrt{n} \,)\nn \\[4pt]
&+ \evar^2\sum_{0< m\le a(n)\sqn}  
 P(X^{m}_{nt_i}\le \fl{nt_iv}+r_i\sqrt{n}\,)
   P( X^{m}_{nt_j}\le \fl{nt_jv}+r_j\sqrt{n} \,)\biggr]. \nn \end{align}
The terms above have been arranged so that the sums match up with the integrals 
in \eqref{sitheta}--\eqref{Ga2a}.  Limit \eqref{LF-1}  now follows
  because each sum  converges to the corresponding integral.  
To illustrate with the last term,  the convergence needed is 
\be\begin{aligned}
&  n^{-1/2} \sum_{ 0 < m\le a(n)\sqn}   
  P(X^{m}_{ns}\le \fl{nsv}+q\sqrt{n}\,)
 P( X^{m}_{nt}\le \fl{ntv}+r\sqrt{n} \,)\\[5pt]
&=  n^{-1/2} \sum_{ 0 < m\le a(n)\sqn}   
  P\Bigl\{\frac{X_{ns}- \fl{nsv}}{\sqrt n}\le q-\frac{m}{\sqrt{n}}\,\Bigr\}\\
&\qquad\qquad\qquad\qquad\qquad    \times 
 P\Bigl\{\frac{X_{nt}- \fl{ntv}}{\sqrt n}\le r-\frac{m}{\sqrt{n}}\,\Bigr\}  \\[5pt] 
&\underset{n\to\infty}\longrightarrow \; 
  \int_{0}^\infty \mP[B_{\sigma_1^2s}\le  q-x]\mP[ B_{\sigma_1^2t}\le r-x]\,dx. 
\end{aligned}\label{fdgoal6a}\ee
This follows from the CLT, a Riemann sum type argument and some estimation.  We skip the
  details.
With this we consider  Theorem \ref{main_thm}
proved.

\subsection*{References}
The results for i.i.d.~walks appeared, with a slightly different definition of the current process,  
in \cite{sepp-rw} and  \cite{kuma-08}.  Earlier related results appeared in
\cite{durr-gold-lebo}.

\section{Independent particles in a random environment}
\setcounter{equation}{0}
In this chapter we generalize the results of Section \ref{iidch} to   particles
in a random environment, with the purpose of seeing how the environment influences
the outcome.    In a fixed environment, that is, conditional on the environment, the particles evolve independently. 
But under the joint distribution of the walks and the environment,   the particles are no
longer independent because their evolution gives information about the environment.  
 The environment is static,  which means that it is fixed in
time.  

\subsection{Model and results}

We formulate the standard  one-dimensional nearest-neighbor random walk in random
environment (RWRE)  model and then put many particles in a fixed environment. 
We describe the (known) law of large numbers and central limit theorem of the walk 
itself,  and then the (newer) results on current fluctuations for many particles.

The space of environments is $\Omega = [0,1]^{\Z}$. For an environment $\w = \{ \w_x \}_{x\in\Z} \in \Omega$   let $\{X^{m,i}_\centerdot \}_{m,i}$ be a family of Markov chains 
on $\Z$ with distribution $P_\w$ determined  by the following properties:
\begin{enumerate}
\item   $\{X^{m,i}_\centerdot\}_{m \in\Z, i\in\N}$ are independent under the measure $P_\w$.  
\item \smallskip $P_\w(X^{m,i}_0 = m ) = 1$, for all $m\in \Z$ and $i\in\N$. 
\item \smallskip Each walk obeys these transition probabilities: 
\[
 P_\w(X^{m,i}_{n+1} = x+1 | X^{m,i}_n = x ) = 1- P_\w(X^{m,i}_{n+1} = x-1 | X^{m,i}_n = x ) = \w_x.
\]
\end{enumerate}
A system of random walks in a random environment may then be constructed by first choosing an environment $\w$ according to a probability distribution $P$ on $\Omega$ and then constructing the system of random walks $\{X^{m,i}_\centerdot \}$ as described above. The distribution $P_\w$ of the random walks given the environment $\w$ is called the \emph{quenched distribution}. 
The \emph{averaged distribution} $\P$ (also called \emph{annealed}) is obtained by averaging the quenched law over all environments:  $\P(\cdot) = \int_{\Omega} P_\w(\cdot) P(d\w)$. 
Expectations with respect to the measures $P$, $P_\w$ and $\P$ are denoted by $E_P$, $E_\w$, and $\E$, respectively, and variances with respect to the measure $P_\w$ will be denoted by $\Var_\w$. 
We  make the following assumptions on the environment. 

\medskip

{\bf Assumption 1.} 
The distribution $P$ on environments is i.i.d.~and uniformly elliptic. That is, $\{\w_x\}_{x\in\Z}$ are  i.i.d.~under the measure $P$, and there exists a $\kappa > 0$ such that $P(\w_x \in [\kappa, 1-\kappa] ) = 1$. 
 Furthermore,  
$E_P (\rho_0^{2+\e_0}) < 1$ for some $\e_0>0$, where $\rho_x = {(1-\w_x)}/{\w_x}$. 

\medskip
 
These assumptions put the RWRE in the regime where it has  
  transience to $+\infty$ with  a strictly positive speed  and also 
  satisfies a CLT   with an environment-dependent  centering.  We summarize these results
  here.   Define  a shift map on environments   by
$(\theta^x\w)_y=\w_{x+y}$. 
 Let 
  $T_1 = \inf\{ n\geq 0: X_n = 1 \}$ be the first  hitting time of   site $1\in\Z$ 
by a RWRE started at the origin, and  define  
 \be Z_{nt}(\w) = \vp \sum_{i=0}^{\fl{nt\vp}-1} (E_{\theta^i \w} (T_1) - \E T_1 ). \label{Znt}\ee 
The asymptotic speed $\vp$ is defined in the first statement of the next theorem
where we summarize some known basic facts about RWRE.  
 
\begin{theorem}[\cite{solomon, pete-phd, zeit-stflour}]\label{QCLTthm}  
Under the  assumptions made above we have these conclusions. 
\begin{enumerate}
\item The RWRE satisfies a law of large numbers with positive speed. That is,  \begin{equation}\label{LLN}
\lim_{n\to\infty} \frac{X_n}{n} = \frac{1-E_P (\rho_0)}{1+E_P (\rho_0)} \equiv \vp > 0, 
\qquad \text{$\P$-a.s.}
\end{equation}
\item The RWRE satisfies a quenched functional central limit theorem with an 
environment-dependent  centering.  Let 
\[
 B^n(t) = \frac{X_{nt} - nt\vp + Z_{nt}(\w)}{\s_1 \sqrt{n}}, \qquad \text{where } \s_1^2 = \vp^3 E_P( \Var_\w T_1 ).
\] 
Then, for $P$-a.e.\ environment $\w$, under the quenched measure $P_\w$, $B^n(\cdot)$ converges weakly to standard Brownian motion as $n\to\infty$. 
\item Let
\[
 \zeta^n(t) = \frac{Z_{nt}(\w)}{\s_2 \sqrt{n}}, \qquad\text{where } \s_2^2 = \vp^2 \Var( E_\w T_1).
\]
Then, under the measure $P$ on environments, $\zeta^n(\cdot)$ converges weakly to standard Brownian motion as $n\to\infty$. 
\item The RWRE satisfies an averaged central limit theorem. Let
\[
 \mathbb{B}^n(t) = \frac{X_{nt} - nt\vp  }{\s \sqrt{n}}, \qquad \text{where } \s^2 = \s_1^2 + \s_2^2.
\] 
Then, under the averaged measure $\P$, $\mathbb{B}^n(\cdot)$ converges weakly to standard Brownian motion. 
\end{enumerate}
\end{theorem}
 
 The requirement that $E_P (\rho_0^2) < 1$  cannot be relaxed in order 
for the CLT to hold \cite{kest-kozl-spit, pete-zeit-09}. 
Centering by $nt\vp - Z_{nt}(\w)$ in  the quenched CLT is   the same 
as centering by the quenched mean on account of this bound: 
\be \label{limqmeanZ}
 \lim_{n\to\infty} P\Bigl\{ \w: \sup_{k\leq n} | E_\w (X_k) - k\vp + Z_k(\w) | \geq \e \sqrt{n} \Bigr\} = 0, \quad \forall \e>0. 
\ee
But $Z_{nt}(\w)$ is more convenient because it is  a sum of  stationary, ergodic random variables.  
 
 These properties of the walk are sufficient for describing the current 
 fluctuations.  Assumptions on the initial occupation variables  
  $\eta(0)=\{\eta_x(0)\}$ are similar to those in the previous section.  
We will allow the distribution  of $\eta(0)$ to depend on the environment (in a measurable way), 
and we  assume a certain  stationarity.  

\medskip

{\bf Assumption 2.} 
Given the environment $\w$, variables 
 $\{ \eta_x(0) \}$ are independent and independent of the   random walks. 
 The conditional distribution of $\eta_x(0)$ given $\w$ is denoted by 
  $  P_\w(\eta_x(0) = k)$, and these   measurable functions of $\w$ satisfy  
   $P_\w( \eta_x(0) = k ) = P_{\theta^x\w}( \eta_0(0) = k )$. Also, 
 for some $\e_0>0$,
\be
 E_P [\,E_\w(\eta_x(0))^{2+\e_0} + \Var_\w(\eta_x(0))^{2+\e_0} ]<\infty.   \label{vpmomass1}\ee
 
 \medskip
 
Let 
\[  \emean = E_P[ E_\w(\eta_x(0)) ] = \E [\eta_x(0)]
\quad\text{and}\quad    \sigma_0^2=E_P\left[ \Var_\w(\eta_x(0)) \right].  \]
 The current is defined as before:
 \be\begin{aligned}
Y_n(t,r) &= 
  \sum_{m\le 0} \sum_{k=1}^{\eta_m(0)} \indd{ X^{m,k}_{nt} >  \fl{nt\vp}+ r\sqrt n \,}\\
&\qquad\qquad\qquad 
-\;  \sum_{m> 0} \sum_{k=1}^{\eta_m(0)} \indd{ X^{m,k}_{nt} \leq \fl{nt \vp}+r\sqrt n\,}. 
 \end{aligned}\label{defYn}\ee
 
Now for the results, beginning with the quenched mean of the current.   This turns out to
essentially follow 
 the   correction  $Z_{nt}(\w)$ of the quenched CLT,  which is of order $\sqrt{n}$.
 Assumptions 1 and 2 are in force for all the results that follow. 
 
 \begin{theorem} \label{QMCurrent}
For any $\e>0$, $0<R, T<\infty$,  
\be  \lim_{n\to\infty}  P\Bigl\{ \w : \sup_{\substack{t\in[0,T]\\ r\in[-R,R]}}  \,
\bigl\lvert E_\w Y_n(t,r) +\emean r \sqrt n + \emean Z_{nt}(\w)  \bigr\rvert \ge \e\sqrt n\, \Bigr\} =0.  
\label{limYZ}\ee
Consequently 
 the two-parameter process $\{ n^{-1/2} E_\w Y_n(t,r): t\in \R_+, \; r\in \R \}$ 
 converges weakly to $\{ -\emean r+ \emean \s_2 W(t) : t\in \R_+, \; r\in\R\}$
 where $W(\cdot)$ is a standard Brownian motion. 
  \end{theorem}

Next we center the current at its 	quenched mean by defining 
\[  V_n(t,r)  =Y_n(t,r) -E_\w [Y_n(t,r)]. \]
 The fluctuations of $V_n(t,r)$ are of order $n^{1/4}$ and similar to the current fluctuations  of classical walks from the previous section.  
 Recall the definitions of $\Gamma_1$ and $\Gamma_2$ from \eqref{Ga1}--\eqref{Ga2}
  and 
abbreviate 
\be   \Gamma\bigl((s,q),(t,r)\bigr) =  \emean \Gamma_1\bigl((s,q),(t,r)\bigr) +\sigma_0^2
 \Gamma_2\bigl((s,q),(t,r)\bigr).  \label{Gammadef}\ee 

Let $(V,Z)=( V(t,r), Z(t):  t\in\R_+, r\in\R)$ be the  process whose joint
distribution is defined as follows:
\begin{itemize}  
\item[(i)] Marginally, $Z(\cdot)=\sigma_2 W(\cdot)$ for a standard Brownian motion $W(\cdot)$.
\item[(ii)] Conditionally on  the path $Z(\cdot)\in C(\R_+,\R)$,  $V$ is the mean zero Gaussian process indexed 
by $\R_+\times\R$ with covariance 
\be   \Ev[V(s,q)V(t,r)\,\vert\, Z(\cdot)]=\Gamma\bigl((s,q+Z(s)),(t,r+Z(t))\bigr)
 \label{vpcov}\ee\end{itemize}
 for $(s,q), (t,r)\in \R_+\times\R$.  
An equivalent way to say this is to first take independent $(V^0,Z)$ with $Z$ as above
and 
  $V^0=\{V^0(t,r): (t,r)\in\R_+\times\R\}$    the mean zero Gaussian process 
with 
covariance   $ \Gamma\big((s,q),(t,r)\big)$ from \eqref{Gammadef}, and then
define $V(t,r)=V^0(t,r+Z(t))$.  

\medskip

\begin{theorem}
Under the averaged probability $\P$, as $n\to\infty$,  
the finite-dimensional distributions of the joint process
$\bigl\{( n^{-1/4}V_n(t,r), n^{-1/2} Z_{nt}(\w) ):  t\in\R_+, r\in\R\bigr\}$ 
converge to those of the process
$(V,Z)$.  
\label{findimthm}\end{theorem}

\medskip
 
 Thus up to a random shift of the spatial argument  we see the same limit process as for
 classical walks: 
the  process
   $\bar V(t,r)=V(t,r-Z(t))$ is a
  mean zero Gaussian process  with
covariance   
$ \Ev[\bar V (s,q)\bar V (t,r)]= \Gamma\big((s,q),(t,r)\big)$ from \eqref{Gammadef}.  
 
As for  classical walks, let us look at the stationary case. 
 The invariant distribution is now valid under a fixed  $\w$:
  the $\{\eta_x(0)\}$ are independent and 
\be \label{fdef}
\eta_x(0) \sim \text{Poisson}(\emean f(\theta^x\w)),
\quad\text{where }\ 
f(\w) = 
\frac{\vp}{\w_0}\biggl( 1 + \sum_{i=1}^{\infty} \prod_{j=1}^i \rho_{j} \biggr). 
\ee 
In this case, $E_\w \eta_0(0) = \Var_\w \eta_0(0) = \emean f(\w)$.  By Assumption 1
  $E_P (\rho_0^{2+\e}) < 1$ for some $\e>0$, and from that
 it can be shown that $E_P[ f(\w)^{2+\e}] < \infty$. Therefore Assumption 2 holds. 
One can also check that, as for classical walks in \eqref{fluxex},  
in this stationary situation  the flux is   linear:     $\flux(\emean)=\emean\vp$.

Recall from Corollary \ref{fBMcor} that for 
  classical random walks the limit process (with fixed space variable
  $r$) in the case  $\emean = \s_0^2$ is
  fractional Brownian motion $\xi$ with covariance 
\[
\Ev[ \xi(s)\xi(t) ] = \frac{\emean\s_1}{\sqrt{2\pi}} ( \sqrt{s} + \sqrt{t} - \sqrt{|s-t|}\,).  
\]
 For RWRE, $\emean = \s_0^2$ implies that
\be \label{condcov}  \begin{aligned} 
&\Ev[V(s,0)V(t,0)\,\vert\, Z(\cdot)]\\
&= \emean \bigl[ \Psi_{\s_1^2 s}(-Z(s)) + \Psi_{\s_1^2 t}(Z(t)) - \Psi_{\s_1^2|s-t|}(Z(t)-Z(s)) \bigr]. 
\end{aligned}\ee
Since the right hand side of \eqref{condcov} is a non-constant random variable, the marginal distribution of $V(t,0)$ is non-Gaussian. Taking expectations of \eqref{condcov} with respect to $Z(\cdot)$ gives that 
\be \label{specialcasecov}
\Ev[V(s,0)V(t,0)] = \frac{\emean\sqrt{\s_1^2 + \s_2^2}}{\sqrt{2\pi}} ( \sqrt{s} + \sqrt{t} - \sqrt{|s-t|}). 
\ee
Thus we get this conclusion: 
 if $\emean = \s_0^2$ for RWRE 
  then the limit process $V(\cdot,0)$ has the same 
  covariance as  fractional Brownian motion, but it is not a Gaussian process. 
 
 As the reader may have surmised,  we can remove  the  random shift  $Z$
  from the limit   process $V$  by introducing 
  the environment-dependent
shift  in the current process itself.    We state this result too.
For $(t,r)\in\R_+\times\R$ define 
\be\begin{aligned}
Y_n^{(q)}(t,r) &= \sum_{m> 0} \sum_{k=1}^{\eta_m(0)} 
\indd{ X^{m,k}_{nt} \leq nt\vp -Z_{nt}(\w)+r\sqrt n\,} \\
&\qquad\qquad 
- \sum_{m\leq 0} \sum_{k=1}^{\eta_m(0)} \indd{ X^{m,k}_{nt} > nt\vp -Z_{nt}(\w)+ r\sqrt n \,}   
\end{aligned}\label{defYnq}\ee
and  its centered version
\[  V_n^{(q)}(t,r) = Y_n^{(q)}(t,r) - E_\w Y_n^{(q)}(t,r).  \]
The process $V_n^{(q)}$ has   the same limit as classical 
random walks.     
Let $V^0=\{V^0(t,r): (t,r)\in\R_+\times\R\}$  be  the mean zero Gaussian process 
with 
covariance   $ \Gamma\big((s,q),(t,r)\big)$ from \eqref{Gammadef}.  

\begin{theorem}
Under the averaged probability $\P$, as $n\to\infty$,  
the finite-dimensional distributions of the joint process 
$\bigl\{( n^{-1/4}V_n^{(q)}(t,r), n^{-1/2} Z_{nt}(\w) ):  t\in\R_+, r\in\R\bigr\}$ 
converge to those of the process
$(V^0,Z)$ where $V^0$ and $Z$ are independent.  
\label{findimthmVq}\end{theorem}

\subsection{Sketch  of the proof for the quenched mean of the current} 

The  basic thrust of the  proof of 
    Theorem \ref{findimthm} for the limit of the centered current is similar to the
one  outlined in Section \ref{rwpfsec} for classical random walks.   The differences lie in
the technical details needed to handle the random environment.  So we omit further
discussion of  that theorem and of the related  Theorem \ref{findimthmVq}.  
In this section 
we explain the main ideas behind the proof of Theorem  \ref{QMCurrent} for
the quenched mean of the current.   To put aside  inessential detail   
 we drop the uniformity, fix $(t,r)$,  and sketch informally the argument for the
 following simplified  statement:

 \begin{proposition} \label{QMa}
For any $\e>0$,  
\be  \lim_{n\to\infty}  P\Bigl\{ \w :    \,
\bigl\lvert E_\w Y_n(t,r) +\emean r \sqrt n + \emean Z_{nt}(\w)  \bigr\rvert \ge \e\sqrt n\, \Bigr\} =0.  
\label{limYZa}\ee
   \end{proposition} 

This proceeds via a sequence of estimations.  Abbreviate the centered quenched
mean by 
\begin{align*}   W_n&= E_\w Y_n(t,r) +\emean r \sqrt n\\
&=  \sum_{m\leq 0}E_\w[ \eta_0(m) ] P_\w( X^m_{nt} >  nt\vp + r\sqrt{n}\, )\\
&\qquad\qquad - \;  \sum_{m> 0} E_\w[ \eta_0(m) ] P_\w( X^m_{nt} \leq nt\vp + r\sqrt{n}\, ) 
+ \dc r\sqrt{n}. \nn
\end{align*}
For a suitable sequence $a(n)\nearrow\infty$  define 
\begin{align*}
\wt{W}_{n}&=  \sum_{m=-\fl{a(n)\sqrt{n}}+1}^0 \!\!\!  E_\w( \eta_0(m)) \Phi_{\s_1^2 t}\left( - \frac{Z_{nt}(\theta^m\w)-m}{\sqrt{n}} - r \right)\\
&\qquad\qquad   -\; \sum_{m=1}^{\fl{ a(n)\sqrt{n}}} E_\w( \eta_0(m)) \Phi_{\s_1^2 t}\left( \frac{Z_{nt}(\theta^m\w)-m}{\sqrt{n}} + r \right)  
  + \emean r \sqrt{n}.
\end{align*}
The quenched CLT (part 2 of Theorem \ref{QCLTthm}) implies that 
the difference 
\[
  \sup_{x\in\R} \,\biggl\lvert  P_\w\left( \frac{X_n - n\vp + Z_n(\w)}{\sqrt{n}} \leq x \right) - \Phi_{\s_1^2}(x) \biggr\rvert 
\]
vanishes $P$-a.s.\ as $n\to\infty$.  Thus it is possible to choose $a(n)\nearrow\infty$ 
slowly enough so that 
\be 
\lim_{n\to\infty} P\left(   \frac{1}{\sqrt{n}}\left|W_{n} - \wt{W}_{n}\right| \geq \e\right) = 0.
\label{re:1} \ee
 
The next step is to remove the shifts from $Z_{nt}$  by defining
\begin{align*}
\wh{W}_{n}&=  \sum_{m=-\fl{a(n)\sqrt{n}}+1}^0 \!\!\!  E_\w( \eta_0(m)) \Phi_{\s_1^2 t}
\left( - \frac{Z_{nt}(\w)-m}{\sqrt{n}} - r \right)\\
&\qquad\qquad   -\; \sum_{m=1}^{\fl{ a(n)\sqrt{n}}} E_\w( \eta_0(m)) \Phi_{\s_1^2 t}\left( \frac{Z_{nt}(\w)-m}{\sqrt{n}} + r \right)  
  + \emean r \sqrt{n}.
\end{align*}
The estimation 
\be 
\lim_{n\to\infty} P\left(   \frac{1}{\sqrt{n}}\left|\wt W_{n} - \wh{W}_{n}\right| \geq \e\right) = 0
\label{re:2} \ee  
follows from   representation \eqref{Znt} of $Z_{nt}$ as a sum of ergodic terms whose
behavior is well understood.    

Subsequently the ergodicity of the environment allows us to average the
quenched means $E_\w( \eta_0(m))$ and replace $\wh W_n$ with 
\begin{align*}
\wb{W}_{n}&=  \sum_{m=-\fl{a(n)\sqrt{n}}+1}^0 \!\!\!  \emean \Phi_{\s_1^2 t}
\left( - \frac{Z_{nt}(\w)-m}{\sqrt{n}} - r \right)\\
&\qquad\qquad\qquad  -\; \sum_{m=1}^{\fl{ a(n)\sqrt{n}}} \emean \Phi_{\s_1^2 t}\left( \frac{Z_{nt}(\w)-m}{\sqrt{n}} + r \right)  
  + \emean r \sqrt{n}.
\end{align*}

After this, replace the sums with integrals to approximate  $n^{-1/2}\wb W_n$ with 
\[  \dc r \; - \; \dc \int_0^{a(n)} \Phi_{\s_1^2 t}\left( \frac{Z_{nt}(\w)}{\sqrt{n}} + r -x \right) - \Phi_{\s_1^2 t}\left( - \frac{Z_{nt}(\w)}{\sqrt{n}} - r -x \right) dx.  \]
A calculus exercise  shows that  
\[
\int_0^{A} \Phi_{\alpha^2}\left( z -x \right) - \Phi_{\alpha^2}\left( -z  -x \right) dx = z + \Psi_{\alpha^2}(A+z) - \Psi_{\alpha^2}(A-z),  
\]
where $\Psi_{\alpha^2}(x)$ is again the function  defined in \eqref{Psi}. 
Therefore,
\begin{align*}
&\int_0^{a(n)} \Phi_{\s_1^2 t}\left( \frac{Z_{nt}(\w)}{\sqrt{n}} +r -x \right) - \Phi_{\s_1^2 t}\left( - \frac{Z_{nt}(\w)}{\sqrt{n}} - r -x \right) dx \\[4pt]
&\qquad = \frac{Z_{nt}(\w)}{\sqrt{n}} + r + \Psi_{\s_1^2 t}\left(a(n)+\frac{Z_{nt}(\w)}{\sqrt{n}} + r \right)\\[4pt]
&\qquad\qquad\qquad
 - \Psi_{\s_1^2 t}\left(a(n)-\frac{Z_{nt}(\w)}{\sqrt{n}} - r \right).
\end{align*}
The last two terms vanish as $n\to\infty$ because $\Psi_{\s_1^2 t}(\infty)=0$ and 
$n^{-1/2}Z_{nt}$ is tight by part 3 of Theorem \ref{QCLTthm}.  

Collecting the steps (and taking the omitted details on faith) leads to the estimate 
\[   W_n = -\emean Z_{nt}  + o(\sqrt n\,)  \qquad\text{in probability} \]
as was claimed in \eqref{limYZa}.

\subsection*{References}
The results for the current of RWRE's is from \cite{pete-sepp}.
 Zeitouni's lecture notes \cite{zeit-stflour}  
 are a  standard reference for background on RWRE.

\section{Random average process} 
\setcounter{equation}{0}
\subsection{Model and results} 

The state of the  random average process (RAP) 
is a height function  $\si: \Zb\to\Rb$  where the value $\si(i)$ can be thought of as 
the height of an interface above site $i$. 
The state evolves  in discrete time  according to the following rule. 
At each time
point $ s=1,2,3,\dotsc$ and 
at each site $k\in\Zb$, a  random  probability vector 
$\om_{k,s}=(\om_{k,s}(j): -R\leq j\leq R )$  of length $2R+1$ is drawn. 
Given the state $\si_{{s}-1}=(\si_{{s}-1}(i)\,:\,i\in\Zb)$ 
at time ${s}-1$,  
at time ${s}$  the height value at site $k$ is  updated to 
 \be
\si_{s}(k)=\sum_{j: \abs{j}\leq R}\om_{k,s}(j) \si_{{s}-1}(k+j).
\label{si-dyn-1}
\ee
This update is performed independently at each site $k$ to form
the state $\si_{s}=(\si_{s}(k)\,:\,k\in\Zb)$ at time ${s}$. 
The weight vectors $\{\om_{k,s}\}_{k\in\bZ, s\in\bN}$ are i.i.d.\ across 
space-time points $(k,s)$. 
This system was originally studied by Ferrari and Fontes \cite{ferr-font-rap}. 

Let \[p(0,j)=\bE[ \om_{0,0}(j)]\] denote the averaged weights with
mean and  variance  
\be
V=\sum_x x\,p(x)  \quad\text{and}\quad 
\s_1^2=\sum_{x\in\Zb} (x-V)^2\; p(0,x).
\label{def-si_a}
\ee
Let $b=-V$.  

Make two nondegeneracy  assumptions on the distribution of the weight vectors.  

(i) There is no 
integer $h>1$ such that, for some $x\in\Zb$, 
\[
\sum_{k\in\Zb} p(0, x+kh)=1.
\]
This is also expressed by saying that the {\sl span}
 of the random walk with  
jump probabilities $p(0,j)$  is 1 \cite[page 129]{durr}.
It follows that the additive group generated by $\{x\in\Zb: p(0,x)>0\}$ 
is all of $\Zb$, in other words this walk is {\sl aperiodic} in 
Spitzer's terminology \cite{spitzer}. 

(ii) Second, we assume that
\be
\bP\{ \max_j \om_{0,0}(j) < 1\}>0. 
\label{ellipt}
\ee

Let $\si_{s}$ 
be a random average process normalized by $\si_0(0)=0$
and whose initial increments $\{\eta_i(0)=\si_0(i)-\si_0(i-1):i\in\Zb\}$ 
are i.i.d.~such that 
\be
\text{there exists $\alpha>0$ 
such that } \quad
  E[\,\lvert\eta_i(0)\rvert^{2+\alpha} \,]
<\infty. 
\label{eta-ass-p}
\ee
As before, the mean and variance of initial increments are 
\[  \emean=E(\eta_i(0))  
\quad\text{and}\quad 
\evar^2=\Var(\eta_i(0)). \]  
The initial increments $\eta_0$ are independent of the weight vectors $\{\om_{k,s}\}$. 


Again we study a suitably scaled  process of fluctuations in the characteristic 
direction:  for $(t,r)\in\Rb_+\times\Rb$,  let
\[
\Ybar_n(t,r)=n^{-1/4}
\bigl\{\si_{\fl{nt}}^n( \fl{r\sqrt{n}\,}+\fl{ntb})
-\emean r\sqrt n\,\bigr\}.  
\]
  In terms of the increment 
process \[ \eta_i(s)=\si_{s}(i)-\si_{s}(i-1), \]
$\Ybar_n(t,r)$ 
is the centered and scaled net flow  from right to left across the  
path $s\mapsto \fl{r\sqrt{n}\,} + \fl{nsb}$,
during the time interval $0\leq s\leq t$, exactly as for independent particles. 

Recall the definitions \eqref{Ga1} and \eqref{Ga2} of the functions $\Gamma_1$ and $\Gamma_2$. 

\begin{theorem} 
 Under the above assumptions  
the finite-dimensional distributions 
of the  process $\{\Ybar_n(t,r): (t,r)\in\Rb_+\times\Rb \}$ converge weakly 
as $n\to\infty$ 
to the finite-dimensional distributions 
of the   mean zero Gaussian
process $\{Z(t,r): (t,r)\in\Rb_+\times\Rb \}$ specified by the covariance 
\be
\begin{split}
\mE Z(s,q)Z(t,r)&= 
\emean^2 \kappa\,\Gamma_1((s,q),(t,r))
+ \evar^2 \Gamma_2((s,q),(t,r)).
\end{split}
\label{rap-cov}
\ee
\label{rap-thm-1}
\end{theorem}

The constant $\kappa$ is determined by the distribution  of the random weights and will
be described precisely later in equation \eqref{def-kappa}.  

Invariant distributions for the general RAP are not known.   The next example may be
the only one where explicit invariant distributions are available. 

\begin{example} Fix positive real parameters $\theta>\alpha>0$. 
Let $\{\om_{k,s}(-1): {s}\in\Nb, k\in\Zb\}$ 
be i.i.d.~ Beta($\alpha, \theta-\alpha$) random variables
with density  
\[
h(u)=\frac{\Gamma(\theta)}{\Gamma(\alpha)\Gamma(\theta-\alpha) }u^{\alpha-1}(1-u)^{\theta-\alpha-1}
\]
on $(0,1)$. Set $\om_{k,s}(0)=1-\om_{k,s}(-1)$. Thus the weights are supported
on $\{-1,0\}$.  
A family of  invariant distributions for the
increment  process 
$\eta(s)=(\eta_{k}(s): k\in\Zb)$ is obtained by
letting the variables $\{\eta_k: k\in\Zb\}$ be i.i.d.~Gamma($\theta, \lambda$) 
distributed with common density 
\be
f(x)=\frac1{\Gamma(\theta)} \lambda e^{-\lambda x}(\lambda x)^{\theta-1}
\label{gamma-eq}
\ee
on $\Rb_+$. This family of invariant distributions is parametrized
by $0<\lambda<\infty$. Under this  distribution 
$
E^\lambda [\eta_k]= {\theta}/\lambda 
$ and  
$\Var^\lambda[\eta_k]={\theta}/{\lambda^2} . 
$
\label{example-1}
In this situation we find again the  fractional Brownian motion limit:
 \be
\mE Z(s,0)Z(t,0)=  
c_1\bigl( \sqrt{s}+\sqrt{t}-\sqrt{\abs{t-s}}\,\bigr).
\label{fBM-1}
\ee
for a certain constant $c_1$.  
 \end{example}

\subsection{Steps of the proof}  

\noindent 
{\bf 1. Representation in terms of space-time RWRE}

\medskip

Let  
$\om=( \om_{k,s}: {s}\in\Nb,\, k\in\Zb)$ represent the i.i.d.\ random weight vectors
that determine the dynamics, coming from a probability space $(\Omega, \kS, \bP)$.  
Given $\om$ and a space-time point $(i,\tau)$, let $\{X^{i,\tau}_s: s\in \Zb_+\}$ denote a random
walk on $\Zb$ that starts at $X^{i,\tau}_0=i$ and whose
transition probabilities are given by 
\be
P^\om(X_{s+1}^{i,\,\tau}=y\,|\,X^{i,\,\tau}_s=x)=\om_{x,\tau-s}(y-x).
\label{X-tr-pr}
\ee
$P^\om$ is the path measure of the walk $X^{i,\tau}_s$, with
expectation denoted by $E^\om$.  Comparison of 
 \eqref{si-dyn-1}  and \eqref{X-tr-pr} gives
 \be
\si_{s}(i)=\sum_j 
P^\om(X_{1}^{i,\,{s}}=j\,|\,X^{i,\,{s}}_0=i)  \si_{{s}-1}(j)
=E^\om\bigl[ \si_{{s}-1}(X_{1}^{i,\,{s}})\bigr].
\label{si-dyn-2}
\ee
Iteration and the Markov property of the walks $X^{i,{s}}_s$  then lead to 
\be
\si_{s}(i)=E^\om\bigl[\si_0(X^{i,\,{s}}_{s})\bigr].
\label{eq:conn}
\ee
Note that the initial height function $\sigma_0$ is a constant 
under the expectation $E^\om$. 

Let us add another coordinate to keep track of time and write
$\Xbar^{i,\tau}_s=(X^{i,\tau}_s, \tau-s)$ for $s\geq 0$.  Then 
$\Xbar^{i,\tau}_s$ is a random walk on the planar lattice 
$\Zb^2$ that always moves down one step in the $e_2$-direction,
and if its current position is $(x,n)$,  then 
  its next position is $(x+y,n-1)$ with probability 
 $\om_{x,n}(y-x)$.  We could call this a backward random
walk in a (space-time, or dynamical)  random environment. 

The opening step of the proof is to use the random walk representation
to rewrite the random variable $\Ybar_n(t,r)$ in 
a manner that allows us to separate the effects of the random initial
conditions from the effects of the random weights.  
Abbreviate 
\[
y(n)=\fl{ntb}+\fl{r\sqrt{n}\,}.
\]
and recall   $\emean=E \eta_i(0)$ and $\si_0(0)=0$. 
\begin{align*}
&\si_{{\fl{nt}}}(y(n)) =
E^\om\bigl[\si_0( X^{y(n),\, \fl{nt}}_{\fl{nt}}) \bigr]\\
&=E^\om\biggl[{\bf1}_{\left\{ X^{y(n),\, \fl{nt}}_{\fl{nt}}>0\right\}}
\sum_{i= 1}^{ X^{y(n),\, \fl{nt}}_{\fl{nt}}}\eta_i(0)
\\  &\qquad\qquad\qquad\qquad
\; - \; {\bf1}_{\left\{ X^{y(n),\, \fl{nt}}_{\fl{nt}}<0\right\}}
\sum_{i= X^{y(n),\, \fl{nt}}_{\fl{nt}}+1}^{0}\eta_i(0)\biggr]\\
 &=\sum_{i>0} \eta_i(0) P^\om\left\{i\leq X^{y(n),\, \fl{nt}}_{\fl{nt}}\right\} \; - \;
  \sum_{i\leq
0}\eta_i(0) P^\om\left\{i> X^{y(n),\, \fl{nt}}_{\fl{nt}}\right\} \\ 
&=
\emean H_n(t,r)+S_n(t,r)
 \end{align*}
 where 
 \begin{align*}
H_n(t,r)&=   \sum_{i\in\Zb} 
\Bigl( \ind\{i>0\} P^\om\bigl\{ i\leq X^{y(n),\, \fl{nt}}_{\fl{nt}}\bigr\}  
\\ &\qquad \qquad  
\;-\; \ind\{i\leq 0\} P^\om\bigl\{ i> X^{y(n),\, \fl{nt}}_{\fl{nt}}\bigr\}
\Bigr)\\
&=    E^\om\bigl( X^{y(n),\, \fl{nt}}_{\fl{nt}} \bigr)\\
\end{align*}
and 
\[
\begin{split}
S_n(t,r)&= \sum_{i\in\Zb}\bigl(\eta_i(0)-\emean\bigr)
\Bigl( \ind\{i>0\} P^\om\bigl\{ i\leq X^{y(n),\, \fl{nt}}_{\fl{nt}}\bigr\}  
\\ &\qquad \qquad 
\; -\; \ind\{i\leq 0\} P^\om\bigl\{ i> X^{y(n),\, \fl{nt}}_{\fl{nt}}\bigr\}
\Bigr). 
\end{split}
\] 

At this point the terms $H_n$ and $S_n$ are dependent,  but in the course
of the scaling limit they become independent and furnish the two independent
pieces that make up the limiting process $Z$.    The limits 
$n^{-1/4}(H_n(t,r)-r\sqrt n)  \limd  H(t,r)$ and
$n^{-1/4}S_n(t,r)  \limd  S(t,r)$ are treated separately, and then together 
\[ \Ybar_n=   n^{-1/4}(H_n-r\sqrt n+S_n)\limd H+S\equiv Z\] 
with  independent terms $H$ and $S$.   This independence comes from
the independence of   the initial height function 
$\si_0$ and the random environment $\om$ that drives the 
dynamics.  The idea   is represented  in the next 
lemma. 
 
\begin{lemma}
Let  $\eta$ and  $\omega$
be  independent random variables  with values in some abstract measurable spaces.
 Let $ h_n(\omega)$  and $ s_n(\omega,\eta)$
be measurable functions of $(\omega,\eta)$.  Let $E^\om(\cdot)$ denote 
conditional expectation, given $\om$. 
Assume the existence of random variables $h$ and $s$ such that  
\begin{itemize}
\item[(i)]  
  $ h_n(\omega)\limd  h$; 
\item[(ii)] for all $\theta\in\Rb$,  
$
E^\om[e^{i\theta  s_n}]\to  E(e^{i\theta  s}) \quad\text{ in  probability 
as $n\to\infty$.}
$
 \end{itemize}
Then  
$ h_n+ s_n\limd h+ s$, 
where $ h$ and $ s$ are independent.
\label{main-lm}
\end{lemma}

\begin{proof} 
 Let $\theta,\lambda\in\bbR$. Then
\begin{align*}
&\left\vert E\bigl( E^\om[e^{i\lambda  h_n+i\theta  s_n}]\bigr) 
- E[e^{i\lambda  h}]\, E[e^{i\theta  s}]\right\vert\\
&\qquad\leq
\left\vert E\left[e^{i\lambda  h_n}
\left( E^\omega e^{i\theta  s_n} - E e^{i\theta  s}\right)\right]
\right\vert  +
\left\vert 
\left(E e^{i\lambda  h_n} - E e^{i\lambda  h}\right)
E e^{i\theta  s}
\right\vert\\
&\qquad\leq  \left\vert E\left[e^{i\lambda  h_n}
\left(E^\omega e^{i\theta  s_n} - E e^{i\theta  s}\right)\right] \right\vert  +
\left\vert E e^{i\lambda  h_n} - E e^{i\lambda  h}
\right\vert.
\end{align*}
By assumption (i), the second term above goes to 0. By assumption (ii), the 
integrand in the first term goes to 0 in  probability. Therefore by
bounded convergence the first term goes to 0 as $n\to\infty$.
\end{proof}

We discuss the term $S_n$   briefly and reserve most of our attention to $H_n$. 
  Two limits combine to give the result. 
The idea is to apply  the   Lindeberg-Feller theorem 
to  $  S_n(t,r) $ under a fixed $\om$.   Then the $\om$-dependent 
coefficients provide no fluctuations but instead converge to Brownian probabilities
due to 
  a quenched central limit theorem for the space-time RWRE.   Here is an 
  informal presentation where we first imagine that the coefficients can be replaced
  by deterministic quantities:  
   \begin{align*}
&S_n(t,r)= \sum_{x\in\Zb}\bigl(\eta_x(0)-\emean\bigr)\\
&\times\biggl(  \ind\{x>0\} 
P^\om\biggl\{  \frac{X^{ \fl{ntb}+\fl{r\sqrt{n}\,}\,,\, \fl{nt}}_{\fl{nt}}-  {r\sqrt{n}\,}}{\sqrt n}
\ge  \frac{x}{\sqrt n} -r  \biggr\}  
\\ &\qquad  
\; -\; \ind\{x\leq 0\} 
P^\om\biggl\{  \frac{X^{ \fl{ntb}+\fl{r\sqrt{n}\,}\,,\, \fl{nt}}_{\fl{nt}}-  {r\sqrt{n}\,}}{\sqrt n}
< \frac{x}{\sqrt n} -r \biggr\}  \,  
\biggr) \\[7pt]
&\approx \sum_{x\in\Zb}\bigl(\eta_x(0)-\emean\bigr)
\biggl(  \ind\{x>0\} 
\mP \Bigl\{ B_{\s_1^2t}   >  \frac{x}{\sqrt n} -r  \Bigr\}  
\\ &\qquad  \qquad  \qquad  
\; -\; \ind\{x\leq 0\} 
\mP \Bigl\{  B_{\s_1^2t} \le  \frac{x}{\sqrt n} -r \Bigr\}  \,  
\biggr) \\
\end{align*} 
Now apply the 
Lindeberg-Feller theorem  to the remaining sum of independent initial occupation
variables, and the limiting covariance comes as:
\begin{align*}
&n^{-1/2} \sum_{i,j} \theta_i\theta_j E[S_n(t_i,r_i) S_n(t_j,r_j)] \\
&\approx \evar^2   \sum_{i,j} \theta_i\theta_j \, 
n^{-1/2} \biggl[ \; \sum_{x>0} \mP \Bigl\{ B_{\s_1^2t_i}   >  \frac{x}{\sqrt n} -r_i  \Bigr\}  
 \mP \Bigl\{ B_{\s_1^2t_j}   >  \frac{x}{\sqrt n} -r_j  \Bigr\} \\
&\qquad\qquad\qquad 
 +   \sum_{x\le 0} \mP \Bigl\{ B_{\s_1^2t_i}   \le \frac{x}{\sqrt n} -r_i  \Bigr\}  
 \mP \Bigl\{ B_{\s_1^2t_j}   \le  \frac{x}{\sqrt n} -r_j  \Bigr\} \; \biggr] \\
&\approx \evar^2   \sum_{i,j} \theta_i\theta_j \, 
  \biggl[ \;\int_0^\infty \mP  \{ B_{\s_1^2t_i}   >  x -r_i   \}  
 \mP  \{ B_{\s_1^2t_j}   >  x -r_j   \}\,dx \\
&\qquad\qquad\qquad 
 +   \int_{-\infty}^0  \mP  \{ B_{\s_1^2t_i}   \le x -r_i   \}  
 \mP  \{ B_{\s_1^2t_j}   \le x -r_j   \}\,dx \; \biggr] \\
&=\evar^2   \sum_{i,j} \theta_i\theta_j \Gamma_2((t_i,r_i),(t_j,r_j)).  
\end{align*} 

Turning this argument rigorous gives  the limit in $\bP$-probability: 
\be \begin{aligned} & \lim_{n\to\infty}  E^\om\bigl[ e^{i \sum_{k}  \theta_k S_n(t_k,r_k)}\bigr]    
=E\bigl[ e^{i \sum_{k}  \theta_k S(t_k,r_k)}\bigr]   \\ 
&\qquad 
= \exp\Bigl\{-\tfrac12  \evar^2   \sum_{i,j} \theta_i\theta_j \Gamma_2((t_i,r_i),(t_j,r_j))\Bigr\}.
\end{aligned}\label{Slim} \ee

In particular,  in the limit the fluctuations of $S$ come from the 
initial occupation variables $\eta_i(0)$ and hence are independent of the
weights $\om$ that determine $H_n$.

 \bigskip
 
\noindent
{\bf 2. Quenched mean of the backward space-time RWRE} 

\medskip

 The remaining piece of the fluctuations comes from  
\be \Hbar_n(t,r) =  n^{-1/4}  
E^\om\bigl( X^{ \fl{ntb}+\fl{r\sqrt{n}\,} ,\, \fl{nt}}_{\fl{nt}} -\fl{r\sqrt{n}\,}\bigr) 
 \label{defHbar}\ee

\begin{theorem} In the sense of convergence of finite-dimensional distributions,
$\Hbar_n\limd H$ where $H(t,r)$ is the mean zero Gaussian process with covariance 
\be\begin{aligned} 
\mE H(s,q)H(t,r) &=  
\kappa \Gamma_1((s,q),(t,r)) \\[3pt] 
&= 
\frac\kappa{2}
\int_{ \s_1^2\lvert t-s\rvert}
^{\s_1^2(t+s)} \frac1{\sqrt{2\pi v}}
\exp\Bigl\{-\frac1{2v}(q-r)^2\Bigr\} \,d v. 
\end{aligned}\label{y-cov}
\ee
\label{qmeanthm}\end{theorem} 

The constant $\kappa$ is defined below in \eqref{def-kappa}. 
 
 By comparing covariances (Exercise \ref{stheateqex1})
  one  checks that $H$ can also be defined by 
 \be  H(t,r)= \sqrt\kappa
\iint_{[0,t]\times\Rb} \varphi_{\sigma^2_1(t-s)}(r-z)\,
dW(s,z).
\label{y-st-int}
\ee
 Formula \eqref{y-st-int} implies
that  process $\{H(t,r)\}$ is a weak solution for this initial value problem 
of the  stochastic heat equation: 
\be
 H_t=\tfrac{\sigma_1^2} 2 H_{rr} + \sqrt\kappa\, \dot{W}\,,
\qquad H(0,r)\equiv 0.  
\label{st-heat-eqn}
\ee
 
\hbox{}

Proof of Theorem \ref{qmeanthm} happens in two steps:  first for multiple space points
at a fixed time, and then across time points. 

\medskip

{\bf Step 1.  Martingale increments for fixed time.  }

\medskip

Let us abbreviate $X^*_s=X^{ \fl{ntb}+\fl{r\sqrt{n}\,} ,\, \fl{nt}}_s$.   Note that 
\[  E(X^*_{\fl{nt}})= \fl{ntb}+\fl{r\sqrt{n}\,} +\fl{nt}V =r\sqrt n + O(1)  \]
so $\Hbar(t,r)$ in \eqref{defHbar} is essentially centered and we can pretend that it is 
exactly centered. 
Let 
\[    g(\om)= E^\om(X_1^{0,0})-V \]
be the centered local drift.  
Recall the space-time walk $\Xbar^{x,s}_m=(X^{x,s}_m, s-m)$. 
By the Markov property  of the walk 
\be
\begin{aligned}
&E^\om(X^{x,s}_n)-x-nV=
 \sum_{k=0}^{n-1} E^{\om}\bigl[
X^{x,s}_{k+1}-X^{x,s}_k-V \bigr]\\
&= \sum_{k=0}^{n-1} E^{\om}\bigl[
E^{ T_{\{\Xbar^{x,s}_k\}}\om} (X^{0,0}_1-V)\bigr]
=\sum_{k=0}^{n-1} E^{\om} g( T_{\Xbar^{x,s}_k}\om).
\end{aligned}
\label{rw-calc-1} 
\ee
   $(T_{x,m}\om)_{y,s}=\om_{x+y,m+s}$ is the space-time shift of environments.  
 The $g$-terms above are martingale increments under the distribution $\P$ of the 
environments, relative to the filtration defined by levels of environments:
writing  $\bar\om_{m,n}=\{\om_{x,s}: x\in\Z, m\le s\le n\}$, and with fixed $(x,m)$ 
and time $n=0,1,2\dotsc$, 
\begin{align*}
&\bE\bigl[ E^{\om} g( T_{\Xbar^{x,m}_n}\om) \big\vert \bar\om_{m-n+1,m}\bigr] \\
&\quad =\sum_{y\in\bZ}  P^\om\{ \Xbar^{x,m}_n = (y,m-n) \} 
 \int g( T_{y,m-n}\om)\,\bP(d\bar\om_{m-n}) = 0. 
\end{align*}
The point above is that the probability $P^\om\{ \Xbar^{x,m}_n = (y,m-n) \} $ is
determined by $\bar\om_{m-n+1,m}$.  

It turns out that we can apply a martingale central limit theorem to conclude that, for a 
fixed $t$,  
a vector 
\[  \bigl(  \Hbar_n(t,r_1), \Hbar_n(t,r_2),  \dotsc, \Hbar_n(t,r_N) \bigr)   \]
becomes a Gaussian vector in the $n\to\infty$ limit.  Let us take this for granted, and
compute the covariance of the limit.   This leads us to study an auxiliary Markov chain
which has been useful for space-time (and more general ballistic) RWRE. 

Given points $(t,q)$ and $(t,r)$, abbreviate $X^{(1)}_s=X^{ \fl{ntb}+\fl{q\sqrt{n}\,} ,\, \fl{nt}}_s$
and $X^{(2)}_s=X^{ \fl{ntb}+\fl{r\sqrt{n}\,} ,\, \fl{nt}}_s$. 
\begin{align*} 
&\E\bigl[ \Hbar_n(t,q)\Hbar_n(t,r) \bigr]\\
&=n^{-1/2} \E \Bigl[  E^\om\bigl( X^{(1)}_{\fl{nt}} -\fl{q\sqrt{n}\,}\bigr) 
E^\om\bigl( X^{(2)}_{\fl{nt}} -\fl{r\sqrt{n}\,}\bigr) \Bigr]\\
&=n^{-1/2} \E \biggl[ \,  \biggl(\,  \sum_{j=0}^{\fl{nt}-1} E^{\om} g( T_{\Xbar^{(1)}_j}\om) \biggr)
 \biggl(\,  \sum_{k=0}^{\fl{nt}-1} E^{\om} g( T_{\Xbar^{(2)}_k}\om) \biggr) \, \biggr] \\
&=n^{-1/2} \sum_{j,k} \sum_{x,y}  
\E \Bigl[  P^\om(X^{(1)}_{j}=x)  P^\om(X^{(2)}_{k}=y)  g(T_{x, \fl{nt}-j}\om)  g(T_{y, \fl{nt}-k}\om)   
 \Bigr]  \\
&=\sigma_D^2 n^{-1/2}  \sum_{k=0}^{\fl{nt}-1}  \E P^\om(X^{(1)}_{k}=X^{(2)}_{k}) . 
\end{align*}
The last step uses the independence of the environment.  
We denote the variance of the drift by $\s_D^2=\E(g^2)$.  Let now 
$Y_k=X^{(2)}_{k}-X^{(1)}_{k}$ be the difference of two independent walks in a common
environment.  $Y_k$ is a Markov chain on $\bZ$ with transition probability 
\[
q(x,y)=  \begin{cases}  \smallskip \ddd\sum_{z\in\Zb} \bE[\om_{0,0}(z)\om_{0,0}(0,z+y)]  &x=0\\ 
         \ddd\sum_{z\in\Zb} p(0,z)p(0,z+y-x)  &  x\neq 0.  
\end{cases} \] 
$Y_n$ can be thought of 
as  a symmetric random walk on $\Zb$ whose transition
has been perturbed at the origin. The corresponding  homogeneous,
unperturbed 
transition probabilities are 
\[
\bar{q}(x,y)=\bar{q}(0,y-x)=
\sum_{z\in\Zb} p(0,z)p(0,z+y-x) 
\qquad (x ,y\in\Zb).
\]

Continuing from above,  with $x_n=\fl{r\sqrt{n}\,}-\fl{q\sqrt{n}\,}$, 
\be 
\E\bigl[ \Hbar_n(t,q)\Hbar_n(t,r) \bigr]=\frac{ \sigma_D^2}{\sqrt n}  \sum_{k=0}^{\fl{nt}-1} 
q^k(x_n, 0) =\frac{ \sigma_D^2}{\sqrt n}  G_{\fl{nt}-1}(x_n,0). 
\label{G4}\ee
If $Y_k$ were a symmetric random walk, we would know this limit exactly from the
local central limit theorem:

\begin{lemma} For a mean $0$, span $1$  random walk $S_n$ 
on $\Z$ with finite variance $\sigma^2$,  $a\in\bR$ 
and points $a_n\in\bZ$ such that $\abs{a_n-a\sqrt n}=O(1)$, 
\[ \lim_{n\to\infty}  \frac1{\sqrt n}\sum_{k=0}^{\fl{nt}-1} P(S_k=a_n)
=\frac1{\sigma^2} \int_0^{\sigma^2t} \frac1{\sqrt{2\pi v}}{\exp\Bigl\{-\frac{a^2}{2v}\Bigr\}}\,dv. 
\] 
\label{Glimlm}\end{lemma}

\begin{proof}  By the local CLT \cite[Section 2.5]{durr}
\[   \lim_{m\to\infty} \sup_{x\in\Z} \sqrt m\,\Bigl\lvert  P(S_m=x) - \frac1{\sqrt{2\pi m\sigma^2}}
\exp\Bigl\{-\frac{x^2}{2m\sigma^2}\Bigr\} 
\Bigr\rvert =0.  \]
Use this in a Riemann sum argument (details as exercise). 
\end{proof} 

For the homogeneous $\qbar$-walk this lemma gives  (using symmetry)
\be  \lim_{n\to\infty} 
\frac{ 1}{\sqrt n}  \sum_{k=0}^{\fl{nt}-1} 
\qbar^k(x_n, 0) = 
\frac{1}{2\sigma_1^2} \int_0^{2\sigma_1^2t} \frac1{\sqrt{2\pi v}}
{\exp\Bigl\{-\frac{(r-q)^2}{2v}\Bigr\}}\,dv. 
\label{G8}\ee
Now the task is to relate the transitions $q$ and $\qbar$.  For this purpose we introduce
one more player: the potential kernel of the symmetric $\qbar$-walk, defined by 
\be\begin{aligned}  \abar(x)   &=\lim_{n\to\infty} \bigl[ \bar G_n(0,0)-\bar G_n(x,0)\bigr] \\
&=\lim_{n\to\infty} \biggl\{ \; \sum_{k=0}^{n}  \qbar^k(0, 0) 
- \sum_{k=0}^{n}  \qbar^k(x, 0) \biggr\}.  \end{aligned}  \label{defabar}\ee
The potential kernel satisfies $\abar(0)=0$, 
  the equations 
\be
\abar(x)=\sum_{y\in\Zb} \qbar(x,y)\abar(y) 
\quad\text{for $x\neq 0$, and} \quad 
\sum_{y\in\Zb} \qbar(0,y)\abar(y)=1, 
\label{abar-eq}
\ee
and the limit 
\be 
\lim_{x\to\pm\infty}\frac{\abar(x)}{\abs{x}} = \frac1{2\sigma_1^2}.
\label{abar-lim}
\ee
(For existence of $\abar$ and its properties, see \cite[Sections 28-29]{spitzer}.)

\begin{example} If for some $k\in\Zb$, $p(0,k)+p(0,k+1)=1$, so that
$\qbar(0,x)=0$ for $x\notin\{-1,0,1\}$, then 
$\abar(x)=\abs{x}/(2\sigma_a^2)$. 
\label{ex-abar-1}
\end{example}   

Define the constant 
\be
\beta=\sum_{x\in\Zb} q(0,x)\abar(x).
\label{def-beta-2}
\ee
This constant accounts for the difference in the limits of the Green's functions
for transitions $q$ and $\qbar$.

\begin{lemma} Let $x\in\Rb$ and 
  $x_n\in\bZ$ be  
 such that 
$x_n- n^{1/2}x$ stays bounded. 
Then
\be
\lim_{n\to\infty} 
{n}^{-1/2} {G}_n\bigl(x_n,0\bigr)
=\frac{1}{2\beta\s_1^2} \int_0^{2\s_1^2} 
\frac1{\sqrt{2\pi v}}
\exp\Bigl\{-\frac{x^2}{2v}\Bigr\}  \,d v. 
\label{G-lim1}
\ee
 \end{lemma}

\begin{proof}
 Given limit \eqref{G8}, it suffices to prove 
\be \sup_{z\in\Z} \, \Bigl\lvert \frac{\beta}{\sqrt{n}} {G}_n(z,0)
\,-\, \frac1{\sqrt{n}} \bar{G}_n(z,0) \Bigr\rvert \to 0.  \label{raptemp9}\ee
First we prove the case   $z=0$. 

By \eqref{G8}, 
\be
\lim_{n\to\infty} 
{n}^{-1/2} \bar{G}_n(0,0)
=\frac1{\sqrt{\pi\s_1^2}}\,.
\label{barG-lim2}
\ee
We need to show 
\be
\lim_{n\to\infty} 
{n}^{-1/2} {G}_n(0,0)
=\frac1{\beta\sqrt{ \pi\s_1^2}}.  
\label{barG-lim2.5}
\ee
Using \eqref{abar-eq}, $\abar(0)=0$,
 and $\qbar(x,y)=q(x,y)$ for $x\neq 0$,   \begin{align*}
\sum_{x\in\Zb} q^m(0,x)\abar(x) &=\sum_{x\neq 0} q^m(0,x)\abar(x) =
\sum_{x\neq 0, y\in\Zb}
 q^m(0,x) \qbar(x,y)\abar(y)\\
& =\sum_{x\neq 0, y\in\Zb}
 q^m(0,x) q(x,y)\abar(y)  \\
&=\sum_{y\in\Zb} q^{m+1}(0,y)\abar(y) 
-q^m(0,0) \sum_{y\in\Zb} q(0,y)\abar(y).
\end{align*}
Constant  $\beta$ appears  in the last term.  
  Sum over $m=0,1,\dotsc, n-1$ to get 
\[
\bigl(1+q(0,0)+\dotsm +q^{n-1}(0,0)\bigr)\beta
= \sum_{x\in\Zb} q^n(0,x)\abar(x)
\]
and write this in the form 
\[
{n}^{-1/2} \beta {G}_{n-1}(0,0) = {n}^{-1/2} E_0\bigl[\abar(Y_n)\bigr].
\]
Recall that $Y_n=X_n-\Xtil_n$ where $X_n$ and $\Xtil_n$ are two
independent walks in the same environment.  By the quenched CLT for
space-time RWRE,  
  ${n}^{-1/2}Y_n\limd \cN(0, 2\s_1^2)$. Marginally  $X_n$ and $\Xtil_n$ are 
 i.i.d.\ walks with bounded steps, hence there is 
enough uniform integrability to conclude that
\[ {n}^{-1/2} E_0\abs{Y_n} \to 2\sqrt{\s_1^2/\pi}. \] 
By \eqref{abar-lim} and some estimation (exercise), 
\be 
{n}^{-1/2} E_0\bigl[\abar(Y_n)\bigr] \to \frac1{\sqrt{\s_1^2\pi}}\,.
\label{raptemp11}\ee
This proves \eqref{barG-lim2.5} and thereby limit \eqref{raptemp9} for $z=0$. 

To get the full statement in  \eqref{raptemp9},
for $k\geq 1$ and $z\ne 0$  let
\[
f^k(z,0)=\ind_{\{z\neq 0\}} \sum_{z_1\neq 0,\dotsc,z_{k-1}\neq 0} 
q(z,z_1)q(z_1,z_2)\dotsm q(z_{k-1},0) 
\]
denote the probability that the first visit to the origin
 occurs at time $k$.
This quantity is the same for both  $q$ and $\bar{q}$ because these processes 
do not differ until the origin is visited. 
Choose $n_0$ so that 
\[\lvert \beta  {G}_{m}(0,0)\,-\,
 \bar{G}_{m}(0,0) \rvert\leq \e\sqrt{m}\quad  \text{for $m\ge n_0$.} \]
Then 
 \begin{align*}
& \sup_{z\ne 0} \Bigl\lvert \frac{\beta}{\sqrt{n}} {G}_n(z,0)
\,-\, \frac1{\sqrt{n}} \bar{G}_n(z,0) \Bigr\rvert \\
&\quad \leq \sup_{z\ne 0}\; \frac1{\sqrt{n}} \sum_{k=1}^n f^k(z,0) 
\bigl\lvert \beta  {G}_{n-k}(0,0)
\,-\,  \bar{G}_{n-k}(0,0) \bigr\rvert \\
&\quad\le \sup_{z\ne 0}\; \frac\e{\sqrt{n}} \sum_{k=1}^{n-n_0} f^k(z,0) \sqrt{n-k}
+ \frac{Cn_0^2}{\sqrt{n}}   \le \e + \frac{Cn_0^2}{\sqrt{n}}. 
 \end{align*}
Letting $n\to\infty$ completes the proof.  \end{proof}

Combining \eqref{G4} and \eqref{G-lim1} gives 
\be\begin{aligned} 
\lim_{n\to\infty} \E\bigl[ \Hbar_n(t,q)\Hbar_n(t,r) \bigr]&= 
\frac{\s_D^2}{2\beta\s_1^2} \int_0^{2\s_1^2}  \frac1{\sqrt{2\pi v}}
\exp\Bigl\{-\frac{x^2}{2v}\Bigr\}  \,d v \\[5pt]
&=\kappa  \Gamma_1((t,q),(t,r)),   
\end{aligned}\label{Hlim12}\ee
where we defined a the new constant 
\be
\kappa=\frac{\sigma_D^2}{\beta\s_1^2  }.
\label{def-kappa}
\ee
While we have not furnished all the details, let us consider proved that for a 
fixed $t$, the finite-dimensional distributions of $\Hbar(t,r)$ converge 
to the Gaussian process $H(t,r)$ with covariance  
 $\kappa  \Gamma_1((t,q),(t,r))$.

\medskip

{\bf Step 2.  Markov property for  time steps.  }

\medskip

This step is overly technical and so we only give a sketch of the idea behind it.  
Stopping  and restarting the walk $ X^{ \fl{ntb}+\fl{r\sqrt{n}\,} ,\, \fl{nt}}_\centerdot$ at level $\fl{ns}$
gives:
\begin{align*}
&E^\om\bigl( X^{ \fl{ntb}+\fl{r\sqrt{n}\,} ,\, \fl{nt}}_{\fl{nt}}\bigr)  -\fl{r\sqrt{n}\,} \\
&=\sum_{x\in\bZ} P^\om\bigl( X^{ \fl{ntb}+\fl{r\sqrt{n}\,} ,\, \fl{nt}}_{\fl{nt}-\fl{ns}}=\fl{nsb} + x\bigr)
E^\om\bigl( X^{ \fl{nsb}+x ,\, \fl{ns}}_{\fl{ns}}\bigr)  -\fl{r\sqrt{n}\,} \\
&=\sum_{x\in\bZ} P^\om\bigl( X^{ \fl{ntb}+\fl{r\sqrt{n}\,} ,\, \fl{nt}}_{\fl{nt}-\fl{ns}}=\fl{nsb} + x\bigr)
\bigl[  \,E^\om\bigl( X^{ \fl{nsb}+x ,\, \fl{ns}}_{\fl{ns}}\bigr) -x\,\bigr] \\
&\qquad + E^\om\bigl( X^{ \fl{ntb}+\fl{r\sqrt{n}\,} ,\, \fl{nt}}_{\fl{nt}-\fl{ns}}\bigr) - \fl{nsb} -\fl{r\sqrt{n}\,}.
\end{align*}
Change summation index to $u=x/\sqrt n$.  Then we have approximately the identity
\begin{align*}
&\Hbar_n(t,r) = \\
&\sum_{u\in n^{-1/2}\bZ}  
P^\om\biggl\{   \frac{ X^{ \fl{ntb}+\fl{r\sqrt{n}\,} ,\, \fl{nt}}_{\fl{nt}-\fl{ns}} - \fl{nsb}-r\sqrt n}{\sqrt n}
=u-r\biggr\}  \Hbar_n(s,u) \\
&\qquad\qquad  +\Hbar^*_n(t-s, r),
\end{align*}
where $\Hbar^*_n(t-s, r)$ is  the same   as $\Hbar_n(t-s, r)$ but with origin shifted
(approximately)  to 
$(nsb, ns)$.   
On the right-hand side, the processes $\Hbar_n(s,\,\cdot)$ and
 $\Hbar^*_n(t-s,\,\cdot)$ are independent
of each other because they depend on disjoint levels of environments:
$\Hbar_n(s,\,\cdot)$ uses $\bar\om_{1, \fl{ns}}$ and
 $\Hbar^*_n(t-s,\,\cdot)$ uses $\bar\om_{\fl{ns}+1, \,\fl{nt}}$.
 As $n\to\infty$ the probability coefficients of the sum converge to 
deterministic Gaussian probabilities by the quenched CLT for the RWRE.  
By the result for fixed $t$,   the right-hand side above converges
in distribution.

Taking the limits and supplying all the technicalities  leads to the equation 
\[  H(t,r)=\int_\R \varphi_{\s_1^2(t-s)}(u-r) H(s,u)\,du + H^*(t-s,r)  \]
where on the right, the processes  $H(s,\,\cdot)$ and $H^*(t-s,\,\cdot)$ are independent.  
From this equation one can verify that the finite-dimensional distributions 
of the process $H(t,r)$ are Gaussian with covariance  $\kappa  \Gamma_1((s,q),(t,r))$
as stated in Theorem \ref{qmeanthm}.  

This concludes the presentation of the random  average process limit.

\subsection*{References}
 The fluctuation results  for RAP presented here are from \cite{bala-rass-sepp}. 
 In addition to  \cite{ferr-font-rap}, RAP was later studied also in \cite{font-mede-vach}.
 The quenched CLT for 
 space-time RWRE has been proved several times with progressively better
 assumptions, see \cite{rass-sepp-05}.

\section{Asymmetric simple exclusion process}\label{asepch}
\setcounter{equation}{0}
The  asymmetric
simple exclusion process (ASEP) is a Markov process that
describes the motion of 
  particles on the one-dimensional 
integer lattice $\Zb$. Each particle executes 
a continuous-time nearest-neighbor random walk on $\bZ$
with jump rate $p$ to the right and $q$ to the left.  
Particles interact through the exclusion rule
which  means that 
 at most one particle is allowed at each site.  
Any attempt to jump onto an already occupied site is prevented from happening.
  The asymmetric case is $p\ne q$.
We  assume  $0\le q< p\le 1$ 
and   $p+q=1$.

For this  process we do not derive precise distributional limits for the current,
but only bounds that reveal the order of magnitude of the fluctuations. 
In contrast with the earlier results for linear flux, the magnitude of current fluctuations
is now $t^{1/3}$. 

The proofs of these bounds are based on couplings, and make heavy use of the 
notion of {\sl second class particle}.  
 
\subsection{Basic properties}

We run quickly through the fundamentals 
of  $(p,q)$-ASEP.

\medskip

{\bf Definition and graphical construction.}
    The state of the system at time $t$ is
a configuration $\eta(t)=(\eta_i(t))_{i\in\bZ}\in\{0,1\}^\bZ$
of zeroes and ones.  The value $\eta_i(t)=1$ means that 
 site $i$ is occupied by a particle at time $t$, while the
value  $\eta_i(t)=0$ means that 
 site $i$ is vacant at time $t$. 

 The motion of the particles is controlled by
independent Poisson processes ({\sl Poisson clocks})  
$\{N^{i\to i+1}, N^{i\to i-1}: i\in\bZ\}$ on $\bR_+$.  These Poisson processes 
are independent of the (possibly random) initial configuration $\eta(0)$. 
Each Poisson clock  $N^{i\to i+1}$ has rate $p$
 and each  $N^{i\to i-1}$ has rate $q$. If
$t$ is  a jump time for $N^{i\to i+1}$ and if
$(\eta_i(t-), \eta_{i+1}(t-))=(1,0)$ then at time
$t$ the particle from site $i$ moves to site $i+1$
and the new values are $(\eta_i(t), \eta_{i+1}(t))=(0,1)$.
Similarly if $t$ is  a jump time for $N^{i\to i-1}$ a 
particle is moved from $i$ to $i-1$ at time $t$,
provided the configuration at time $t-$ permits this
move.  If the jump prompted by a Poisson clock is not
permitted by the state of the system, this jump attempt
 is simply
ignored and the particles resume waiting for the next
prompt coming from the Poisson clocks.   

This construction of the process is known as the {\sl graphical
construction} or the {\sl Harris construction}. 
When the initial state is a fixed configuration $\eta$,  
$P^\eta$ denotes the distribution of the process. 
 
We write $\eta$, $\omega$, etc for elements of the state space
$\{0,1\}^\bZ$, but also for the entire process so that 
$\eta$-process
 stands for $\{\eta_i(t): i\in\bZ, 0\le t<\infty\}$. 
The configuration $\delta_i$ is the state that has a single
particle at position $i$ but otherwise the lattice is
vacant.  

\begin{remark}
 When infinite particle systems such as ASEP are constructed
with Poisson clocks, there is an issue of well-definedness that needs to be resolved.
Namely, if we ask whether site $x$ is occupied at time $t$, we need to look
backwards in time at all the possible sites from which a particle could have 
moved to $x$ by time $t$.  This might involve an infinite regression:
perhaps there is a sequence of times $t>t_1>t_2>t_3>\dotsm >0$ such that 
Poisson clock $N^{x-k\to x-k+1}$ signaled a jump attempt at time $t_k$.  Such a
sequence of jumps could in principle bring a particle to $x$ ``from infinity.''

However, it is easily seen that this happens only with probability zero.  
For any fixed $T<\infty$
there is a positive probability that both $N^{i\to i+1}$ and 
$N^{i+1\to i}$ are empty in $[0,T]$.  Consequently almost surely there
are infinitely many edges $(i,i+1)$ across which no jump attempts
are made before time $T$.  This way the construction can actually be
performed for finite portions of the lattice at a time. 

Similar issue arises with the possibility of simultaneous 
conflicting jump commands.  By excluding a zero-probability set
of realizations of $\{N^{i\to i\pm 1}\}$ we can assume that there
are no  simultaneous jump attempts. 
\end{remark}

\medskip

{\bf Invariant distributions.} 
A basic fact is that i.i.d.~Bernoulli distributions 
$\{\nu^\rho\}_{\rho\in[0,1]}$ are 
extremal invariant distributions for ASEP. 
For each density value $\rho\in[0,1]$, $\nu^\rho$ is the 
probability measure on $\{0,1\}^\bZ$ 
 under which the occupation variables 
$\{\eta_i\}$ are i.i.d.\ with common mean 
$\int \eta_i\,d\nu^\rho=\rho$.  When the process
$\eta$ is stationary with time-marginal $\nu^\rho$, we write
$P^\rho$ for the probability distribution of the entire
process.  
The {\sl stationary density-$\rho$ process} means the ASEP
$\eta$ that is stationary in time and has marginal
distribution $\eta(t)$ $\sim$ $\nu_\rho$ for all $t\in\bR_+$. 

\begin{remark}
A note about how one would check the invariance. 
 In general, the infinitesimal generator  $L$ 
of a Markov process is an operator 
 defined as the derivative of the semigroup:
\be
L\varphi(\eta)  =\lim_{t\searrow 0}  \frac{ E^\eta[f(\eta(t))] -f(\eta) }{t}.
\label{tasepL1}\ee
Above $E^\eta$ denotes expectation under $P^\eta$, the distribution of the process 
when the initial state is $\eta$.  
 The generator of ASEP  is 
\be \begin{aligned}  L\varphi(\eta) &=
p\sum_{i\in\bZ} \eta_i(1-\eta_{i+1})[\varphi(\eta^{i,i+1})-\varphi(\eta)]\\
&\qquad +  q\sum_{i\in\bZ} \eta_i(1-\eta_{i-1})[\varphi(\eta^{i,i-1})-\varphi(\eta)]
\end{aligned} \label{tasepL}\ee
that acts on    cylinder functions $\varphi$ on the state space 
 $\{0,1\}^\bZ$
and  $\eta^{i,j}=\eta-\delta_i+\delta_{j}$ is the configuration that results from
moving one particle from site $i$ to $j$.  
Equation \eqref{tasepL1} can be derived 
from the graphical construction 
with some estimation.  

Invariance of a probability distribution can be checked by
a generator computation.   For ASEP it is enough to check that 
\be  \int L\varphi \,d\mu=0 \label{tasepinv}\ee
for  cylinder functions $\varphi$ to conclude that $\mu$ is invariant.  
This can be used  to check that Bernoulli measures $\nu^\rho$ are invariant for ASEP. 
\label{asepinvrem}\end{remark}

\medskip

{\bf Basic coupling and second class particles.}
The {\sl basic coupling} of two exclusion processes
$\eta$ and $\om$ means that they obey a common set
of Poisson clocks $\{N^{i\to i+1},N^{i\to i-1}\}$. 
Suppose the two processes $\eta$ and $\eta^+$ satisfy 
$\eta^+(0)=\eta(0)+\delta_{Q(0)}$ at time zero, for 
some position $Q(0)\in\bZ$. This means that $\eta^+_i(0)=\eta_i(0)$
for all $i\ne Q(0)$, $\eta^+_{Q(0)}(0)=1$ and $\eta_{Q(0)}(0)=0$.
  Then throughout the evolution in the basic coupling
there is a single discrepancy between $\eta(t)$ and  $\eta^+(t)$
at some position $Q(t)$:  $\eta^+(t)=\eta(t)+\delta_{Q(t)}$.
From the perspective of $\eta(t)$, $Q(t)$ is called a 
second class particle. By the same token,
from the perspective of $\eta^+(t)$, $Q(t)$ is a 
second class {\sl anti}particle.
  In particular, we shall call the 
pair $(\eta, Q)$ {\sl a $(p,q)$-ASEP with a second class particle}. 

We write a boldface $\Pv$ for the probability measure
when more than one process are coupled together. 
In particular, $\Pv^\rho$ represents the
situation where the initial occupation 
variables  $\eta_i(0)=\eta^+_i(0)$
are i.i.d.~mean-$\rho$ Bernoulli for $i\ne 0$, and the
second class particle $Q$ starts at $Q(0)=0$.   

More generally, if two processes $\eta$ and $\om$ are in basic
coupling and $\om(0)\ge\eta(0)$ (by which we mean coordinatewise
ordering $\om_i(0)\ge\eta_i(0)$ for all $i$) then the ordering
$\om(t)\ge\eta(t)$ holds for all $0\le t<\infty$.  The effect of
the basic coupling is to give priority to the $\eta$ particles
over the $\om-\eta$ particles. Consequently we can think
of the $\om$-process as consisting of  first class particles 
(the $\eta$ particles) and second class particles 
(the $\om-\eta$ particles).

\medskip

{\bf Current.} 
For $x\in\bZ$ and $t>0$,
$J_x(t)$ stands for the net left-to-right
particle current across the straight-line space-time path 
from $(1/2,0)$ to $(x+1/2,t)$.  More precisely, 
$J_x(t)=J_x(t)^+-J_x(t)^-$ where 
$J_x(t)^+$ is the number of particles that lie
in $(-\infty,0]$ at time $0$ but lie in 
$[x+1,\infty)$ at time $t$, while 
$J_x(t)^-$ is the number of particles that lie
in $[1,\infty)$ at time $0$ and  in 
 $(-\infty,x]$ at time $t$. When more than one process
($\om$, $\eta$, etc)
is considered in a coupling, the currents of the processes
are denoted by $J^\om_x(t)$, $J^\eta_x(t)$, etc. 

\subsection{Results}

The average net rate at which
particles in the stationary  $(p,q)$-ASEP at density $\rho$
 move across a fixed edge $(i,i+1)$ is the {\sl flux}
\be
\flux(\rho)= E^\rho [J_0(t)]=(p-q)\rho(1-\rho).
\label{def:flux}\ee
This formula follows from the fact that this process $M(t)$ is a mean zero martingale:
\be\begin{aligned}
 M(t) &=J_0(t) -\int_0^t \Bigl( p \ind\{\eta_0(s)=1,\,\eta_1(s)=0\} \\
&\qquad\qquad\qquad\qquad  - q  \ind\{\eta_1(s)=1,\,\eta_0(s)=0\}\Bigr)  
 \,ds.  
\end{aligned}  \label{asepmg}\ee
For the more general currents
 \be
E^\rho[J_x(t)]=t\flux(\rho)-x\rho  \qquad (x\in\bZ, t\ge 0)
\label{eq:EJ}\ee
as can be seen by noting that particles that crossed the
edge $(0,1)$  either also crossed $(x,x+1)$ and 
contributed to $J_x(t)$ or did not.

The {\sl characteristic speed} at density $\rho$ is 
\be
V^\rho=\flux'(\rho)=(p-q)(1-2\rho).
\label{def:Vrho}\ee

The derivation of the fluctuation bounds for the current rests on several key identities which we
collect in the next theorem. 

\begin{theorem}  Let the second class particle   start at the origin: $Q(0)=0$. 
For any density $0<\rho< 1$, 
$z\in\bZ$ and $t>0$ we have these formulas.    
\be
\Var^\rho[J_z(t)]=\sum_{j\in\bZ} \abs{j-z} \Cov^\rho[\eta_j(t),\eta_0(0)],  
 \label{goal1}\ee
\be
\Cov^\rho[\eta_j(t),\,\eta_0(0)]
=\rho(1-\rho) \Pv^\rho\{Q(t)=j\}, 
\label{eq:CovQ}\ee
and 
\be
\Ev^\rho[Q(t)]=V^\rho t.
\label{eq:EQH'a}\ee 
\label{asepthm1} \end{theorem}  

Formulas \eqref{goal1} and \eqref{eq:CovQ} combine to give 
\be 
\Var^\rho[J_{z}(t)]
=\rho(1-\rho)\Ev^\rho\lvert Q(t)-z\rvert.\label{VarJQ}\ee
In particular, for the current across the characteristic,  
\be 
\Var^\rho[J_{\fl{V^\rho t}}(t)]
=\rho(1-\rho)\Ev^\rho\lvert Q(t)-\fl{V^\rho t}\rvert.\label{VarJQa}\ee
Thus to get variance bounds on the current, we derive moment bounds on the 
second class particle. 
 
We  now state the main result, the moment bounds on the
second class particle.  It is of interest to  see how the bounds depend
on the bias $\bias=p-q$ so we include that in the estimates.

\begin{theorem}  There exist constants $0<c_0,C<\infty$ such 
that, for all $0<\bias<1/2$, $0<\rho<1$, $1\le m<3$, and $t\ge c_0\bias^{-4}$, 
\be \frac1{C}\bias^{2m/3}t^{2m/3} \le 
 \Ev^\rho  \bigl[\,\abs{Q(t)-V^\rho t}^m \,\bigr]\le \frac{C}{3-m}\bias^{2m/3}t^{2m/3}.   
\label{Qmom}\ee
For the upper bound the constants are fixed for all values of the parameters. For the 
lower bound both constants $c_0,C$ depend on the density $\rho$. 
\label{Qmomthm}\end{theorem} 
As a corollary for $m=1$, we obtain the bounds for the variance of the current 
seen by an observer traveling at the 
characteristic speed $V^\rho$: for  
 $t\ge c_0(\rho)\bias^{-4}$, 
\be
C_1(\rho)\bias^{1/3}t^{2/3} \le \Var^\rho[J_{\fl{V^\rho t}}(t)] 
\le C_2 \bias^{1/3} t^{2/3}.
\label{asepJvar}\ee

 It follows from the variance bound \eqref{asepJvar}   that for
 $v\ne V^\rho$ a Gaussian limit in the central limit scale holds: 
 \be    \frac{ J_{[tv]}(t) - E^\rho( J_{[tv]}(t))}{t^{1/2}} \limd\chi  \ee
 for a centered normal random variable $\chi$. 
To observe this,  take the case $v>V^\rho$.  Let $J^*$ be the current across
 the straight-line space-time path from $((v-V^\rho)t,0)$ to $(vt,t)$.  This current has
 variance of order $t^{2/3}$.  Then use 
\[  J^* =  J_{[tv]}(t) + \sum_{i=1}^{(v-V^\rho)t}  \eta_i(0).  \]

A distributional limit exists for the current for the case of the
stationary totally
asymmetric simple exclusion process (TASEP). We state the result here.
 In TASEP particles march only
to the right (say), and so $p=1$ and $q=0$. 

\begin{theorem} \cite{ferr-spoh-06} In stationary TASEP, the following distributional convergence holds:
\be 
\lim_{t\to\infty} P^\rho\biggl\{ \, \frac{J_{\fl{V^\rho t}}(t) - \rho^2t}
{\rho^{2/3}(1-\rho)^{2/3}t^{1/3}} \le x \biggr\} =F_{0}(x) 
\label{taseptwlim}\ee
\label{taseptwthm}\end{theorem} 
The distribution function $F_{0}$ above is defined in \cite{ferr-spoh-06}  as
$F_{0}(x) =(\partial/\partial x)(F_{\text{GUE}}(x)g(x,0))$  where $F_{\text{GUE}}$ is
the Tracy-Widom GUE distribution and $g$ a certain scaling function.  

Theorem \ref{taseptwthm}  will not be discussed further,
and we turn to proofs of   Theorem \ref{asepthm1} and  
  Theorem \ref{Qmomthm}.  In the next section we give partial proofs of the identities in
Theorem \ref{asepthm1}.  Section \ref{coupsec}  describes a coupling that we use
to control second class particles, and a random walk bound that comes in handy.
The last two sections of this chapter prove the upper and lower bounds of Theorem
\ref{Qmomthm}.   

\subsection{Proofs for the  identities}

Let $\om$ be a stationary exclusion process  with i.i.d.~Bernoulli($\rho$) distributed
occupations $\{\om_i(t)\}$ at any fixed time $t$. 

\begin{proof}[Proof of equation \eqref{goal1}]
This is partly a hand-waiving proof.  What is missing is justification for certain limits.

To approximate the infinite system with
finite systems, for each $N\in\bN$ 
let process $\om^N$ have initial configuration
\be
\om^N_i(0)=\om_i(0) \ind_{\{-N\le i\le N\}}.
\label{Ninitial}\ee 
We assume that all these processes are coupled through
common  Poisson clocks.
Let $J^N_z(t)$ denote the current in process $\om^N$.

Let $z(0)=0$, $z(t)=z$, and introduce the counting variables
\be
\pcount^N_+(t)=\sum_{n> z(t)}\om^N_n(t)\,,\quad
\pcount^N_-(t)=\sum_{n\le z(t)}\om^N_n(t).
\label{defpcount}\ee
Then the current can be expressed as 
\[
J_z^N(t)= \pcount^N_+(t)-\pcount^N_+(0)=\pcount^N_-(0)-\pcount^N_-(t),
\]
and its variance as 
\begin{align*}
\Var J_z^N(t) &=\Cov\bigl(\pcount^N_+(t)-\pcount^N_+(0), 
\,\pcount^N_-(0)-\pcount^N_-(t)\bigr)\\
&=\Cov\bigl(\pcount^N_+(t), \pcount^N_-(0))  
+ \Cov(\pcount^N_+(0),\pcount^N_-(t))\\
&\qquad -\Cov(\pcount^N_+(0),\pcount^N_-(0)) 
-\Cov(\pcount^N_+(t),\pcount^N_-(t)).  
\end{align*}

Independence of initial occupation variables  gives
\[\Cov(\pcount_+^N(0),\pcount_-^N(0))=0\] and the identity
above simplifies
to
\be
\begin{split}
&\Var J_z^N(t)=\Cov\bigl(\pcount^N_+(t), \pcount^N_-(0))  
+ \Cov(\pcount^N_+(0),\pcount^N_-(t))
\\[5pt]   &\qquad\qquad\qquad   
-\Cov(\pcount^N_+(t),\pcount^N_-(t)) \\[5pt]
 &=
\sum_{k\le 0,\,m>z}  \Cov[\om^N_m(t),\om^N_k(0)]\\
&\qquad + \sum_{k\le z,\,m>0}  \Cov[\om^N_k(t),\om^N_m(0)]
-\Cov(\pcount_+^N(t),\pcount_-^N(t)).
\end{split}
\label{simplified}\ee 
In the $N\to\infty$ limit variables  $\om^N_i(t)$ converge (a.s.\ and in $L^2$)
to the i.i.d.\ occupation variables  $\om_i(t)$ of the 
 stationary process.  
 It follows from the graphical construction that on a fixed time interval 
  covariances can be bounded exponentially, uniformly over $N$: for a fixed $0<t<\infty$, 
 \[  \abs{\,\Cov[\om^N_m(t),\om^N_k(s)]\,} \le Ce^{-c_1\abs{m-k}}
 \quad \text{for $s\in[0,t]$.} \]
   Hence in the limit the last covariance in \eqref{simplified} vanishes.  Furthermore, 
$J^N_z(t)\to J_z(t)$  similarly, so in the limit we get 
\begin{align}
\Var J_z(t)&= \sum_{k\le 0,\, m>z}  \Cov[\om_m(t),\om_k(0)]
+ \sum_{k\le z,\, m>0}  \Cov[\om_k(t),\om_m(0)] \nn 
\\
&=\sum_{n\in\bZ} \abs{n-z} \Cov[\om_n(t),\om_0(0)]. \nn
\end{align}
This proves equation \eqref{goal1}.  
\end{proof}

\begin{proof}[Proof of equation \eqref{eq:CovQ}]
This is a straight-forward calculation. 
\begin{align*}
&\Cov^\rho[\om_j(t),\om_0(0)]
=E^\rho[\om_j(t)\om_0(0)]-\rho^2\\[4pt]
&=\rho E^\rho[\om_j(t)\,\vert\,\om_0(0)=1]-\rho E^\rho[\om_j(t)] \\
&=\rho\Bigl( E^\rho[\om_j(t)\,\vert\,\om_0(0)=1]
-\rho E^\rho[\om_j(t)\,\vert\,\om_0(0)=1]\\
&\qquad\qquad\qquad\qquad - \; (1-\rho) E^\rho[\om_j(t)\,\vert\,\om_0(0)=0]\Bigr)\\
&=\rho(1-\rho)\Bigl( E^\rho[\om_j(t)\,\vert\,\om_0(0)=1]
- E^\rho[\om_j(t)\,\vert\,\om_0(0)=0]\Bigr)\\
&=\rho(1-\rho)\bigl(  \Ev^\rho[\om^+_j(t)] -  \Ev^\rho[\om_j(t)]\bigr)
=\rho(1-\rho)  \Ev^\rho[\om^+_j(t) -  \om_j(t)]\\
&=\rho(1-\rho) \Pv^\rho[ Q(t)=j].
\qedhere\end{align*}
\end{proof} 

   \begin{proof}[Proof of equation \eqref{eq:EQH'a}]  

 Let again  $\om^N$ be  the finite process with
initial condition \eqref{Ninitial}.  
 Let $I^N=\sum_i\om^N_i(t)$ be the number of particles in the
process $\om^N$. $I^N$ is a Binomial($2N+1,\rho$) random variable. 
For $0<\rho<1$
\begin{align}
&\frac{d}{d\rho} E[J_z^N(t)] =
\frac{d}{d\rho} \sum_{m=0}^{2N+1} \binom{2N+1}{m} 
\rho^m(1-\rho)^{2N+1-m} E[J_z^N(t)\vert I^N=m] \nn\\
&=
 \sum_{m=0}^{2N+1} P(I^N=m)
\Bigl(\frac{m}{\rho} -\frac{2N+1-m}{1-\rho}\Bigr) E[J_z^N(t)\vert I^N=m] \nn\\
&=\frac1{\rho(1-\rho)} E\Bigl[ J_z^N(t)  \bigl(I^N-(2N+1)\rho\bigr)\Bigr]\nn\\
&=\frac1{\rho(1-\rho)} \Cov\bigl[ \pcount^N_+(t)-\pcount^N_+(0)\,,\,
\pcount^N_-(0)+\pcount^N_+(0)\bigr]\nn\\
&=\frac1{\rho(1-\rho)}\Bigl( \Cov\bigl[ \pcount^N_+(t)\,,\,
\pcount^N_-(0)\bigr]+\Cov\bigl[ \pcount^N_+(t)-\pcount^N_+(0)\,,\,
\pcount^N_+(0)\bigr]\Bigr). 
 \label{diffJline4} 
\end{align}
The last equality used $\Cov[\pcount^N_+(0),\pcount^N_-(0)]=0$
that comes from the i.i.d.\ distribution of initial occupations. 
The first covariance on line \eqref{diffJline4} write directly
as 
\[
\Cov\bigl[ \pcount^N_+(t)\,,\,
\pcount^N_-(0)\bigr]
= 
\sum_{k\le 0,\,m>z}  \Cov[\om^N_m(t),\om^N_k(0)]. 
\]
The second 
 covariance on line \eqref{diffJline4} write as 
\begin{align*} 
&\Cov\bigl[ \pcount^N_+(t)-\pcount^N_+(0)\,,\,
\pcount^N_+(0)\bigr]
= \Cov\bigl[ \pcount^N_-(0)-\pcount^N_-(t)\,,\,
\pcount^N_+(0)\bigr]\\
&\qquad =\;-\Cov\bigl[ \pcount^N_-(t)\,,\,
\pcount^N_+(0)\bigr]
=\ - \sum_{k\le z,\, m>0}  \Cov[\om^N_k(t),\om^N_m(0)]. 
\end{align*}
Inserting these back on line  \eqref{diffJline4} gives
\begin{align*}
\frac{d}{d\rho} E[J_z^N(t)] \; &= \; 
\frac1{\rho(1-\rho)}\Bigl( \;
\sum_{k\le 0,\,m>z}  \Cov[\om^N_m(t),\om^N_k(0)] \\
& \qquad  \; - \; 
\sum_{k\le z,\, m>0}  \Cov[\om^N_k(t),\om^N_m(0)] \;
\Bigr).
\end{align*}
Compared to line \eqref{simplified} we have the difference
instead of the sum.  Integrate over the density $\rho$ and
take $N\to\infty$, as was taken in \eqref{simplified},    to obtain
\begin{align*}
E^\rho[J_z(t)]-E^\la[J_z(t)] &= \int_\la^\rho \frac1{\theta(1-\theta)}
\sum_{j\in\bZ} (j-z) \Cov^\theta[\om_j(t),\,\om_0(0)]\,d\theta \\
&= \int_\la^\rho 
(\Ev^\theta[Q(t)]-z)\,d\theta 
\end{align*}
for $0<\lambda<\rho<1$.
Couplings show the continuity of these expectations:
\be
E^\la[J_z(t)]\to E^\rho[J_z(t)]
\quad\text{and}\quad
\Ev^\lambda[Q(t)]\to \Ev^\rho[Q(t)]
\label{cont}\ee
 as $\lambda\to\rho$ in $(0,1)$. Thus the identity above  can be
differentiated in $\rho$.  With $z=0$ and via \eqref{def:flux}  
identity \eqref{eq:EQH'a}  follows.  
\end{proof}

\subsection{A coupling and a random walk bound}
\label{coupsec} 

As observed in \eqref{eq:EQH'a} the mean speed of the second class particle 
in a density-$\rho$ ASEP is $\flux'(\rho)$.  
Thus by the concavity of  $\flux$ 
 a defect travels on average slower  in a denser system (recall that 
we assume $p>q$ throughout).  However,
  the basic coupling does not respect this, except in the
totally asymmetric ($p=1, q=0$) case.  To see this,
consider two pairs of
processes $(\om^+,\om)$ and $(\eta^+,\eta)$ such that both
pairs have one discrepancy: $\om^+(t)=\om(t)+\delta_{Q^\om(t)}$
and $\eta^+(t)=\eta(t)+\delta_{Q^\eta(t)}$.  Assume that 
$\om(t)\ge \eta(t)$. 
In basic coupling  the jump from state
\[
\begin{bmatrix}  \om^+_i&\om^+_{i+1}\\
\om_i&\om_{i+1}\\ \eta^+_i&\eta^+_{i+1}\\ 
\eta_i&\eta_{i+1}\end{bmatrix} 
=\begin{bmatrix} 1&1\\0&1\\1&0\\0&0\end{bmatrix}
\quad\text{to state}\quad
\begin{bmatrix} 1&1\\1&0\\1&0\\0&0\end{bmatrix}
\]
happens at rate $q$ and results in $Q^\om=i+1>i=Q^\eta$.

In this section we construct  a different coupling that combines 
the basic coupling with auxiliary clocks for second
class particles.  The idea is to think of a single ``special''
second class particle as performing a random walk on the 
process of $\om-\eta$ second class particles.   This  coupling 
preserves the expected ordering of the special second class particles, 
 hence it can be regarded as a form 
of {\sl microscopic concavity}.

This theorem  summarizes the outcome. 

\begin{theorem} 
Assume given two  initial configurations
 $\{\zeta_i(0)\}$ and $\{\xi_i(0)\}$ and two
not necessarily distinct  positions 
$Q^\zeta(0)$ and $Q^\xi(0)$  on $\bZ$.
Suppose the coordinatewise ordering  $\zeta(0)\ge\xi(0)$ holds,
$Q^\zeta(0)\le Q^\xi(0)$, 
and $\zeta_i(0)=\xi_i(0)+1$ for $i\in\{Q^\zeta(0), Q^\xi(0)\}$.
Define the configuration $\zeta^-(0)=\zeta(0)-\delta_{Q^\zeta(0)}$. 

Then there exists a coupling of processes 
\[(\zeta^-(t), Q^\zeta(t), \xi(t), Q^\xi(t))_{t\ge 0}\] with initial 
state $(\zeta^-(0), Q^\zeta(0), \xi(0), Q^\xi(0))$ as described  in 
the previous paragraph, such that both pairs 
$(\zeta^-, Q^\zeta)$ and $(\xi, Q^\xi)$ are $(p,q)$-ASEP's
with a second class particle, and $Q^\zeta(t)\le Q^\xi(t)$
for all $t\ge 0$. 
\label{th:newcoupling}
\end{theorem} 

To begin the construction, 
put two exclusion processes
$\zeta$ and $\xi$ in basic coupling, obeying 
Poisson clocks $\{N^{i\to i\pm 1}\}$. 
They are  ordered so that $\zeta\ge\xi$. 
The $\zeta-\xi$ second class particles are labeled
in increasing order 
$\dotsm<X_{m-1}(t)<X_{m}(t)<X_{m+1}(t)<\dotsm$. 
We assume  there is at least
one such second class particle, but
beyond that we make no assumption about their
number.  Thus there is 
 some
finite or infinite subinterval $I\subseteq\bZ$ of indices 
such that 
  the positions of the  $\zeta-\xi$ second class particles
are given by $\{X_{m}(t): m\in I\}$. 

We introduce two dynamically evolving labels 
$\lbla(t)$, $\lblb(t)$ $\in I$ in such a manner that 
$X_{\lbla(t)}(t)$ is the position of a second class 
antiparticle
in the $\zeta$-process, 
$X_{\lblb(t)}(t)$ is the position of a second class particle
in the $\xi$-process, and the ordering 
\be
X_{\lbla(t)}(t)\le X_{\lblb(t)}(t)
\label{eq:2classorder} \ee
 is preserved by the 
dynamics. 

The labels $\lbla(t)$, $\lblb(t)$ are allowed to jump from
$m$ to $m\pm 1$ 
only when particle $X_{m\pm 1}$ is adjacent  
 to $X_m$. The labels  do not take
jump commands from the
Poisson clocks $\{N^{i\to i\pm 1}\}$
that govern $(\xi,\zeta)$.  Instead, the directed
edges $(i,i+1)$ and $(i,i-1)$ are given another
collection of independent Poisson clocks so that the
following 
jump rates are realized. 

(i) If $\lbla=\lblb$ and $X_{\lbla+1}=X_{\lbla}+1$ 
then
\[
 (\lbla,\lblb)\ \text{jumps to }\ \begin{cases}
 (\lbla,\lblb+1) &\text{with rate $p-q$}\\ 
(\lbla+1,\lblb+1)&\text{with rate $q$.}
\end{cases} \]

(ii) If $\lbla=\lblb$ and $X_{\lbla-1}=X_{\lbla}-1$ 
then
\[
 (\lbla,\lblb)\ \text{jumps to }\ \begin{cases}
 (\lbla-1,\lblb) &\text{with rate $p-q$}\\ 
(\lbla-1,\lblb-1)&\text{with rate $q$.}
\end{cases} \]

(iii) If $\lbla\ne\lblb$ then $\lbla$ and $\lblb$ jump
independently with these rates:
\begin{align*}
 \lbla\ \text{jumps to }\ &\begin{cases}
 \lbla+1 &\text{with rate $q$ if $X_{\lbla+1}=X_{\lbla}+1$}\\ 
 \lbla-1 &\text{with rate $p$ if $X_{\lbla-1}=X_{\lbla}-1$;}
\end{cases} \\
 \lblb\ \text{jumps to }\ &\begin{cases}
 \lblb+1 &\text{with rate $p$ if $X_{\lblb+1}=X_{\lblb}+1$}\\ 
 \lblb-1 &\text{with rate $q$ if $X_{\lblb-1}=X_{\lblb}-1$.}
\end{cases} 
\end{align*}

Let us emphasize that the pair process $(\xi,\zeta)$ is still
governed by the old clocks  $\{N^{i\to i\pm 1}\}$
in the basic coupling. The new clocks on edge $\{i,i+1\}$  that realize
rules (i)--(iii) are not observed except when sites
$\{i,i+1\}$ are both occupied by $X$-particles and 
at least one of $X_{\lbla}$ or $X_{\lblb}$ lies in
$\{i,i+1\}$.    

First note that if initially $\lbla(0)\le\lblb(0)$ then
 jumps (i)--(iii)  preserve the inequality
$\lbla(t)\le \lblb(t)$ which gives \eqref{eq:2classorder}.  
(Since the jumps in point (iii) happen
independently, there cannot be two simultaneous jumps. 
So it is not possible for $\lbla$ and $\lblb$ to
cross each other with a $(\lbla,\lblb)\to(\lbla+1,\lblb-1)$
move.)  

Define processes  $\zeta^-(t)=\zeta(t)-\delta_{X_{\lbla(t)}(t)}$
and  $\xi^+(t)=\xi(t)+\delta_{X_{\lblb(t)}(t)}$.
In other words, to produce $\zeta^-$ remove particle 
$X_{\lbla}$ from $\zeta$, and  to produce $\xi^+$ add particle 
$X_{\lblb}$ to $\xi$.  The second key point is that,
even though these new processes are no longer defined
by the standard graphical construction, distributionwise
 they are still  ASEP's with second class particles. 
We argue this point for  $(\zeta^-, X_{\lbla})$ and leave the argument 
for $(\xi, X_{\lblb})$ to the reader. 

\begin{lemma}
The pair $(\zeta^-, X_{\lbla})$ 
is a $(p,q)$-ASEP with a second class particle. 
\label{lm:zeta-asep}\end{lemma}

\begin{proof} We check that the jump rates for the 
process  $(\zeta^-, X_{\lbla})$, produced by the combined
effect of the basic coupling with 
 clocks  $\{N^{i\to i\pm 1}\}$ and the new clocks, 
are  the same jump rates that result from defining
an (ASEP, second class particle) pair in terms of the
graphical construction.  

To have notation for the
possible jumps, let $0$ denote an empty site, $1$ 
a $\zeta^-$-particle, and $2$ particle $X_{\lbla}$. 
Consider a fixed pair $(i,i+1)$ of sites and write
$xy$ with $x,y\in\{0,1,2,\}$ for the contents of sites
$(i,i+1)$ before and after the jump.  Then here 
are the possible moves across
the edge $\{i,i+1\}$, and the rates that these moves
would have in the basic coupling. 
\begin{align*}
\text{Type 1} \qquad &10\longrightarrow 01 \ \text{ with rate $p$}\\
                  &01\longrightarrow 10 \ \text{ with rate $q$}\\[4pt]
\text{Type 2} \qquad &20\longrightarrow 02 \ \text{ with rate $p$}\\
                  &02\longrightarrow 20 \ \text{ with rate $q$}\\[4pt]
\text{Type 3} \qquad &12\longrightarrow 21 \ \text{ with rate $p$}\\
                  &21\longrightarrow 12 \ \text{ with rate $q$}\\
\end{align*}
Our task is to check that the construction of 
$(\zeta^-,X_{\lbla})$ actually
realizes these rates.  

Jumps of types 1 and 2 are prompted by
the clocks  $\{N^{i\to i\pm 1}\}$
of the graphical construction of $(\xi,\zeta)$, and hence have the correct
rates listed above.  

Jumps of type 3 occur in two distinct
ways. 

(Type 3.1) First there can be a $\xi$-particle next to $X_{\lbla}$, and then
the rates shown above are again realized by 
the clocks  $\{N^{i\to i\pm 1}\}$ because in
the basic coupling the $\xi$-particles have priority 
over the $X$-particles.  

(Type 3.2) The other 
alternative is 
that both sites $\{i,i+1\}$ are 
 occupied by $X$-particles  and one of them is
$X_{\lbla}$. 
The clocks  $\{N^{i\to i\pm 1}\}$
cannot interchange the $X$-particles across the 
 edge $\{i,i+1\}$ because in the $(\xi,\zeta)$-graphical construction
these are lower priority $\zeta$-particles that do not jump on
top of each other.  The otherwise missing jumps are now
supplied by  the ``new'' clocks that govern the jumps 
described in rules (i)--(iii). 

Combining (i)--(iii)  we can read that
if $X_{\lbla}=i+1$ and  $X_{\lbla-1}=i$, 
then $\lbla$ jumps to $\lbla-1$ with rate $p$. 
This is the first case of type 3 jumps above,
 corresponding to a $\zeta^-$-particle moving from
$i$ to $i+1$ with rate $p$, and the second class particle
$X_{\lbla}$ yielding. 
On the other hand,  if  $X_{\lbla}=i$ and  $X_{\lbla+1}=i+1$ then 
$\lbla$ jumps to $\lbla+1$ with rate $q$. This is
the second case in type 3, 
corresponding to a $\zeta^-$-particle moving from
$i+1$ to $i$ with rate $q$ and exchanging places with
 the second class particle
$X_{\lbla}$. 

We have  verified that the process $(\zeta^-, X_{\lbla})$ 
operates with the correct rates. 

To argue from the rates to the correct distribution of the 
process, we can make use of the process  $(\zeta^-, \zeta)$.
 The processes  $(\zeta^-, X_{\lbla})$ 
and  $(\zeta^-, \zeta)$ determine each other uniquely. 
The virtue of  $(\zeta^-, \zeta)$ is that it has a compact
state space and only nearest-neighbor jumps
with bounded rates.  Hence by the basic  theory
of semigroups and generators of particle systems as developed
in \cite{ligg-85}, 
given the initial configuration,
the distribution of the process  is uniquely determined by the
action of the generator on local functions.  Thus
it suffices to check that individual jumps have the 
correct rates across each edge $\{i,i+1\}$. This  is 
exactly what we did above in the language
of  $(\zeta^-, X_{\lbla})$.
\end{proof}

Similar argument shows that $(\xi, X_{\lblb})$ 
is a $(p,q)$-ASEP with a second class particle.  
To prove 
 Theorem \ref{th:newcoupling} take 
 $Q^\zeta=X_{\lbla}$ and $Q^\xi= X_{\lblb}$. 
This gives the coupling whose existence is claimed in the 
theorem.  

To conclude, let us observe  that the four
processes $(\xi, \xi^+,\zeta^-,\zeta)$ are not in
basic coupling.   For example,  the jump from state
\[
\begin{bmatrix}  \zeta_i&\zeta_{i+1}\\
\zeta^-_i&\zeta^-_{i+1}\\ \xi^+_i&\xi^+_{i+1}\\ 
\xi_i&\xi_{i+1}\end{bmatrix} 
=\begin{bmatrix} 1&1\\0&1\\1&0\\0&0\end{bmatrix}
\quad\text{to state}\quad
\begin{bmatrix} 1&1\\1&0\\0&1\\0&0\end{bmatrix}
\]
happens at rate $q$ (second case of rule (i)), while
in basic coupling this move is impossible.  

\bigskip

As the second point of this section we prove 
a random walk estimate.  
Let $Z(t)$ be a continuous-time nearest-neighbor 
random walk on  state-space
$S\subseteq\bZ$ that contains $\bZ_-=\{\dotsc,-2,-1,0\}$. Initially $Z(0)=0$.
 $Z$ attempts to jump from $x$ to $x+1$ with rate $p$ for $x\le -1$,
and from $x$ to $x-1$ with rate $q$ for $x\le 0$.
Assume $p>q=1-p$ and let $\theta=p-q$.  The rates on $S\smallsetminus\bZ_-$ need not
be specified. 
 
Whether jumps are permitted or not is determined by a fixed
environment expressed in terms of $\{0,1\}$-valued 
functions $\{u(x,t): x\in S,\, 0\le t<\infty\}$.  
A jump across edge $\{x-1,x\}$  in either direction is permitted
at time $t$ if $u(x,t)=1$, otherwise not.  
In other words, $u(x,t)$ is the
indicator of the event that edge $\{x-1,x\}$ is open at time $t$. 

\medskip

{\bf Assumption.}  Assume that for all $x\in S$ and  $T<\infty$,
$u(x,t)$ flips between $0$ and $1$ only finitely many times
during $0\le t\le T$.  Assume for convenience
 right-continuity: $u(x,t+)=u(x,t)$.

\begin{lemma} For all $t\ge 0$ and  $k\ge 0$,
\[
P\{ Z(t) \le -k\} \le e^{-2\theta k}.
\]
This bound holds for any fixed environment $\{u(x,t)\}$
subject to the assumption above. 
\label{lm:RW}\end{lemma}

\begin{proof} 
 Let $Y(t)$ be a walk that operates exactly
as $Z(t)$ on $\bZ_-$ but is restricted to remain in $\bZ_-$
by setting the  rate of jumping from $0$ to $1$ to zero. 
Give $Y(t)$ geometric initial distribution
\[ P\{Y(0)=-j\}= \pi(j)\equiv
 \Bigl(1-\frac{q}{p}\Bigr)\Bigl(\frac{q}{p}\Bigr)^j
\quad\text{for $j\ge 0$.}\]
The initial points satisfy $Y(0)\le Z(0)$ a.s.
Couple the walks through Poisson clocks so that the inequality
 $Y(t)\le Z(t)$ is preserved for all time $0\le t<\infty$.

Without the inhomogeneous  environment $Y(t)$ would be a stationary, 
reversible  birth and death
process.  We argue that even with the environment the time marginals
$Y(t)$ still have distribution $\pi$. 
This suffices for the conclusion, for then
\[
P\{Z(t)\le -k\}\le P\{Y(t)\le -k\}=(q/p)^k
=\exp\bigl(k\log\tfrac{1-\theta}{1+\theta}\bigr)
\le e^{-2\theta k}.
\]

To justify the claim about $Y(t)$, consider approximating processes
$Y^{(m)}(t)$, $m\in\bN$, with the same initial value
$Y^{(m)}(0)=Y(0)$.  $Y^{(m)}(t)$ evolves  so that the environments
$\{u(x,t)\}$ restrict its motion only on edges
$\{x-1,x\}$ for $-m+1\le x\le 0$.  In other words, for walk $Y^{(m)}(t)$
we set $u(x,t)\equiv 1$ for $x\le -m$ and $0\le t<\infty$. 
We couple the walks together so that $Y(t)=Y^{(m)}(t)$ 
until the first time one of the walks exits the interval
$\{-m+1,\dotsc,0\}$.  

Fixing $m$ for a moment,
let $0=s_0<s_1<s_2<s_3<\dotsc$ be a partition of the 
time axis so that $s_j\nearrow\infty$ and the environments
$\{u(x,t): -m<x\le 0\}$ are constant on each interval 
$t\in[s_i,s_{i+1})$.  Then on each time interval $[s_i,s_{i+1})$
$Y^{(m)}(t)$ is a continuous time Markov chain with time-homogeneous
jump rates
\begin{align*}
c(x,x+1)&=\begin{cases} pu(x+1,s_i), &-m\le x\le 0\\
                        p, &x\le -m-1 \end{cases}\\
\intertext{and}
c(x,x-1)&=\begin{cases} qu(x,s_i), &-m+1\le x\le 0\\
                        q, &x\le -m. \end{cases}
\end{align*}
One can check that detailed balance $\pi(x)c(x,x+1)=\pi(x+1)c(x+1,x)$
 holds for all $x\le -1$.  Thus $\pi$ is a reversible measure
for walk $Y^{(m)}(t)$ on each  time interval $[s_i,s_{i+1})$,
and we conclude that 
 $Y^{(m)}(t)$ has distribution $\pi$ for all $0\le t<\infty$. 

The coupling ensures that $Y^{(m)}(t)\to Y(t)$ almost surely
as $m\to\infty$,
and consequently also $Y(t)$  has distribution $\pi$ for all 
$0\le t<\infty$. 
\end{proof}

\subsection{Proof of the upper bound for second class  particle moments}  \label{asepubsec}

Abbreviate 
\be
\Psi(t)=\Ev^\rho\lvert Q(t)-V^\rho t\rvert.
\label{def:Psi}\ee

\begin{lemma}   Let $B\in(0,\infty)$.  
 Then there exists a 
numerical  constant
$C\in(0,\infty)$  and another constant $c_1(B)\in(0,\infty)$
    such that, for
all densities  $0< \rho< 1$, $u\ge 1$, $0<\bias<1/2$, and  
  $t\ge c_1(B)\bias^{-1}$,  
\be\begin{aligned} 
&\Pv^\rho\{Q(t)\ge V^\rho t+u\} \\[4pt]
&\quad \le \begin{cases} 
 C\bias^2 \Bigl(\,\dfrac{t^2}{u^4}\Psi(t) + \dfrac{t^2}{u^3} \Bigr)+e^{-u^2/Ct}, 
 &B\bias^{2/3} t^{2/3}\le  u\le 20t/3\\[9pt]
e^{-u/C}, 
&u\ge 20 t/3.  \end{cases}  \end{aligned}
\label{eq:UBge1/2}\ee
 \label{lm:UB1}
\end{lemma}

\begin{proof} 
First we get an easy case out of the way.  

\medskip

{\bf Case 1.}   $u\ge 5\rho{\bias} t$.

\medskip

This comes from an exponential Chebyshev argument.  Let $Z_t$ be  a nearest-neighbor random
walk   with rates $p=(1+\bias)/2$ to the right and 
$q=(1-\bias)/2 $ to the left.  For $\alpha\in(0,1]$,   using 
\[  \frac{e^\alpha+e^{-\alpha}}2\le 1+\alpha^2
\quad\text{and}\quad  \frac{e^\alpha-e^{-\alpha}}2 \le \alpha+\alpha^2, \]
we get 
\be\begin{aligned}
\Ev[e^{\alpha Z_t}]
&=\exp\Bigl( - t+t\frac{e^\alpha+e^{-\alpha}}2  
+\bias t\frac{e^\alpha-e^{-\alpha}}2 \,\Bigr)\\
&\le \exp\bigl( 2\alpha^2t + \alpha\bias t   \bigr).   
\end{aligned} \label{Ztrans}\ee
  We have the stochastic domination 
$Z_t\ge Q(t)$ because no matter what the environment next to $Q(t)$ is, 
$Q(t)$   has a weaker right drift than $Z_t$. 
Then, since $V^\rho={\bias}(1-2\rho)$ and   $2\rho{\bias} t\le 2u/5$, 
 \be\begin{split}
\Pv^\rho\{Q(t)\ge V^\rho t+ u\} 
&\le P\{Z_t\ge {\bias} t+ \tfrac35u\}\\
&\le 
\exp\bigl(-\tfrac35\alpha u+ 2\alpha^2t  \bigr)\\
& \le\begin{cases}  \exp(-\tfrac{9u^2}{200t}) &u\le 20t/3\\
              \exp(-3u/10)  &u> 20t/3. \end{cases} 
\end{split} 
\label{eq:aux15.5}\ee
In the last inequality above choose $\alpha =1\wedge\frac{3u}{20t}$. 
 Note that   $20t/3>5\rho{\bias} t$.
 
 It remains to consider this range of $u$: 
  
\medskip

{\bf Case 2.} $B\bias^{2/3}t^{2/3} \le u\le 5\rho{\bias} t$.

\medskip

  By an adjustment
of  the constant $C$ we can 
 assume   that  
$u$ is a positive integer.  
Fix a density
 $0<\rho<1$ and an auxiliary  density 
 \be
\lambda=\rho-\frac{u}{10{\bias} t}.
 \label{eq:lachoice}\ee
 Start with the basic coupling of three exclusion processes
$\om\ge\om^-\ge \eta$ with this initial set-up: 
 
(a) 
Initially $\{\om_i(0):i\ne 0\}$ are i.i.d.\ Bernoulli($\rho$) 
distributed 
and $\om_0(0)=1$.

(b) 
Initially $\om^-(0)=\om(0)-\delta_0$.   

(c) Initially
 variables $\{\eta_i(0):i\ne 0\}$ 
are i.i.d.\ Bernoulli($\lambda$) and $\eta_0(0)=0$. 
 The   coupling of the 
initial occupations is 
such that $\om_i(0)\ge\eta_i(0)$ for all $i\ne 0$.

Recall that  basic coupling meant that  these processes
obey  common  Poisson clocks. 

Let $Q(t)$ be the position of the single second class 
particle between $\om(t)$ and $\om^-(t)$, initially
at the origin.  Let $\{X_i(t):i\in\bZ\}$ be the positions
of the $\om-\eta$ second class particles, initially 
labeled so that 
\[
\dotsm<X_{-2}(0)<X_{-1}(0)<X_0(0)=0<X_1(0)<X_2(0)<\dotsm 
\]
These second class particles preserve their labels in the
dynamics and stay ordered.  Thus the $\om(t)$ configuration
consists of first class particles (the $\eta(t)$ process)
and second class particles (the $X_j(t)$'s). 
$\Pv$ denotes the probability measure under which all these
coupled processes live. Note that the marginal distribution
of $(\om,\om^-,Q)$ under $\Pv$ is exactly as it would
be under $\Pv^\rho$. 

For $x\in\bZ$,  
 $J^\om_x(t)$ is  the net current in the 
$\om$-process between space-time positions 
$(1/2,0)$ and $(x+1/2,t)$.  Similarly 
 $J^\eta_x(t)$  in the $\eta$-process, 
and  $J^{\om-\eta}_x(t)$ is  the net current of
second class particles.  
Current 
in the $\om$-process is a sum of  the first class particle
current and the second class particle current:  
\be
J^\om_x(t)= J^\eta_x(t)+J^{\om-\eta}_x(t).
\label{eq:currsum}\ee

 $Q(t)\in\{X_j(t)\}$ for all time
because  the basic coupling preserves the
  ordering
$\om^-(t)\ge\eta(t)$.   Define the label $\Qlb(t)$ by
$Q(t)=X_{\Qlb(t)}(t)$ with initial value $\Qlb(0)=0$. 
 
\begin{lemma} 
For all $t\ge 0$ and $k\ge 0$,
\[
\Pv\{\Qlb(t)\ge k\} \le e^{-2\theta k}.
\]
\label{lm:mQ}\end{lemma} 
\begin{proof}[Proof of Lemma \ref{lm:mQ}]
In the basic coupling 
 the label $\Qlb(t)$ evolves as follows.  When $X_{\Qlb-1}$ is
adjacent to $X_{\Qlb}$, $\Qlb$ jumps down by one at rate $p$. 
And when  $X_{\Qlb+1}$ is
adjacent to $X_{\Qlb}$, $\Qlb$ jumps up by one at rate $q$. 
When $X_{\Qlb}$ has no  $X$-particle in either neighboring site,
the label $\Qlb$ cannot jump. 
Thus the situation is like that in  Lemma \ref{lm:RW} 
(with a reversal of lattice directions) 
with  environment  given
by the adjacency of $X$-particles:  
$u(m,t)=\ind\{X_m(t)=X_{m-1}(t)+1\}$. 
However,  the basic coupling mixes together the
evolution of the environment and the walk $\Qlb$,
so the environment is not specified in advance
as required by  Lemma \ref{lm:RW}. 

We can get around this difficulty by imagining an alternative
but distributionally equivalent 
construction for the joint process $(\eta, \om^-,\om)$. 
Let $(\eta,\om)$ obey basic coupling with the given Poisson
clocks $\{N^{x\to x\pm 1}\}$ attached to 
directed edges $(x,x\pm 1)$.  Divide
the $\om-\eta$ particles further into
class II  consisting  of the 
particles   $\om^--\eta$  and 
class III that consists only of the single 
particle  $\om-\om^-=\delta_Q$.  Let class II have 
priority over class III.  
 Introduce another independent set of Poisson
clocks $\{\widetilde N^{x\to x\pm 1}\}$, also attached
to directed edges $(x,x\pm 1)$ of the space $\bZ$ where
particles move.  Let clocks $\{\widetilde N^{x\to x\pm 1}\}$  
 govern the exchanges between 
classes II and III.  In other words, for each edge
$\{x,x+1\}$  clocks
$\widetilde N^{x\to x+ 1}$ and $\widetilde N^{x+1\to x}$ 
are observed if 
sites $\{x,x+1\}$ are both occupied by $\om-\eta$
particles.  All other jumps are prompted by the 
original clocks. 

The rates for individual jumps are the same in this
alternative construction as in the earlier one where
all processes were together in basic coupling.  Thus
the same distribution for the process  $(\eta, \om^-,\om)$
 is created.  

To apply  Lemma \ref{lm:RW} perform the construction
in two steps. First construct the process  $(\eta,\om)$
for all time.  This determines the environment 
$u(m,t)=\ind\{X_m(t)=X_{m-1}(t)+1\}$.  Then run the 
dynamics of classes II and III in this environment.
Now Lemma \ref{lm:RW} gives the bound for $\Qlb$. 
\end{proof}

 Let
$u$   be a positive integer  and  
 \be
 k=\left\lfloor\frac{u^2}{20{\bias} t}\right\rfloor-3.
\label{eq:kchoice}\ee
By  assuming $t\ge C(B)\bias^{-1}$ we guarantee that
\[  u\ge 1  
    \quad\text{and}\quad  \frac{u^2}{40{\bias} t} \ge \frac{B^2\bias^{1/3} t^{1/3}}{40}
\ge 4. \]
Then 
\be
k\ge \frac{u^2}{40{\bias} t}\ge 4. 
\label{eq:kprop}\ee
We begin 
  a series of inequalities.
\begin{align}
&\Pv\{Q(t)\ge V^\rho t +u\} \nn\\
&\qquad \le \Pv\{\Qlb(t)\ge k\} 
+ \Pv\{  J^\om_{\fl{V^\rho t}+u}(t) \; - \; J^\eta_{\fl{V^\rho t}+u}(t)
>-k\}. 
\label{line2}\end{align}
To explain the inequality above,
if $Q(t)\ge V^\rho t +u$ and 
$\Qlb(t)<k$ then $X_k(t)>\fl{V^\rho t}+u$. This puts
the bound  \[J^{\om-\eta}_{\fl{V^\rho t}+u}(t)> -k\] on the
second class particle current, because 
at most particles  $X_1,\dotsc, X_{k-1}$ could have
made a negative contribution to this current.

Lemma \ref{lm:mQ} takes care of
the first probability on line \eqref{line2}.
We work on the second probability on line
\eqref{line2}.  

Here is a simple observation that will be used repeatedly.  
Process $\om$ can be coupled
with a stationary density-$\rho$ process $\om^{(\rho)}$ 
 so that the coupled pair $(\om,\om^{(\rho)})$ has 
at most 1 discrepancy.  In this coupling
\be
\lvert J^\om_{x}(t)-J^{\om^{(\rho)}}_{x}(t)\rvert\le 1. 
\label{eq:coupJ}\ee 
This way we can  use computations for 
stationary processes at the expense of small errors.  

 Recall
that $V^\rho=\flux'(\rho)$. Let
$c_1$ below be a constant that absorbs the errors
from using means of stationary processes  
 and from ignoring  integer parts. It satisfies 
$\abs{c_1}\le 3$.  
\begin{align}
&\Ev J^\om_{\fl{V^\rho t}+u}(t) - \Ev J^\eta_{\fl{V^\rho t}+u}(t)\\
&= 
t\flux(\rho)-(\flux'(\rho)t+u)\rho -t\flux(\lambda)+(\flux'(\rho)t+u)\lambda
+c_1\nn\\
&= -\tfrac12 t\flux''(\rho)(\rho-\lambda)^2-u(\rho-\lambda)+c_1 \nn\\
&=t\theta(\rho-\lambda)^2-u(\rho-\lambda)+c_1  \label{line1.8}\nn \\
&=t\theta(\rho-\lambda)^2-u(\rho-\lambda)+c_1+k-k \nn\\
&\le  \frac{u^2}{100t\bias}- \frac{u^2}{10t\bias}+\frac{u^2}{20t\bias}-k \nn\\
&=-\,\frac{u^2}{25t\bias}-k.
\end{align}
 The $-3$ in the definition \eqref{eq:kchoice}
of  $k$ absorbed $c_1$ above. 

Let   $\overline  X=X-EX$ denote a centered random
variable.  Continuing with the  second probability
 from line \eqref{line2}:
\begin{align}
&\Pv\{  J^\om_{\fl{V^\rho t}+u}(t) \; - \; J^\eta_{\fl{V^\rho t}+u}(t)
>-k\}\nn\\
&\le \Pv\Bigl\{  \overline J^\om_{\fl{V^\rho t}+u}(t) \; - \; 
\overline J^\eta_{\fl{V^\rho t}+u}(t) \ge  
\frac{u^2}{25t\theta} \Bigr\}\nn\\
&\le \frac{C\theta^2t^2}{u^4} 
\Vv\Bigl\{ J^\om_{\fl{V^\rho t}+u}(t) \; - 
\; J^\eta_{\fl{V^\rho t}+u}(t)\Bigr\}
\nn\\
&\le \frac{C\theta^2t^2}{u^4} 
\Bigl(\Vv\bigl\{ J^\om_{\fl{V^\rho t}+u}(t)\bigr\} 
\; + \; \Vv\bigl\{J^\eta_{\fl{V^\rho t}+u}(t)\bigr\}\,\Bigr).\label{line6}
\end{align} 
  $C$ is a numerical constant that can change from line to
line but  is independent of all the parameters. 

We develop bounds on  the variances above, 
first for $J^\om$. Pass to the stationary density-$\rho$ 
 process via \eqref{eq:coupJ} and apply 
\eqref{VarJQ}: 
\begin{align}
\Vv\bigl\{ J^\om_{\fl{V^\rho t}+u}(t)\bigr\}
&\le 2\Var^\rho\bigl\{ J_{\fl{V^\rho t}+u}(t)\bigr\}+2 \nn\\
&=
2\rho(1-\rho) 
\Ev\bigl\lvert Q(t)-\fl{V^\rho t}-u\bigr\rvert + 2\nn\\
&\le
\Ev\lvert Q(t)-V^\rho t\,\rvert + u +3\nn \\
&\le  \Psi(t) + 4u.
\label{eq:aux7}
\end{align}  
 
 Let $\Var^\lambda$ denote variance in the stationary 
density-$\lambda$  process
and let  $Q^\eta(t)$ denote the position
of a second class particle added to a process 
$\eta$. 
 \begin{align}
&\Vv\bigl\{ J^\eta_{\fl{V^\rho t}+u}(t)\bigr\}
\le 2\Var^\lambda\bigl\{ J_{\fl{V^\rho t}+u}(t)\bigr\}+2 \nn\\
&\qquad \le 
\Ev^\lambda\bigl\lvert Q^\eta(t)-\fl{V^\rho t}-u\bigr\rvert + 2\nn\\
&\qquad \le
\Ev^\lambda\lvert Q^\eta(t)-V^\rho t\,\rvert + 4u \nn\\
\intertext{Introduce process 
$(\zeta^-(t), Q^\zeta(t), \eta(t), Q^\eta(t))_{t\ge 0}$
coupled as  in Theorem \ref{th:newcoupling}, where 
$\zeta$ starts with Bernoulli($\rho$) occupations away
from the origin and  initially $Q^\zeta(0)=Q^\eta(0)=0$. 
Below apply the triangle inequality and use inequality
   $Q^\zeta(t)\le Q^\eta(t)$ from   Theorem \ref{th:newcoupling}. 
Thus continuing from above: }
&\qquad =
\Ev\lvert Q^\eta(t)-Q^\zeta(t)+Q^\zeta(t)-V^\rho t\,\rvert + 4u\nn\\
&\qquad \le
\Ev\bigl\{ Q^\eta(t)-Q^\zeta(t)\bigr\} +
\Ev\lvert Q^\zeta(t)-V^\rho t\,\rvert + 4u\nn\\
&\qquad= 
V^\lambda t-V^\rho t + \Psi(t) + 4u\nn\\
&\qquad = 2\theta t(\rho-\lambda) + \Psi(t) + 4u\nn\\
&\qquad = \Psi(t) + 5u. 
\label{eq:aux9}
\end{align}  
  Marginally the process
$(\zeta, Q^\zeta)$ is the same as the process
$(\om, Q)$ in the coupling of this section, 
hence the appearance of $\Psi(t)$ above. Then 
we used \eqref{eq:EQH'a} for the expectations of the second class
particles and  the  choice 
\eqref{eq:lachoice} of $\lambda$. 

Insert bounds \eqref{eq:aux7} and \eqref{eq:aux9}
into \eqref{line6} to get
\begin{align}
\Pv\{  J^\om_{\fl{V^\rho t}+u}(t) \; - \; J^\eta_{\fl{V^\rho t}+u}(t)
>-k\}
\le C\theta^2
\Bigl(\,\frac{t^2}{u^4}\Psi(t) + \frac{t^2}{u^3}\,\Bigr).
\label{eq:aux13} \end{align}

Insert \eqref{eq:kprop}  and \eqref{eq:aux13} into line \eqref{line2} 
to get 
\be 
\Pv\{Q(t)\le V^\rho t-u\} \le 
 C\bias^2 \Bigl(\,\frac{t^2}{u^4}\Psi(t) + \frac{t^2}{u^3}
\Bigr) + e^{-u^2/20t} 
\label{eq:aux14}\ee
and we have verified  \eqref{eq:UBge1/2}
for {\bf Case 2}.  

\medskip

 Combining   \eqref{eq:aux14}  and  
 \eqref{eq:aux15.5}  
gives the conclusion  of Lemma \ref{lm:UB1}. 
\end{proof} 

Next we extend the bound to both tails. 

\begin{lemma}   Let $B\in(0,\infty)$.  
 Then there exists a 
numerical  constant
$C\in(0,\infty)$  and another constant $c_0(B)\in(0,\infty)$
    such that, for
all densities  $0< \rho< 1$ and  
  $t\ge c_0(B)\bias^{-1}$,  
\be\begin{aligned} 
&\Pv^\rho\{\,\abs{Q(t)- V^\rho t}\ge u\} \\[4pt]
& \le \begin{cases} 
 C\bias^2 \Bigl(\,\dfrac{t^2}{u^4}\Psi(t) + \dfrac{t^2}{u^3} \Bigr)+2e^{-u^2/Ct}, 
 &B\bias^{2/3} t^{2/3}\le  u\le 20t/3\\[9pt]
2e^{-u/C}, 
&u\ge 20 t/3.  \end{cases}  \end{aligned}
\label{eq:UB2}\ee
 \label{lm:UB2}
\end{lemma}

\begin{proof}  The corresponding lower tail bound is obtained
from \eqref{eq:UBge1/2} by a particle-hole interchange followed by
a reflection of the lattice.  For details we refer to Lemma 5.3
in \cite{bala-sepp-08ALEA}.  
\end{proof}

 \begin{proof}[Proof of the upper bound of Theorem \ref{Qmomthm}]
We integrate \eqref{eq:UB2} to get the  bound \eqref{Qmom}  on the moments
of the second class particle.  First for $m=1$.  
\begin{align*}
\Psi(t)&= \int_0^\infty 
\Pv^\rho\{\,\lvert Q(t) -V^\rho t\rvert \ge u\}\,du\\
&\le 
B\bias^{2/3}t^{2/3} 
+ C\bias^2  \int_{B\bias^{2/3}t^{2/3} }^\infty
\Bigl(\,\frac{t^2}{u^4}\Psi(t) + \frac{t^2}{u^3}\,\Bigr)\,du\\
& \qquad \qquad +  2\int_{B\bias^{2/3}t^{2/3} }^\infty  e^{-u^2/Ct} \,du  + 
2 \int_{20t/3}^\infty  e^{-u/C}\,du 
  \\
&\le \frac{C}{3B^{3}}\Psi(t) + 
\Bigl(B+\frac{C}{2B^2} \Bigr) \bias^{2/3}t^{2/3}  \\
&\qquad \qquad +\frac{C_1(B)t^{1/3}}{\bias^{2/3}} e^{-\bias^{4/3}t^{1/3}/C_1(B)}  
+ 2C e^{-t/C}.
\end{align*}
$C_1(B)$ is a new constant that depends on $B$.
Set $B=C^{1/3}$  to turn the above inequality into 
\begin{align*}
\Psi(t)&\le \frac{9C^{1/3}}4 \bias^{2/3}t^{2/3}  
+ \frac{C_1t^{1/3}}{\bias^{2/3}} \exp\Bigl(\frac{-\bias^{4/3}t^{1/3}}{C_1} \Bigr) 
+ 2C e^{-t/C}.
\end{align*}
The second term on the right above forces us to restrict  $t$ further.  
We can fix a constant $c_0$ large enough so that,
 for a new constant $C$, 
\be
\Psi(t)\le C \bias^{2/3}t^{2/3}   
\quad\text{provided $t\ge c_0\bias^{-4} $.}  \label{eq:Psibd2}\ee
Restrict to $t$ that satisfy this requirement and  substitute this bound on $\Psi(t)$ 
into \eqref{eq:UB2}. 
Then upon using $u\ge B\bias^{2/3}t^{2/3} $  and redefining
$C$ once more, we have 
   for $B\bias^{2/3}t^{2/3} \le  u\le 20t/3$:
\be \begin{aligned} 
\Pv^\rho\{\,\lvert Q(t) -V^\rho t\rvert \ge u\} \le 
 C\frac{\bias^{2} t^2}{u^3} + 2e^{-u^2/Ct}. 
 \end{aligned}\label{eq:aux18} \ee
  
Now take $1<m<3$ 
and use   \eqref{eq:aux18} together with the second case of \eqref{eq:UB2}:
\begin{align}
&\Ev^\rho\lvert Q(t)-V^\rho t\rvert^m 
=m \int_0^\infty \Pv^\rho\{\,\lvert Q(t) -V^\rho t\rvert \ge u\}
u^{m-1}\,du\nn\\
&\qquad \le 
B^m \bias^{2m/3} t^{2m/3} \;+\; Cm\bias^{2} t^2 \int_{B\bias^{2/3}t^{2/3} }^\infty
{u^{m-4}}\,du \nn\\
&\qquad \;+\;2m\int_{B\bias^{2/3}t^{2/3} }^\infty  e^{-u^2/Ct} u^{m-1}\,du
\;+\;2m\int_{20t/3}^\infty e^{-u/C} u^{m-1}\,du.\nn
\end{align}
Performing and approximating the integrals  gives  
\[
\Ev^\rho\lvert Q(t)-V^\rho t\rvert^m  \le 
 \frac{C}{3-m} \bias^{2m/3} t^{2m/3}    \]
 provided $t\ge c_0\bias^{-4}$ for a large enough constant $c_0$.    
 \end{proof}  
 
 \subsection{Proof of the lower bound for second class  particle moments}  \label{aseplbsec}

By Jensen's inequality it suffices to prove the lower bound for $m=1$. 
 Let  $C_{UB}$ denote the constant in the upper bound statement that we just proved.  
 We can also assume $c_0\ge 1$.  
 Fix a constant $b>0$ and  set  
\begin{align*}
a_1=2C_{UB}+1  
\quad\text{and}\quad 
a_2=8 + \sqrt{32 b}+8 \sqrt{C_{UB}}.  \end{align*}
Increase  $b$ if necessary so that 
\be  b^2-2a_2\ge 1.  \label{ba_2}\ee

Fix a density $\rho\in(0,1)$ and define an auxiliary density 
$\lambda=\rho-bt^{-1/3}\bias^{-1/3}$. 
 Define  positive
integers 
\be
u=\fl{a_1t^{2/3}\bias^{2/3}} \quad\text{and}\quad 
n=\fl{V^\lambda t}-\fl{V^\rho t}+u.
\label{def:n}\ee
By taking $c_0$ large enough in the statement of Theorem \ref{Qmomthm} 
we can ensure that $\lambda\in(\rho/2,\rho)$ and $u\in\N$.

Construct a basic coupling of
three processes $\eta\le\eta^+\le \zeta$ with the following initial state:

\medskip

(a) Initially $\eta$ has
i.i.d.~Bernoul\-li($\lambda$) occupations 
$\{\eta_i(0):i\ne -n\}$ and $\eta_{-n}(0)=0$. 

(b)  Initially $\eta^+(0)=\eta(0)+\delta_{-n}$. 
$Q^{(-n)}(t)$ is the location of the unique
discrepancy between $\eta(t)$ and $\eta^+(t)$. 

(c)  Initially $\zeta$ has independent 
occupation variables, coupled with $\eta(0)$
as follows: 

(c.1) $\zeta_i(0)=\eta_i(0)$ for $-n<i\le 0$.

 (c.2) $\zeta_{-n}(0)=1$.

 (c.3)   For  $i< -n$ and $i>0$ variables $\zeta_i(0)$ 
are i.i.d.~Bernoulli($\rho$) and 
$\zeta_i(0)\ge\eta_i(0)$. 

 Thus the initial density of $\zeta$ is piecewise
constant:  on the segment 
$\{-n+1,\dotsc,0\}$ $\zeta(0)$ is i.i.d.\ with
density $\lambda$, at site $-n$ $\zeta(0)$ has density $1$,
and elsewhere on $\bZ$ $\zeta(0)$ is i.i.d.\ with
density $\rho$.   The reason for   the gap in the $\zeta-\eta$ second class
particles across $(-n,0]$ is to get an upper bound on the second-class particle
current that is not too large for subsequent arguments (\eqref{eq:aux37} below).  

\medskip

Label the $\zeta-\eta$ second class particles
as $\{Y_m(t):m\in\bZ\}$ 
so that initially 
\[
\dotsm< Y_{-1}(0)<Y_{0}(0)= -n=Q^{(-n)}(0)<0 < Y_{1}(0)<Y_{2}(0)<\dotsm
\]
Let again $\Qlb(t)$ be the label such that 
 $Q^{(-n)}(t)=Y_{\Qlb(t)}(t)$. 
Initially $\Qlb(0)=0$. 
The inclusion 
$Q^{(-n)}(t)\in\{Y_m(t)\}$ 
  persists for all
time because the basic coupling preserves the 
ordering $\zeta(t)\ge\eta^+(t)$. 
Through the basic coupling
 $\Qlb$ jumps to the left
with rate $q$ and to the right with rate $p$, but only
when there is a $Y$-particle  adjacent to $Y_{\Qlb}$.
As in the proof of Lemma \ref{lm:mQ} we can apply
Lemma \ref{lm:RW} to prove this statement: 
\be 
\Pv\{\Qlb(t)\le -k\} \le e^{-2\theta k}.
\label{mQ2} \ee

By the upper bound already proved and by the choice of $a_1$,  
\be\begin{aligned}
&\Pv\{ Q^{(-n)}(t)\ge \fl{V^\rho t}\}
=\Pv\{ Q^{(-n)}(t)\ge -n+\fl{V^\lambda t} +u\} \\
&\qquad \le  u^{-1}  \Ev\abs{ Q^{(-n)}(t)- n-\fl{V^\lambda t} } 
\le \frac{C_{UB}t^{2/3}\bias^{2/3}}{\fl{a_1t^{2/3}\bias^{2/3}}}\\
 &\qquad \le \tfrac12. 
\end{aligned}\label{lb:ucond1}\ee
 
 For the complementary event we get a lower bound:
\be\begin{split}
\tfrac12 &\le \Pv\{ Q^{(-n)}(t)\le  \fl{V^\rho t}\}\\
&\le \Pv\{m(t)\le -k\}
+\Pv\{ J^\zeta_{\fl{V^\rho t}}(t)-J^\eta_{\fl{V^\rho t}}(t) \le k\}. 
\end{split}\label{eq:aux36}\ee
The reasoning behind the second inequality above is this.  
If $Q^{(-n)}(t)\le \fl{V^\rho t}$ and  $\Qlb(t)>-k$ then 
 $Y_{-k}(t)\le  \fl{V^\rho t}$. This  implies a bound on the second class particle current:
\be 
J^\zeta_{\fl{V^\rho t}}(t)-J^\eta_{\fl{V^\rho t}}(t)
=J^{\zeta-\eta}_{\fl{V^\rho t}}(t)\le k.
\label{eq:aux37}\ee

 Put $k=\fl{a_2t^{1/3}\bias^{1/3}}-2$.  Then by   $t\ge \bias^{-4}$ and the definition of $a_2$,
 \be \Pv\{\Qlb(t)\le -k\}\le e^{-2}<1/4.\label{eq:aux38}\ee
   
 Combine   \eqref{eq:aux36} and \eqref{eq:aux38} 
and split the probability:   
\begin{align}
\tfrac14&\le 
\Pv\{ J^\zeta_{\fl{V^\rho t}}(t)-J^\eta_{\fl{V^\rho t}}(t) \le
a_2t^{1/3}\bias^{1/3}-2\} \nn\\
&\le 
\Pv\{ J^\zeta_{\fl{V^\rho t}}(t) \le 2a_2t^{1/3}\bias^{1/3} 
+t{\bias}(2\rho\lambda-\lambda^2)  \} 
\nn\\
&\qquad + 
\Pv\{ J^\eta_{\fl{V^\rho t}}(t) \ge  a_2t^{1/3}\bias^{1/3} 
+t{\bias}(2\rho\lambda-\lambda^2)+2\}. 
\label{line19}
\end{align}

Consider next line \eqref{line19}.   The 
$\eta$-process can be coupled with a stationary $P^\lambda$-process
with at most one discrepancy. The mean current in
the stationary process is  
\begin{align*}
E^\lambda[ J_{\fl{V^\rho t}}(t)]
&= tH(\lambda)-\lambda\fl{V^\rho t} \\
&\le tH(\lambda)-\lambda{V^\rho t}+1
=t{\bias}(2\rho\lambda-\lambda^2)+1. 
\end{align*}
Hence 
\begin{align}
&\text{line \eqref{line19}} 
\le 
P^\lambda\{ J_{\fl{V^\rho t}}(t) \ge  a_2t^{1/3}\bias^{1/3} 
+t{\bias}(2\rho\lambda-\lambda^2)+1\}\nn\\[3pt]
&\le P^\lambda\bigl\{ \overline J_{\fl{V^\rho t}}(t) \ge  a_2t^{1/3}\bias^{1/3} \bigr\}
\le a_2^{-2}t^{-2/3}\bias^{-2/3}
\Var^\lambda\bigl[ J_{\fl{V^\rho t}}(t)\bigr]\nn\\[3pt]
&\le \frac{\Ev^\lambda\lvert Q(t)-\fl{V^\rho t}\rvert}{a_2^2t^{2/3}\bias^{2/3}}
\; \le \; \frac{\Ev^\lambda\lvert Q(t)-V^\lambda t\rvert}{a_2^2t^{2/3}\bias^{2/3}}
+ \frac{2 b }{a_2^2 }
+ \frac{ 1}{a_2^2t^{2/3}\bias^{2/3}}\nn\\
&\le C_{UB}a_2^{-2} + \tfrac1{16} + \tfrac1{64}\le \tfrac18.
\label{line22}
\end{align} 
After Chebyshev above  we applied the basic identity 
\eqref{VarJQ}  for which we introduced a second class particle 
$Q(t)$ in a density-$\lambda$ system under the measure
 $\Pv^\lambda$. 
Then 
we replaced $\fl{V^\rho t}$ with $V^\lambda t$ and    applied  the upper
bound    and 
properties of $a_2$.  

Put this last bound back into line \eqref{line19} to be left with  
\be
\tfrac18\le 
\Pv\{ J^\zeta_{\fl{V^\rho t}}(t) \le 2a_2t^{1/3}\bias^{1/3} 
+t{\bias}(2\rho\lambda-\lambda^2) \}. 
\label{line24}\ee

Next we replace the $\zeta$-process with
a stationary density-$\rho$ process by inserting the
Radon-Nikodym factor.  Let $\gamma$
denote the distribution of the initial $\zeta(0)$ 
configuration described by (c1)--(c3) in the beginning
of this section. As before $\nu^\rho$ is the
density-$\rho$ i.i.d.~Bernoulli measure.  The Radon-Nikodym
derivative is 
\[
f(\omega)=\frac{d\gamma}{d\nu^\rho}(\omega)= 
\frac1{\rho}\ind\{\omega_{-n}=1\}\cdot
\prod_{i=-n+1}^0\Bigl( \frac{\lambda}{\rho}\ind\{\omega_i=1\}
+ \frac{1-\lambda}{1-\rho}\ind\{\omega_i=0\}\Bigr).
\]
Bound its second moment: 
\be\begin{aligned}
E^\rho(f^2)&=\frac1\rho\Bigl(1+\frac{(\rho-\lambda)^2}{\rho(1-\rho)}
\Bigr)^n 
\le \rho^{-1}e^{n(\rho-\lambda)^2/\rho(1-\rho)}
\le  c_2(\rho) 
\end{aligned} \label{def:c1}\ee
where condition $t\ge c_0\bias^{-4}$ implies a bound  $c_2(\rho)<\infty$  
 independent of $t$ and $\bias$. 
 
Let $\mathcal A$ denote the exclusion process event
\[
\mathcal A=\{ J_{\fl{V^\rho t}}(t) \le 2a_2t^{1/3}\bias^{1/3} 
+t{\bias}(2\rho\lambda-\lambda^2)  \}.
\]
  Then from \eqref{line24}
\begin{align}
\tfrac18&\le \Pv\{\zeta\in \mathcal A\}
= \int P^\omega(\mathcal A)\,\gamma(d\omega) 
= \int P^\omega(\mathcal A)f(\omega)\,\nu^\rho(d\omega)\nn\\
&\le \bigl(P^\rho(\mathcal A)\bigr)^{1/2}\bigl(E^\rho(f^2)\bigr)^{1/2}
\le c_2(\rho)^{1/2} \bigl(P^\rho(\mathcal A)\bigr)^{1/2}.
\label{line26} 
\end{align}

 Note the stationary mean  
\begin{align*}
E^\rho\bigl[J_{\fl{V^\rho t}}(t)\bigr]
= t\flux(\rho)-\rho\fl{V^\rho t}
=t{\bias}\rho^2+\rho V^\rho t-\rho\fl{V^\rho t}
\ge t{\bias}\rho^2.
\end{align*}
 Continue from line \eqref{line26}, recalling    
\eqref{ba_2}:   
\begin{align*}
(64c_2(\rho))^{-1}&\le P^\rho(\mathcal A)= P^\rho\{ J_{\fl{V^\rho t}}(t) \le 2a_2t^{1/3}\bias^{1/3} 
+t{\bias}(2\rho\lambda-\lambda^2) \}\\
&\le P^\rho\{ \,\overline J_{\fl{V^\rho t}}(t) \le 2a_2t^{1/3}\bias^{1/3} 
-t{\bias}(\rho-\lambda)^2 \}\\
&= P^\rho\{ \,\overline J_{\fl{V^\rho t}}(t) 
\le -(  b^2-2a_2)t^{1/3}\bias^{1/3} \}  \\ 
&\le P^\rho\{ \, \overline J_{\fl{V^\rho t}}(t) 
\le -t^{1/3}\bias^{1/3} \}\\
&\le  t^{-2/3}\bias^{-2/3}
\Var^\rho \bigl[J_{\fl{V^\rho t}}(t)\bigr] \; \le \; t^{-2/3} \bias^{-2/3}  \Ev^\rho\abs{Q(t)-V^\rho t}. 
\end{align*}

 This completes the proof of  the lower bound.    We finish with some observations 
 about the need for the two key assumptions, asymmetry and $\flux''(\rho)\ne 0$. 
 
For symmetric 
 SEP $\bias=0$ and consequently the Chebyshev step above cannot be taken. 
  
  To observe where   $\flux''(\rho)<0$ came in we need to backtrack a little.   At stage
\eqref{line24}  we have the inequality (ignoring now small errors due to integer parts etc.)
\[  
\tfrac18\le 
\Pv\bigl\{ J^\zeta_{\fl{V^\rho t}}(t) \le 2a_2t^{1/3}\bias^{1/3} 
+ E^\lambda(J_{\fl{V^\rho t}}(t) )  \bigr\}.  \]
The Radon-Nikodym and Schwarz trick turned this into an inequality for
a stationary process:
\be \begin{aligned} 
0<c&\le 
P^\rho\bigl\{ J_{\fl{V^\rho t}}(t) \le 2a_2t^{1/3}\bias^{1/3} 
+ E^\lambda(J_{\fl{V^\rho t}}(t) )  \bigr\} \\[3pt]
&=
P^\rho\bigl\{ \overline J_{\fl{V^\rho t}}(t) \le 2a_2t^{1/3}\bias^{1/3} 
+ E^\lambda(J_{\fl{V^\rho t}}(t) )  - E^\rho(J_{\fl{V^\rho t}}(t) )  \bigr\}.
\end{aligned}\label{line30}\ee
Compute the means on the right-hand side inside the probability, remembering
that $V^\rho=\flux'(\rho)$ and  Taylor expanding
$\flux(\lambda)$:
\begin{align*} 
E^\lambda(J_{\fl{V^\rho t}}(t) )  - E^\rho(J_{\fl{V^\rho t}}(t) ) 
&=t\bigl[\flux(\lambda)-\lambda \flux'(\rho) -\flux(\rho)+\rho \flux'(\rho)\bigr]\\[4pt]
& = t\bigl[  \tfrac12 \flux''(\rho)(\lambda-\rho)^2  + O\bigl( \bias\abs{\lambda-\rho}^3\bigr) \bigr]\\[3pt]
&= -a_3b^2 t^{1/3}\bias^{1/3} + O(1) 
\end{align*}
 with $\frac12\flux''(\rho)=-a_3\bias<0$ in the last step.  Thus the 
 constants can be adjusted so that the probability in \eqref{line30} is a 
 deviation. Chebyshev can be applied to conclude that the current variance
 is of order $t^{2/3}\bias^{2/3}$.    But if $\flux''(\rho)=0$ there is no
 deviation to take advantage of.

\subsection*{Further comments and references}
 The proofs of this chapter are based on  \cite{bala-sepp-08ALEA}.  This
 article gives simpler proofs for the results in  
   \cite{bala-sepp-07JSP} and  \cite{bala-sepp-aom}.
 Precursors of these variance bounds  were first proved for last-passage models that 
 correspond to totally asymmetric versions of particle systems:
 in \cite{cato-groe-06}  for the Hammersley process and in \cite{bala-cato-sepp}
 for the corner growth model associated with TASEP. 

The Tracy-Widom type limit distribution for TASEP current was first proved for the step
initial condition in \cite{joha},  then for the stationary case (Theorem \ref{taseptwthm})
 in \cite{ferr-spoh-06}. 
The larger picture of TASEP fluctuations from various initial conditions is presented
in \cite{bena-corw}.  

Another line of work has produced comparison theorems that allow one to conclude
that Laplace transforms of $t^{-1}\Vv^\rho[Q(t)]$ for different asymmetric exclusion
processes are within constant multiples of each other.  In this sense,  for 
 the order of this Laplace transform there is universality for all finite range
asymmetric exclusion processes.  These results come from resolvent techniques
\cite{seth-03, quas-valk-cmp07, quas-valk-08}. 

The central limit theorem for the current in directions other than the characteristic 
$V^\rho$  was proved first by Ferrari and Fontes \cite{ferr-font-94a}.    This was generalized to other
particle systems such as certain zero-range and  bricklayer  processes
by Bal\'azs \cite{bala-03}. 

Consider symmetric simple exclusion, namely the case $p=q=1/2$.   
Then $V^\rho=0$.  
Equation \eqref{VarJQ} together with the 
 observation that   the second class particle is a simple symmetric random
walk tell us that $\Var^\rho[J_0(t)]$ is of order $t^{1/2}$, exactly as for independent 
particles in Section \ref{iidch}.  And indeed the current process does converge
to fractional Brownian motion (see \cite{peli-seth-08} and its references).

\section{Zero range process} 
\setcounter{equation}{0}
\subsection{Model and results} 

From the perspective of universality it would be highly desirable to extend
the results of Section \ref{asepch}  beyond exclusion processes.   Throughout the 40-year history
of the subject of interacting particle systems, the {\sl zero range process}  has been
a much-studied relative of the   exclusion process.   In this section we indicate
how the bounds for second class particles and current variance are proved for 
a class of totally asymmetric zero range processes (TAZRP) with concave jump rate functions. 
 
\medskip

{\bf Definition and graphical construction.}
  In contrast with the exclusion process,
the zero range process does not restrict the number of particles allowed at a site.  
The state of the process at time $t$ is $\eta(t)=(\eta_i(t))_{i\in\bZ}\in\bZ_+^\bZ$ where
$\eta_i(t)\in\bZ_+$ denotes the number of particles present at site $i$ at time $t$.
 We consider the case where particles take  only nearest-neighbor jumps to the right. 

Each zero range process is characterized by a jump rate function 
$\jr:\bZ_+\to\bR_+$.  It automatically has  the value $\jr(0)=0$.  The meaning of $\jr$ for the
process is that when the current state is $\eta$,  $\jr(\eta_i)$ is the rate at which
one  particle is moved from site $i$ to $i+1$.    You can interpret this as saying
that each of the particles at site $i$ jumps independently with rate $\eta_i^{-1}\jr(\eta_i)$
or that some particular one moves next (say, the bottom one is moved to the top of the
next pile) at rate $\jr(\eta_i)$.
It is immaterial for we do not label our particles. (Except again we will label
certain second class particles as we did for ASEP, but we will come to that later.) 
These jump events take place independently at all sites, exactly as for ASEP. 

We shall assume that 
\be \text{ $\jr$ is nondecreasing and  $0<\jr(k)\le 1$ for $k>0$. }\label{jrass}\ee
With a bounded $\jr$  we can perform the following concrete 
construction of the process $\eta(t)$  in terms of independent rate $1$ Poisson processes $\{N_i\}$ 
and i.i.d.\ Uniform$(0,1)$  variables $\{U_{i,k}\}$.  Attach clock $N_i$ to
site $i$, and give each jump time of $N_i$ its own   $U_{i,k}$. Now 
if $t$ is a  jump time for $N_i$ with its uniform $U_{i,k}$, then 
move one particle from $i$ to $i+1$  if $U_{i,k}<\jr(\eta_i(t-))$, otherwise not.  
Repeat this step at all sites and all jump times.  
The result is the desired one: independently at each 
site $i$,  a jump from $i$ to $i+1$ occurs in the next small  time interval $(t,t+dt)$
with probability $\jr(\eta_i(t))dt+O(dt^2)$. 
 
  The generator of this TAZRP is 
\[ L\varphi(\eta) =\sum_{i\in\bZ} \jr(\eta_i)[\varphi(\eta^{i,i+1})-\varphi(\eta)] \]
that acts on   bounded cylinder functions $\varphi$ on $\bZ_+^\bZ$
and  $\eta^{i,i+1}=\eta-\delta_i+\delta_{i+1}$.  We will not use the 
generator in the text.   It can be  used to check the invariance
of certain distributions on the state space, as for ASEP in Remark
\ref{asepinvrem}.  

\medskip

{\bf Invariant distributions.} 
Part of the reason for the popularity of ZRP is that, just like ASEP, it has
  i.i.d.~invariant distributions. We denote  these by 
$\{\nu^\rho\}_{0\le \rho<\infty}$ indexed by  
 density $\rho=E^\rho(\eta_i)$. 
 
 Here is the definition of these measures. 
  Let $\theta$ denote a real parameter, and on $\bZ_+$ define a probability distribution  
  \[\lambda^\theta(k)=\frac1{Z_\theta}\cdot\frac{e^{\theta k}}{\jr(k)!}\,,  \]
defined  for    $\theta$ such that  
\[ Z_\theta=\sum_k \frac{e^{\theta k}}{\jr(k)!} <\infty .  \] 
Here $\jr(0)!=1$  and $\jr(k)!=\jr(1)\dotsm\jr(k)$ for $k>0$. 
Define the mean density function $ \rho(\theta)=\sum_k k \lambda^\theta(k) .$  It is 
smooth and strictly increasing on the open interval where $Z_\theta<\infty$.  
Let its inverse function be $\theta(\rho)$ and then 
reparametrize the distributions  in terms of density:  
\[  \nu_0^\rho(k)=\lambda^{\theta(\rho)}(k)
=\frac1{Z_{\theta(\rho)}}\cdot\frac{e^{\theta(\rho) k}}{\jr(k)!}\,.    \]  
 Finally, the actual invariant measures for ZRP are the 
product measures on the state space $\bZ_+^\bZ$:  
 \be\nu^\rho(d\eta)=\bigotimes_{i\in\bZ} \nu^\rho_0(d\eta_i). \label{zrpnurho}\ee
We write $P^\rho$ for probabilities and $E^\rho$ for expectations for the 
stationary process whose marginal $\eta(t)$ has distribution $\nu^\rho$.  

\medskip

{\bf Basic coupling and second class particles.}
  {\sl Basic coupling} works exactly as it did for   exclusion processes:  
 two (or more) zero range processes  obey a common set
of Poisson clocks  $\{N_i\}$ and uniform variables $\{U_{i,k}\}$.  
 
We write a boldface $\Pv$ for the probability measure
when more than one process are coupled together. 
In particular, $\Pv^\rho$ represents the
situation where the initial occupation 
variables  $\eta_i(0)=\eta^+_i(0)$
are i.i.d.~mean-$\rho$ Bernoulli for $i\ne 0$, and the
second class particle $Q$ starts at $Q(0)=0$.   

More generally, if two processes $\eta$ and $\om$ are in basic
coupling and $\om(0)\ge\eta(0)$ (by which we mean coordinatewise
ordering $\om_i(0)\ge\eta_i(0)$ for all $i$) then the ordering
$\om(t)\ge\eta(t)$ holds for all $0\le t<\infty$.  The effect of
the basic coupling is to give priority to the $\eta$ particles
over the $\om-\eta$ particles. Consequently we can think
of the $\om$-process as consisting of  first class particles 
(the $\eta$ particles) and second class particles 
(the $\om-\eta$ particles).

\medskip

{\bf Current.} 
The current is defined as for ASEP:  for $x\in\bZ$ and $t>0$,
$J_x(t)$ is  the net left-to-right
particle current across the straight-line space-time path 
from $(1/2,0)$ to $(x+1/2,t)$.

 The flux function is again 
\begin{align*}
\flux(\rho)= E^\rho[\text{rate of particle flow across a fixed edge}] 
=E^\rho[\jr(\eta_i)]. 
\end{align*}
Expectations of  currents can be computed as for ASEP: 
 \be
E^\rho[J_x(t)]=t\flux(\rho)-x\rho,  \qquad x\in\bZ, t\ge 0.  
\label{zrpEJ}\ee

The {\sl characteristic speed} at density $\rho$ is defined the same way as before:
\be
V^\rho=\flux'(\rho). 
\label{def:zrpVrho}\ee

As for ASEP, the first task is to establish the   identities that connect current
variance and the second class particle. 
 These identities  for ZRP develop the same way as for ASEP, except 
that a new initial distribution for the coupled process appears.   Define  a
probability distribution $\hat\nu_0^\rho$ on $\bZ_+$ by 
\[    \hat\nu^\rho_0(k) = \frac1{\Var^\rho(\eta_0)} \sum_{m=k+1}^\infty (m-\rho)\nu^\rho_0(m),
\qquad k\in\bZ_+. \]
Define  a product distribution $\hat\nu^\rho$ on the state space $\bZ_+^\bZ$
 that obeys the marginals $\nu_0^\rho$ of the stationary distribution at all sites
 except at the origin where the distribution is $\hat\nu_0^\rho$:
\[\hat\nu^\rho(d\eta)=\Bigl(\bigotimes_{i\ne 0} \nu_0^\rho(d\eta_i)\Bigr)
\otimes \hat\nu^\rho_0(d\eta_0).\] 
Let $\wh\Pv^\rho$ be the probability distribution of a   pair   $(\eta,\eta^+)$ 
that satisfies  $\eta^+(t)=\eta(t)+\delta_{Q(t)}$ (so there is one discrepancy), 
  obeys basic coupling, 
and whose initial distribution is such that $\eta(0)\sim\hat\nu^\rho$ and 
$\eta^+(0)=\eta(0)+\delta_{0}$  (in other words,  $Q(0)=0$).

\begin{theorem}  
For any density $0<  \rho<\infty$, 
$z\in\bZ$ and $t>0$ we have these formulas.    
\be
\Var^\rho[J_z(t)]=\sum_{j\in\bZ} \abs{j-z} \Cov^\rho[\eta_j(t),\eta_0(0)],  
 \label{zrpgoal1}\ee
\be
\Cov^\rho[\eta_j(t),\,\eta_0(0)]
= \Var^\rho(\eta_0) \wh\Pv^\rho\{Q(t)=j\}, 
\label{zrpCovQ}\ee
and 
\be
\wh\Ev^\rho[Q(t)]=V^\rho t.
\label{zrpEQH'a}\ee 
\label{zrpthm1} \end{theorem}  

Equation \eqref{zrpgoal1} is proved the same way as for ASEP.  
Equation \eqref{zrpCovQ} and the definition of $\hat\nu_0^\rho$ come from
a short calculation which we show below in Section \ref{zrpvarsec}.  
We omit the proof of \eqref{zrpEQH'a}.  
Formulas \eqref{zrpgoal1} and \eqref{zrpCovQ} combine to give the key
equation that links current variance with the second class particle: 
\be 
\Var^\rho[J_{z}(t)]
=\Var^\rho(\eta_0)\wh\Ev^\rho\lvert Q(t)-z\rvert.\label{zrpVarJQ}\ee

Next we state the main result which again consists of upper and lower  moment bounds
for a second class particle, this time under the measure $\wh\Pv^\rho$. 
We need a significant restriction on the concavity of the jump rate $\jr$:
\be
\text{$\exists$  $0<r<1$  such that  $\jr(k+1)-\jr(k)\le r(\jr(k)-\jr(k-1))$.}  
\label{zrpass}\ee
A class of examples satisfying this hypotheses is given by 
 $\jr(k)=1-\exp(-ak^b)$ with $a>0$, $b\ge 1$.    
$\jr(k)$ can also be constant
from some $k_0$ onwards.

\begin{theorem}  Fix a density $0<\rho<\infty$ 
and consider a pair of coupled ZRP's under the measure $\wh\Pv^\rho$.
Assume the jump rate function $\jr$ satisfies \eqref{jrass} and  \eqref{zrpass}. 
Then for $1\le m<3$,   large enough $t\in\bR_+$ and a constant $C$, 
 \be \frac1{C} t^{2m/3} \le 
\wh\Ev^\rho  \bigl[\,\abs{Q(t)-V^\rho t}^m \,\bigr]\le \frac{C}{3-m} t^{2m/3}.   
\label{zrpQmom}\ee
\label{zrpQmomthm}\end{theorem} 

  The constants $C$ and how large $t$ needs to be may depend on the density $\rho$. 
 Combining \eqref{zrpgoal1} and \eqref{zrpQmom} 
 gives the  bounds for the variance of the current 
seen by an observer traveling at the 
characteristic speed $V^\rho$: for large enough $t$, 
\[
C^{-1}t^{2/3} \le \Var^\rho[J_{\fl{V^\rho t}}(t)] 
\le C t^{2/3}.
\]

For the remainder of this chapter we discuss parts of the proof.  
Once we have the fundamental identities that tie together moments of the 
second class particle and the variance of the current, 
the proofs for the upper and lower bounds given for ASEP in Sections 
 \ref{asepubsec} and  \ref{aseplbsec} can be adjusted to work  for TAZRP.  
 We will not repeat those derivations.   Instead, we focus on the 
   key ingredients that made the argument  work for ASEP,  
and discuss how to  provide these ingredients  for TAZRP. 
There are two key points that we need in order to repeat the argument for TAZRP:  

\begin{enumerate}
\item\label{a1}  We need a construction that includes  a given second class particle 
as a labeled member of  a density of second class particles, and then we need
a tail bound for the  label as the one given for ASEP in Lemma \ref{lm:mQ} and \eqref{mQ2}.  
\item\label{a2}  We need a coupling that keeps the second class particle of a 
  system with higher density  behind the second class particle of a system with 
  lower density. For ASEP this was Theorem \ref{th:newcoupling}, which was 
  used to obtain \eqref{eq:aux9} for the ASEP upper bound.  
\end{enumerate} 
We turn to these points in Section \ref{zrpcouplingsec} below, after developing
the variance formula.

 \subsection{Variance identity}\label{zrpvarsec} 
 
 Define  $F(-1)=0$ and 
 \[  F(k)=\sum_{m=k+1}^\infty (m-\rho) \frac{\nu_0^\rho(m)}{\nu_0^\rho(k)}
 = \frac{\Var^\rho(\eta_0)}{\nu_0^\rho(k)}\hat\nu_0^\rho(k), \quad k\ge 0. \]
Interpret below  $\nu_0^\rho(-1)$ as $0$.  
\begin{align*}
&\Cov^\rho[\eta_i(t), \eta_0(0)]=E^\rho[\eta_i(t)(\eta_0(0)-\rho)] 
=\sum_{k\ge 0}  E[\eta_i(t)\vert\eta_0(0)=k] (k-\rho) \nu_0^\rho(k)\\
&\quad = \sum_{k\ge 0}  E[\eta_i(t)\vert\eta_0(0)=k] \bigl(  F(k-1) \nu_0^\rho(k-1) -
F(k) \nu_0^\rho(k)  \bigr) \\
&\quad = \sum_{k\ge 0}  E[\eta_i(t)\vert\eta_0(0)=k+1]   F(k) \nu_0^\rho(k) 
-  \sum_{k\ge 0}  E[\eta_i(t)\vert\eta_0(0)=k] F(k) \nu_0^\rho(k).  \\
\end{align*}
Construct a coupling of $\eta^+(t)=\eta(t)+\delta_{Q(t)}$ with the  discrepancy 
initially at the origin $Q(0)=0$, and so that $\eta(0)$ has $\nu^\rho$-distribution. 
Then, due to the product form of the initial distribution, 
\begin{align*}
\Ev[\eta_i(t)\vert\eta_0(0)=k+1] = \Ev[\eta^+_i(t)\vert\eta^+_0(0)=k+1] 
= \Ev[\eta^+_i(t)\vert\eta_0(0)=k].  \end{align*} 
Then continuing from above, 
\begin{align*}
\Cov^\rho[\eta_i(t), \eta_0(0)]
&= \sum_{k\ge 0}  \Ev[\eta_i^+(t)\vert\eta_0(0)=k]   F(k) \nu_0^\rho(k) \\
&\qquad \qquad 
-  \sum_{k\ge 0}  \Ev[\eta_i(t)\vert\eta_0(0)=k] F(k) \nu_0^\rho(k)  \\
&= \sum_{k\ge 0}  \Ev[\eta_i^+(t)-\eta_i(t) \vert\eta_0(0)=k]   F(k) \nu_0^\rho(k)  \\
&=\Ev\bigl[ (\eta_i^+(t)-\eta_i(t)) F(\eta_0(0)) \bigr] \\[2pt]
&=\Ev\bigl[ \ind\{Q(t)=i\} F(\eta_0(0)) \bigr] \\
&=\sum_{k\ge 0} \Ev[ \ind\{Q(t)=i\} \vert \eta_0(0)=k]F(k) \nu_0^\rho(k) \\
&=\Var^\rho(\eta_0) \sum_{k\ge 0} \Ev[ \ind\{Q(t)=i\} \vert \eta_0(0)=k]\hat\nu_0^\rho(k) \\
&=\Var^\rho(\eta_0)\wh\Pv^\rho\{Q(t)=i\}. 
\end{align*} 
This proves \eqref{zrpCovQ}. 

\subsection{Coupling for the zero range process} \label{zrpcouplingsec}

Next we describe a coupling of two processes with labeled discrepancies
(second class particles)  between them, and then two randomly evolving labels
that achieve simultaneously both goals \eqref{a1} and \eqref{a2} 
mentioned above.  
This construction works for any TAZRP with concave jump rate function $\jr$. 
Getting tail bounds for the label processes   
  is the 
serious bottleneck of this proof, and that is where we need the restrictive
assumption \eqref{zrpass}.  

 Let two processes $\eta\le\om$ evolve in basic coupling.  
 This pair  $(\eta,\om)$  together with the labeled and ordered  $\om-\eta$ 
 second class particles 
 \(\dotsm\le X_{-2}(t)\le X_{-1}(t)\le X_0(t)\le X_1(t)\le X_2(t)\le\dotsm\)
 form a ``background'' process on which we define two label processes
 $y(t)$ and $z(t)$.  
 The $\om-\eta$ second class  particles are kept in order by requiring that,
 whenever a second class particle jumps to the right, the $X$-particle with highest
 label is moved.  
 
   The label processes will satisfy $y(t)\le z(t)$, we will
 be able to bound $y(t)$ stochastically from above,   $z(t)$ stochastically
 from below,  and the following  two pairs of processes will individually be in basic
 coupling:
 \be
 (\om^-,\om)=(\om-\delta_{X_y},\om)\quad\text{and}\quad (\eta,\eta^+)=(\eta,\eta+\delta_{X_z}).
 \label{zrppairs}\ee
 
 The definition of the label processes is partly forced on us by the 
 requirement  that jumps of $X_y$ and $X_z$ must replicate the rates required
 by the basic coupling.  Additionally, we devise the joint process $(y,z)$ so that the
 order $y\le z$ is maintained.  
 
 The rule is that after each jump among $(\om,\eta)$ 
  that in any way affects the site where $X_y$ resides,
 the value of $y$ is refreshed randomly.    Let $a$ and $b$ denote the minimal and
 maximal labels at the site $i$ where $X_y$ resides {\sl after} the jump.  If $X_y$ 
 resides at a site other than $X_z$, then 
 $y$ chooses a new value $y'$ according to these probabilities:
 \be y'= 
 \begin{cases}  
 a &\text{with probability \ }\ddd\frac{\rr(\om_i-1)-\rr(\eta_i)}{\rr(\om_i)-\rr(\eta_i)}  \\[12pt]
b &\text{with probability \ } \ddd\frac{\rr(\om_i)-\rr(\om_i-1)}{\rr(\om_i)-\rr(\eta_i)}.
 \end{cases}
 \label{zrpyrule}\ee
 If $\rr(\om_i)-\rr(\eta_i)=0$ then $y'=a$.  
 
 Similarly, after a jump in  the background process that affects the site 
 where  $X_z$ resides,  if $X_z$ and $X_y$ are not together, then $z$ takes the
 new value $z'$ as follows (with   $b$ again the maximal label at the site 
  $i=X_z$ after the jump): 
  \[ z'= 
 \begin{cases}  
 b-1 &\text{with probability \ }\ddd\frac{\rr(\om_i)-\rr(\eta_i+1)}{\rr(\om_i)-\rr(\eta_i)}  \\[12pt]
b &\text{with probability \ } \ddd\frac{\rr(\eta_i+1)-\rr(\eta_i)}{\rr(\om_i)-\rr(\eta_i)}. 
 \end{cases}
 \]
 If $\rr(\om_i)-\rr(\eta_i)=0$ then $z'=b$.  

Finally, if after the jump $X_y=X_z=i$, then the labels are refreshed jointly 
as follows:
  \[ \binom{y'}{z'}= 
 \begin{cases}  
\ddd \binom{a}{b-1} &\text{with probability \ }\ddd\frac{\rr(\om_i)-\rr(\eta_i+1)}{\rr(\om_i)-\rr(\eta_i)}  
\\[15pt]
\ddd\binom{a}{b} &\text{with probability \ } 
\ddd\frac{\rr(\eta_i+1)-\rr(\eta_i)}{\rr(\om_i)-\rr(\eta_i)}  \\[10pt]
 & \qquad\qquad \;-\;
\ddd\frac{\rr(\om_i)-\rr(\om_i-1)}{\rr(\om_i)-\rr(\eta_i)}  \\[15pt]
\ddd\binom{b}{b} &\text{with probability \ } \ddd\frac{\rr(\om_i)-\rr(\om_i-1)}{\rr(\om_i)-\rr(\eta_i)}. 
 \end{cases}
 \]
 If $\rr(\om_i)-\rr(\eta_i)=0$ then $(y',z')=(a,b)$.  
 
 The jump rules preserve $y\le z$.  Note that marginally $y'$ obeys probabilities
 \eqref{zrpyrule}, and similarly for $z'$.  Concavity of $\jr$ was used  to define the
 middle case in the joint rule.  
 
 Let us observe why these rules give the pair $ (\om^-,\om)=(\om-\delta_{X_y},\om)$
 the same rates this pair would have in basic coupling.  The requirement is that 
 a jump across edge $(i,i+1)$ occur for both processes at rate $\jr(\om^-_i)$,
 and only for $\om$ at rate $\jr(\om_i)-\jr(\om^-_i)$.   This requires thinking through
 a few cases. Only the site  where $X_y$ resides needs attention since elsewhere
$ (\om^-,\om)$ jump together according to ZRP rates.   
 \begin{enumerate}
 \item[(i)]    In the basic coupling of $(\eta,\om)$,  an $(i,i+1)$ jump occurs
 in $\eta$ at rate $\jr(\eta_i)$. Then both $\om$ and $\om^-$ experience this jump. 
 \item[(ii)]  An $\om-\eta$ second class particle jumps at rate $\jr(\om_i)-\jr(\eta_i)$.
   Prior to this jump $y$ chose the top label with probability 
   given by the second line of \eqref{zrpyrule}, hence the rate at which $X_y$ jumps
   is 
   \[  \bigl( \jr(\om_i)-\jr(\eta_i)\bigr)\cdot 
    \frac{\rr(\om_i)-\rr(\om_i-1)}{\rr(\om_i)-\rr(\eta_i)} = \rr(\om_i)-\rr(\om_i-1). \]
    Thus at this rate $\om$ experiences the jump but $\om^-$ does not.
    
    If prior to this jump  $y$ chose the bottom label, then both $\om$ and $\om^-$ 
    experience this jump, and this happens with rate 
   \[  \bigl( \jr(\om_i)-\jr(\eta_i)\bigr)\cdot 
    \frac{\rr(\om_i-1)-\rr(\eta_i)}{\rr(\om_i)-\rr(\eta_i)} = \rr(\om_i-1)-\rr(\eta_i). \]     
   \end{enumerate} 
  Adding up the rates we see that the rates of basic coupling have been realized.  
A similar argument is given for $(\eta,\eta^+)=(\eta,\eta+\delta_{X_z}).$

\medskip

We come to the unique point in the  proof where assumption \eqref{zrpass}
is used, namely the tail bounds for the labels.

\begin{lemma}  Let the labels start with $y(0)=z(0)=0$.  
Under assumption \eqref{zrpass} we have these bounds:
$\Pv\{y(t)\ge k\}\le r^k$  and $\Pv\{z(t)\le -k\}\le r^k$  for all $k\in\bZ_+$ and $t\ge 0$.
\end{lemma} 

\begin{proof}  We do the proof for $y(t)$. The bounds are valid conditionally
on the evolution $(\eta,\om)$ of the background process.  So assume this background 
evolution given.  Then we think of $y(t)$   as an integer-valued  Markov chain 
that is subject to jumps  triggered by the background
environment.   Each 
   jump happens 
  at some   site $i$ with range of labels  $\{a,\dotsc,b\}$  and  occupation variables   
 $\om_i>\eta_i\ge 0$ that together satisfy  \[\om_i-\eta_i=b-a+1.\]  
Given the current value $y$,  the  new   value $y'$   is obtained by the following rules, which
of course are consistent with \eqref{zrpyrule}:  
 if $\rr(\om_i)-\rr(\eta_i)=0$ then 
  \be y'= 
 \begin{cases}  
 y &\text{if $y<a$ or $y>b$} \\
 a &\text{ if $y\in \{a,\dotsc,b\}$}  
 \end{cases}
 \label{zrpyrule1}\ee 
while  if $\rr(\om_i)-\rr(\eta_i)>0$ then  
  \be y'= 
 \begin{cases}  
 y &\text{if $y<a$ or $y>b$} \\
 a &\text{with probability \ }\ddd\frac{\rr(\om_i-1)-\rr(\eta_i)}{\rr(\om_i)-\rr(\eta_i)} 
 \ \text{ if $y\in \{a,\dotsc,b\}$}  \\[12pt]
b &\text{with probability \ } \ddd\frac{\rr(\om_i)-\rr(\om_i-1)}{\rr(\om_i)-\rr(\eta_i)}
 \ \text{ if $y\in \{a,\dotsc,b\}$}.
 \end{cases}
 \label{zrpyrule2}\ee
 There are  infinitely many such jumps in any finite time interval, but all but finitely
 many leave $y$ unchanged. 
 To get around this difficulty, we can first freeze all the Poisson clocks outside
 space interval $[-M,M]$, prove the lemma there, and then let $M\nearrow\infty$. 
 Since rates are bounded, in any given bounded block of space-time  the 
 finite-$M$ process  agrees with the infinite process 
  for all large enough $M$. 

To prove the lemma we show that every  jump of type  \eqref{zrpyrule1}--\eqref{zrpyrule2}
 preserves the geometric tail bound,    
 regardless of the values
  $a, b, \om_i, \eta_i$.  
  So suppose $y$ is an integer-valued  random variable 
 such that 
  \[  P(y\ge k)\le r^k\quad\text{ for}\quad  k\ge 0, \]   
  and define $y'$ via  \eqref{zrpyrule1}--\eqref{zrpyrule2}.
  We wish to show that  $P(y'\ge k)\le r^k$ for $k\ge 0$. 
  
 The case \eqref{zrpyrule1} is clear since there $y'\le y$. Let us  consider the case 
$\rr(\om_i)-\rr(\eta_i)>0$.  
   Since the jump only  redistributes the probability
  mass in $\{a,\dotsc,b\}$ to  $\{a,b\}$, 
 it suffices to check that 
 \be   P(y'\ge b)\le r^b  \label{zrpcheck2}\ee 
 in the case $b\ge 0$. 
 Using the jump rule \eqref{zrpyrule2},
\begin{align*}
P(y'\ge b)&=P(y'=b)+P(y'\ge b+1) \\
&= \frac{\rr(\om_i)-\rr(\om_i-1)}{\rr(\om_i)-\rr(\eta_i)}  P(a\le y\le b)  +P(y \ge b+1). 
\end{align*}
If $\rr(\om_i)-\rr(\om_i-1)=0$ the conclusion \eqref{zrpcheck2} follows from the
assumption on $y$.  So we assume $\rr(\om_i)-\rr(\om_i-1)>0$.  
Then by concavity all the $\jr$-increments 
 between $\eta_i,\dotsc,\om_i$ are positive.  Next write 
\be \begin{aligned}
&P(y'\ge b) =\frac{\rr(\om_i)-\rr(\om_i-1)}{\rr(\om_i)-\rr(\eta_i)}\sum_{k=a}^b (1-r)r^k \\
&\qquad +  \frac{\rr(\om_i)-\rr(\om_i-1)}{\rr(\om_i)-\rr(\eta_i)} \bigl( P(a\le y\le b) -r^a+r^{b+1}\bigr)\\
&\qquad\qquad \qquad \qquad  +P(y \ge b+1).
\end{aligned} \label{zrptemp1}\ee
For $a\le k\le b$
 \begin{align*}
  (1-r)r^k&= (1-r)r^b\cdot\frac1{r^{b-k}} 
\le  (1-r)r^b \prod_{\ell=\om_i-b+k}^{\om_i-1}  \frac{\jr(\ell)-\jr(\ell-1)}{\jr(\ell+1)-\jr(\ell)}\\
&=  (1-r)r^b  \cdot\frac{\jr(\om_i-b+k)-\jr(\om_i-b+k-1)}{\jr(\om_i)-\jr(\om_i-1)}. 
 \end{align*}
Adding these up over $a\le k\le b$ gives  
\begin{align*}
&\text{first term on the right in  \eqref{zrptemp1}}\\
&\quad \le \frac{\rr(\om_i)-\rr(\om_i-1)}{\rr(\om_i)-\rr(\eta_i)}  (1-r)r^b \\
&\qquad \qquad\qquad\times \sum_{k=a}^b 
\frac{\jr(\om_i-b+k)-\jr(\om_i-b+k-1)}{\jr(\om_i)-\jr(\om_i-1)}\\
&\quad =   (1-r)r^b \cdot \frac{\jr(\om_i)-\jr(\om_i-b+a-1)}{\rr(\om_i)-\rr(\eta_i)} 
=  (1-r)r^b=r^b-r^{b+1} .
\end{align*}
 Substitute this 
bound back up to   \eqref{zrptemp1} and use $P(y \ge k) \le r^{k}$ twice:   
\begin{align*}
P(y'\ge b) &\le  r^b +   P(y \ge b+1) -r^{b+1}  \\
&\qquad 
  +  \frac{\rr(\om_i)-\rr(\om_i-1)}{\rr(\om_i)-\rr(\eta_i)} \bigl( P(a\le y\le b) -r^a+r^{b+1}\bigr)\\
 &\le  r^b  
  +  \frac{\rr(\om_i)-\rr(\om_i-1)}{\rr(\om_i)-\rr(\eta_i)} \bigl( P(y\ge a) -r^a \bigr)\\ 
  &\le  r^b.  
\end{align*}
Thus  \eqref{zrpcheck2}  has been checked and thereby the lemma has been
proved for  $y(t)$.  
  \end{proof}

\subsection*{References}
 The results  for ZRP are proved in  \cite{bala-komj-sepp}.  
 The case of constant jump rate
was done earlier in  \cite{bala-komj-08}.

\providecommand{\bysame}{\leavevmode\hbox to3em{\hrulefill}\thinspace}
\providecommand{\MR}{\relax\ifhmode\unskip\space\fi MR }
\providecommand{\MRhref}[2]{%
  \href{http://www.ams.org/mathscinet-getitem?mr=#1}{#2}
}
\providecommand{\href}[2]{#2}

\vskip .3in

\noindent {{Timo Sepp\"al\"ainen}}

\noindent {{Mathematics Department, University of Wisconsin-Madison}}


\noindent {{Madison, Wisconsin 53706, USA}}

\noindent {{seppalai@math.wisc.edu}}

\noindent {http://{www.math.wisc.edu/$\sim$seppalai}}
  
\end{document}